\documentclass[a4paper,reqno,11pt]{amsart}

\usepackage[american]{babel}

\usepackage[a4paper, scale={0.72,0.74}, marginratio={1:1}, footskip=7mm, headsep=10mm]{geometry}

\usepackage{indentfirst}

\usepackage{emptypage}  

\usepackage{perpage}

\makeatletter
\renewcommand{\@secnumfont}{\bfseries}
\renewcommand\section{\@startsection{section}{1}%
\z@{.7\linespacing\@plus\linespacing}{.5\linespacing}%
{\large\bfseries\centering}}
\renewcommand\subsection{\@startsection{subsection}{2}%
  \z@{.5\linespacing\@plus.7\linespacing}{-.5em}%
  {\bfseries}}
\renewcommand\subsubsection{\@startsection{subsubsection}{3}%
  \z@{.5\linespacing\@plus.7\linespacing}{-.5em}%
  {\scshape}}
\makeatother

\usepackage{enumitem}

\usepackage{fixltx2e}

\usepackage{color}
\newcommand\blue{\textcolor{blue}}
\newcommand\red{\textcolor{red}}

\definecolor{eb}{rgb}{0.2,0,0.8}

\usepackage{multicol}

\usepackage{hyperref}

\usepackage{booktabs}
\usepackage{graphicx}
\usepackage{tabularx}
\usepackage{caption,subcaption}

\usepackage{longtable}
\usepackage{tikz,tikz-cd}
\usetikzlibrary{matrix, calc, arrows, decorations.markings, positioning}

\usepackage{amsmath, amssymb, amsthm}
\usepackage{stmaryrd, mathtools, mathdots}
\usepackage{eufrak}
\usepackage{wasysym}
\usepackage{mathrsfs}

\usepackage{accents}

\usepackage{bm,bbm}

\usepackage{braket}

\numberwithin{equation}{section}

\usepackage{subdepth}

\usepackage[style=alphabetic, backend=bibtex8, citestyle=alphabetic, firstinits=true, maxnames=99, maxalphanames=5, sorting=nyt]{biblatex}
\DeclareFieldFormat[article, inbook, incollection, inproceedings, patent, thesis, unpublished]{title}{#1} 
\renewbibmacro{in:}{}
\appto{\bibsetup}{\sloppy}

\newenvironment{myitemize}{%
\begin{list}{$\bullet$}%
 	{%
	\setlength{\itemsep}{0.4em}%
	\setlength{\topsep}{0.5em}%
	\setlength\leftmargin{2.45em}%
	\setlength\labelwidth{2.05em}%
	\setlength{\labelsep}{0.4em}%
	}%
	}%
{\end{list}}
\renewenvironment{itemize}{
\begin{myitemize}}%
{\end{myitemize}}

\MakePerPage{footnote}
\makeatletter
\newcommand*{\myfnsymbolsingle}[1]{%
  \ensuremath{%
    \ifcase#1
    \or 
      \dagger
    \else 
      \@ctrerr  
    \fi
  }%
}   
\makeatother

\usepackage{alphalph}
\newalphalph{\myfnsymbolmult}[mult]{\myfnsymbolsingle}{}

\theoremstyle{plain}
\newtheorem{theorem}{Theorem}[section]
\newtheorem{lemma}[theorem]{Lemma}
\newtheorem{proposition}[theorem]{Proposition}
\newtheorem{corollary}[theorem]{Corollary}

\theoremstyle{definition}
\newtheorem{definition}[theorem]{Definition}
\newtheorem{remark}[theorem]{Remark}

\makeatletter
\DeclarePairedDelimiter\abs{\lvert}{\rvert} 
\let\oldabs\abs
\def\abs{\@ifstar{\oldabs}{\oldabs*}}
\DeclarePairedDelimiterX{\norm}[1]{\lVert}{\rVert}{#1} 
\let\oldnorm\norm
\def\norm{\@ifstar{\oldnorm}{\oldnorm*}}
\DeclarePairedDelimiterX{\ceil}[1]{\lceil}{\rceil}{#1} 
\let\oldceil\ceil
\def\ceil{\@ifstar{\oldceil}{\oldceil*}}
\DeclarePairedDelimiterX{\floor}[1]{\lfloor}{\rfloor}{#1} 
\let\oldfloor\floor
\def\floor{\@ifstar{\oldfloor}{\oldfloor*}}
\makeatother

\makeatletter
\DeclareRobustCommand{\cev}[1]{%
  {\mathpalette\do@cev{#1}}%
}
\newcommand{\do@cev}[2]{%
  \vbox{\offinterlineskip
    \sbox\z@{$\m@th#1 x$}%
    \ialign{##\cr
      \hidewidth\reflectbox{$\m@th#1\vec{}\mkern4mu$}\hidewidth\cr
      \noalign{\kern-\ht\z@}
      $\m@th#1#2$\cr
    }%
  }%
}
\makeatother

\renewcommand{\P}{\mathbb{P}} 
\newcommand{\E}{\mathbb{E}} 
\newcommand{\1}{\mathbbm{1}} 
\newcommand{\diff}{\mathop{}\!\mathrm{d}}
\DeclareMathOperator{\e}{\mathrm{e}} 
\DeclareMathOperator{\tr}{tr}
\newcommand{\maxeig}{\lambda_{\max}}
\newcommand{\mineig}{\lambda_{\min}}
\newcommand{\meas}[1]{\mathfrak{m} {#1}}
\newcommand{\bdd}[1]{\mathfrak{b} {#1}}

\renewcommand{\emptyset}{\varnothing}
\newcommand{\Z}{\mathbb{Z}} 
\newcommand{\R}{\mathbb{R}} 
\newcommand{\C}{\mathbb{C}} 
\newcommand{\pos}{\mathcal{P}}
\newcommand{\spos}{\overline{\mathcal{P}}}
\newcommand{\GL}{\mathrm{GL}}
\newcommand{\ort}{\mathcal{O}}
\newcommand{\trian}[2]{\mathcal{T}^{#1}_{#2}}
\newcommand{\Sym}{\mathrm{Sym}}
\newcommand{\Diag}{\mathrm{Diag}}
\newcommand{\Scal}{\mathrm{Scal}}

\renewcommand{\epsilon}{\varepsilon}
\renewcommand{\rho}{\varrho}
\renewcommand{\phi}{\varphi}

\DeclareMathSymbol{\widehatsym}{\mathord}{largesymbols}{"62}

\renewcommand{\hat}{\widehat}

\renewcommand{\tilde}{\widetilde}

\newcommand{\widesim}[2][1.5]{
  \mathrel{\overset{#2}{\scalebox{#1}[1]{$\sim$}}}
}

\makeatletter
\newcommand{\doubletilde}[1]{{%
  \mathpalette\double@tilde{#1}%
}}
\newcommand{\double@tilde}[2]{%
  \sbox\z@{$\m@th#1\tilde{#2}$}%
  \ht\z@=.9\ht\z@
  \tilde{\box\z@}%
}
\makeatother

\newenvironment{myenumerate}{%
\renewcommand{\theenumi}{(\roman{enumi})}%
\renewcommand{\labelenumi}{\theenumi}%
\begin{list}{\labelenumi}
	{%
	\setlength{\itemsep}{0.4em}%
	\setlength{\topsep}{0.5em}%
	\setlength\leftmargin{2.45em}%
	\setlength\labelwidth{2.05em}%
	\setlength{\labelsep}{0.4em}%
	\usecounter{enumi}%
	}%
	}%
{\end{list}
}
\renewenvironment{enumerate}{
\begin{myenumerate}}%
{\end{myenumerate}}

  {\left\lbrace\begin{array}{@{}l@{}}}%
  {\end{array}\right.}

\makeatletter
\newsavebox{\mybox}\newsavebox{\mysim}
\newcommand{\distras}[1]{%
  \savebox{\mybox}{\hbox{\kern3pt$\scriptstyle#1$\kern3pt}}%
  \savebox{\mysim}{\hbox{$\sim$}}%
  \mathbin{\overset{#1}{\kern\z@\resizebox{\wd\mybox}{\ht\mysim}{$\sim$}}}%
}
\makeatother

\author[J.~Arista]{Jonas Arista}
\address{Universit\"at Bielefeld\\
Fakult\"at f\"ur Mathematik \\
Universit\"atsstra{\ss}e 25\\
33615 Bielefeld, Germany}
\email{jarista@math.uni-bielefeld.de}

\author[E.~Bisi]{Elia Bisi}
\address{Technische Universit\"at Wien \\
Institut f\"ur Stochastik und Wirtschaftsmathematik \\
E 105-07 \\
Wiedner Hauptstra{\ss}e 8-10 \\
1040 Wien, Austria}
\email{elia.bisi@tuwien.ac.at}

\author[N.~O'Connell]{Neil O'Connell}
\address{School of Mathematics and Statistics\\
University College Dublin\\
Dublin 4, Ireland}
\email{neil.oconnell@ucd.ie}

\thanks{Research supported by the European Research Council (grant 669306).}

\keywords{Whittaker functions of matrix arguments; Intertwining relations; Interacting Markov dynamics; Noncommutative polymer models; Constrained minimisation; Directed graphs}

\subjclass[2010]{Primary:
60K35, 
82B23, 
60B20. 
Secondary:
33C15, 
05E05, 
22E30. 
}

\begin{document}

\title[Matrix Whittaker processes]{Matrix Whittaker processes}
\maketitle

\begin{abstract}
We study {a discrete-time Markov process} on triangular arrays of matrices of {size} $d\geq 1$, driven by inverse Wishart random matrices.
The components of the right edge evolve as multiplicative random walks on positive definite matrices with one-sided interactions and can be viewed as a $d$-dimensional generalisation of log-gamma polymer partition functions.
We establish intertwining relations to prove that, for suitable initial configurations of the triangular process, the bottom edge has an autonomous Markovian evolution with an explicit transition kernel.
We then show that, for a {special} singular initial configuration, the fixed-time law of the bottom edge is a {matrix Whittaker measure}, which we define.
To achieve this, we perform a Laplace approximation that requires solving a constrained minimisation problem for certain energy functions of matrix arguments on directed graphs.
\end{abstract}

\tableofcontents

\section{Introduction}
\label{sec:intro}

In the last few decades, we have witnessed a surge of research on stochastic integrable models, often motivated by problems in mathematical physics and enriched by deep connections with algebraic combinatorics, representation theory, symmetric functions, and integrable systems~\cite{borodinPetrov14, borodinGorin16}.
Some of the most intensively studied models are interacting particle systems and stochastic growth processes in the Kardar-Parisi-Zhang (KPZ) universality class~\cite{corwin16, zygouras22}.

{From a mathematical perspective, it is natural to consider noncommutative versions of these models, {which have very recently received some attention}.}
In~\cite{oConnell21} a system of interacting Brownian particles in the space of positive definite matrices was considered and shown to have an integrable structure, related to the non-Abelian Toda chain and Whittaker functions of matrix arguments (the latter introduced in that article).
In the discrete-time setting, \cite{aristaBisiOConnell23} proved Matsumoto-Yor and Dufresne type theorems for a random walk on positive definite matrices.

{On the other hand, from the theoretical physics point of view, such matrix models may find interesting applications in quantum stochastic dynamics, as set out in~\cite{gautieEtAl21}.
In particular, \cite{gautieEtAl21} introduced a matrix generalisation of the classical Kesten recursion and studied a related quantum problem of interacting fermions in a Morse potential}.
Quoting the authors, their initial motivation was ``to explore possible matrix (non-commuting) generalizations of the famous directed polymer problem (which is related to the KPZ stochastic growth equation)''.

The subject of the present article is an integrable model of random walks on positive definite matrices with local interactions.
This constitutes, on the one hand, a discrete-time analogue of the matrix-valued interacting diffusions studied in~\cite{oConnell21} and, on the other hand, a matrix generalisation of the \emph{log-gamma polymer} model.

To motivate the contributions of this article, let us first define a discrete-time \emph{exclusion process} $\mathcal{Z}$ of $N\geq 1$ ordered particles $\mathcal{Z}^1\leq \mathcal{Z}^2 \leq \dots \leq \mathcal{Z}^N$ on $\Z$ moving to the right.
Let $(\mathcal{V}^1(n), \dots, \mathcal{V}^N(n))_{n\geq 1}$ be a collection of independent random variables supported on $\Z_{\geq 0}$.
At each time $n$, the particle positions are updated sequentially from the $1$-st one to the $N$-th one, as follows.
The $1$-st particle simply evolves as a random walk on $\Z$ with time-$n$ increment $\mathcal{V}^1(n)$.
Once the positions of the first $i-1$ particles have been updated, if the $(i-1)$-th particle has overtaken the $i$-th particle, then the latter is pushed forward to a temporary position to maintain the ordering; next, to complete its update, the $i$-th particle takes $\mathcal{V}^i(n)$ unit jumps to the right.
The particle locations then satisfy the recursive relations
\begin{align}
\label{eq:IPS1}
\mathcal{Z}^1(n) &= \mathcal{Z}^1(n-1) + \mathcal{V}^1(n) \, , \\
\label{eq:IPS2}
\mathcal{Z}^i(n) &= \max\left[\mathcal{Z}^{i-1}(n), \mathcal{Z}^i(n-1)\right] + \mathcal{V}^i(n) \, , \quad 2\leq i\leq N \, .
\end{align}
If one considers the initial state
\begin{equation}
\label{eq:stepInitialConf_IPS}
\mathcal{Z}^1(0)=\mathcal{Z}^2(0)=\cdots=\mathcal{Z}^N(0)=0 \, ,
\end{equation}
{then the following {\emph{last passage percolation} formula} holds:
\begin{equation*}
\mathcal{Z}^i(n) = \max_{\pi} \sum_{(m,k)\in \pi} \mathcal{V}^k(m) \, , \qquad 1\leq i\leq N\, ,
\end{equation*}
where the maximum is over all directed lattice paths $\pi$ in $\Z^2$ (i.e., at each lattice site $(m,k)$, $\pi$ is allowed to head either rightwards to $(m+1,k)$ or upwards to $(m,k+1)$) that start from $(1,1)$ {and} end at $(n,i)$.}
{As a process of last passage percolation times, $\mathcal{Z}$ can be also} associated with the \emph{corner growth process} with step (or `narrow wedge') initial configuration.
Remarkable integrable versions of this model are those with geometrically and exponentially distributed jumps, first studied in~\cite{johansson00}.

A positive temperature version of $\mathcal{Z}$ can be obtained by formally replacing the operations $(\max,+)$ with $(+,\times)$ in the relations~\eqref{eq:IPS1}-\eqref{eq:IPS2}.
Namely, given a collection of independent positive random variables $(V^1(n), \dots, V^N(n))_{n\geq 1}$, we can consider the discrete-time Markov process $Z$ defined by
\begin{align}
\label{eq:IPS_posTemp1}
Z^1(n) &= Z^1(n-1) V^1(n) \, , \\
\label{eq:IPS_posTemp2}
Z^i(n) &= \left[Z^{i-1}(n) + Z^i(n-1)\right] V^i(n) \, , \quad 2\leq i\leq N \, .
\end{align}
Considering the initial configuration
\begin{equation}
\label{eq:stepInitialConf}
Z^1(0)=1 \, , \qquad
Z^2(0)=\cdots=Z^N(0)=0 \, ,
\end{equation}
we have the closed-form expression
\begin{equation}
\label{eq:polymer}
Z^i(n) = \sum_{\pi} \prod_{{(m,k)}\in \pi} V^k(m) \, , \qquad 1\leq i\leq N\, ,
\end{equation}
where the sum is over all directed lattice paths $\pi$ in $\Z^2$ from $(1,1)$ to {$(n,i)$}.
The variables~\eqref{eq:polymer} can be regarded as partition functions of the $(1+1)$-dimensional \emph{directed polymer}, an intensively studied model of statistical mechanics.
Of particular importance is the model with inverse gamma distributed weights $V^i(n)$, known as the \emph{log-gamma polymer}, first considered in~\cite{seppalainen12}.
In~\cite{corwinEtAl14} it was shown that the laws of log-gamma polymer partition functions are marginals of \emph{Whittaker measures}; the latter are defined in terms of $\GL_d(\R)$-Whittaker functions and were introduced in that article.

In this article, we study a noncommutative generalisation of the above Markov process of log-gamma polymer partition functions.
The `particles' of this process live in $\pos_d$, the set of $d\times d$ positive definite real symmetric matrices.
The random weights $V^i(n)$ are now independent inverse Wishart matrices (a matrix generalisation of inverse gamma random variables{; see \S~\ref{subsec:notation}}).
We define $Z$ by setting
\begin{align}
\label{eq:rightEdge1}
Z^1(n) &:= Z^1(n-1)^{1/2} V^1(n) Z^1(n-1)^{1/2} \, , \\
\label{eq:rightEdge2}
Z^i(n) &:= \left[Z^{i-1}(n) + Z^i(n-1)\right]^{1/2} V^i(n) \left[Z^{i-1}(n) + Z^i(n-1)\right]^{1/2}, \quad 2\leq i\leq N \, ,
\end{align}
where, for $a\in\pos_d$, $a^{1/2}$ denotes the unique $b\in\pos_d$ such that $b^2=a$.
The above matrix products are symmetrised to ensure that, starting from any initial configuration $Z^i(0)\in \pos_d$, each $Z^i(n)$ still belongs to $\pos_d$ for all $n\geq 1$.
The $1$-st particle~\eqref{eq:rightEdge1} evolves as a ($\GL_d$-invariant) multiplicative random walk on $\pos_d$; on the other hand, the other particles~\eqref{eq:rightEdge2} can be viewed as analogous random walks with one-sided interactions.
From this point of view, the Markov process as a whole can be also regarded as a noncommutative version of the exclusion process $\mathcal{Z}$ defined in~\eqref{eq:IPS1}-\eqref{eq:IPS2}.
The natural generalisation of the initial configuration~\eqref{eq:stepInitialConf} is
\begin{equation}
\label{eq:stepInitialConf_matrix}
Z^1(0)=I_d \, , \qquad
Z^2(0) = \cdots = Z^N(0)=0_d \, ,
\end{equation}
where $I_d$ and $0_d$ are the $d\times d$ identity and zero matrices, respectively.
Notice that, although all but the first particle are initially zero, the process $Z$ starting from~\eqref{eq:stepInitialConf_matrix} lives in $\pos_d^N$ at all times $n\geq 1$.

In~\S~\ref{sec:MarkovDynamics}, we introduce a Markov process $X=(X(n))_{n\geq 0}$, $X(n) = (X^i_j(n))_{1\leq j\leq i\leq N}$, on triangular arrays of positive definite matrices whose {`right edge', namely $(X^1_1,\dots,X^N_1)$, equals $Z$}.
The evolution of $X$ may be viewed as a noncommutative version of the dynamics on Gelfand-Tsetlin patterns with blocking and pushing interactions, studied in various contexts {in~\cite{warren07, warrenWindridge09, nordenstam10, borodinFerrari14, borodinCorwin14, BorodinPetrov16}.} We refer to Fig.~\ref{fig:energyWhittaker} for a graphical representation of such a triangular array.
{Moreover, as we detail in Remark~\ref{rem:d=1}, the `left edge' of $X$ may be regarded as a noncommutative generalisation of the \emph{strict-weak polymer} studied in~\cite{oConnellOrtmann15, corwinSeppalainenShen15}.}

The first main result of this article (Theorem~\ref{thm:bottomRowMarkov}) states that, for certain special (random) initial configurations $X(0)$, the `bottom edge' $X^N = (X^N_1,\dots,X^N_N)$ of $X$ also has an autonomous Markovian evolution.
The transition kernel of $X^N$ is explicit and has an interpretation as a Doob $\mathit{h}$-transform with $\mathit{h}$-function given by a Whittaker function of matrix arguments.
To obtain this, we prove certain intertwining relations between kernels associated to the process $X$ and use the theory of Markov functions (reviewed in Appendix~\ref{app:markovFunctions}).
Another consequence of these intertwinings is that Whittaker functions are eigenfunctions of certain integral operators and possess a Feynman--Kac type interpretation.

Next, in~\S~\ref{sec:WhittakerMeasures}, we define {matrix Whittaker measures} on $\pos_d^N$ after proving an integral identity of Whittaker functions of matrix arguments (Theorem~\ref{thm:stade_matrix}), analogous to the well-known Cauchy-Littlewood identity for Schur functions.
The second main result of this article (Theorem~\ref{thm:specialInitialCond}) states that, for a {special} initial state, the fixed-time law of the bottom edge $X^N$ of $X$ is a {matrix Whittaker measure} on $\pos_d^N$.
Such an initial state, designed to match~\eqref{eq:stepInitialConf_matrix}, is {\emph{singular},} in the sense that the particles are at the `boundary' of $\pos_d$.

Due to the singularity of the initial configuration, the proof of Theorem~\ref{thm:specialInitialCond} will be based on a suitable limiting procedure and a careful integral approximation via Laplace's method.
This will require a digression on a constrained minimisation problem for certain energy functions of matrix arguments.
We chose to include this analysis in a separate section and to present it in the more general framework of directed graphs, as it may be of independent interest; see~\S~\ref{sec:minimisation}.
For us, the main application will be the asymptotic formula~\eqref{eq:WhittakerAsymptotics} for Whittaker functions of matrix arguments.

From our main results {we deduce} (see Corollary~\ref{coro:rightMarginal}) that, under the initial configuration~\eqref{eq:stepInitialConf_matrix}, the particles of the process $Z$ defined in~\eqref{eq:rightEdge1}-\eqref{eq:rightEdge2} have a fixed-time law given by the first marginal of a {matrix Whittaker measure} on $\pos_d^N$.
In the scalar $d=1$ case, we recover the aforementioned result of~\cite{corwinEtAl14} for the law of the log-gamma polymer partition functions. {{In Corollary~\ref{coro:rightMarginal}, we also} obtain an analogous result concerning the fixed-time law of the `left edge' of the triangular array $X$.}

It is worth mentioning that the log-gamma polymer partition functions \eqref{eq:polymer} were also studied in~\cite{corwinEtAl14} as embedded in a {dynamic} on triangular arrays.
However, such a {dynamic} was constructed via the combinatorial mechanism of the geometric Robinson--Schensted--Knuth correspondence; in particular, at each time step, the right edge is updated using $N$ new (independent) random variables, whereas all the other components are updated via deterministic transformations of the current state and the newly updated right edge.
It turns out that, for $d=1$, the processes considered in~\cite{corwinEtAl14} and in the present article have an identical right edge and, under the special initial configuration of Theorem~\ref{thm:bottomRowMarkov}, also a bottom edge process with the same Markovian evolution.
However, even in the $d=1$ case, the two processes, as a whole, differ.
The {dynamic} introduced in this article is driven by random updates with $N(N+1)/2$ degrees of freedom, since each particle of the triangular array is driven by an independent source of randomness (as well as by local interactions with the other particles).

\subsubsection*{Organisation of the article.}
In~\S~\ref{sec:WhittakerFunctions}, we define Whittaker functions of matrix arguments.
In~\S~\ref{sec:MarkovDynamics}, we introduce a Markov {dynamic} on triangular arrays of matrices and study the evolution of its bottom edge, using the theory of Markov functions; we also obtain a Feynman--Kac interpretation of Whittaker functions.
In~\S~\ref{sec:WhittakerMeasures}, we define {matrix Whittaker measures} (through a Whittaker integral identity) and prove that they naturally arise as fixed-time laws in the aforementioned triangular process under a singular initial configuration.
To do so, we need a Laplace approximation of Whittaker functions, which can be justified by solving a constrained minimisation problem for certain energy functions of matrix arguments on directed graphs: this is the content of~\S~\ref{sec:minimisation}.
In Appendix~\ref{app:cauchyLittlewood}, we give a proof of the Cauchy-Littlewood identity for Schur functions that resembles our proof of the Whittaker integral identity.
In Appendix~\ref{app:markovFunctions}, we review the theory of Markov functions for inhomogeneous discrete-time Markov processes.
Finally, in Appendix~\ref{app:convergenceLemma}, we prove a convergence lemma related to weak convergence of probability measures.

\subsection{Notation and preliminary notions}
\label{subsec:notation}

Here we introduce some notation and preliminary notions that we use throughout this work.
For background and proofs, we refer to~\cite{hornJohnson13, terras16}.

\subsubsection*{Positive definite matrices}
Let $\pos_d$ be the set of all $d\times d$ \emph{positive definite} matrices, i.e.\ $d\times d$ real symmetric matrices with positive eigenvalues.
Throughout this article, for $x\in\pos_d$, we denote by $\abs{x}$ the determinant of $x$ and by $\tr[x]$ its trace.

The following properties hold:
\begin{itemize}
\item $x\in\pos_d$ if and only if $x^{-1} \in \pos_d$;
\item if $x\in\pos_d$ and $\lambda>0$, then $\lambda x \in \pos_d$;
\item if $x,y\in\pos_d$, then $x+y\in\pos_d$ (but in general $xy\notin \pos_d$);
\item $x-y\in \pos_d$ if and only if $y^{-1} - x^{-1} \in\pos_d$.
\end{itemize}

For $x\in\pos_d$, there exists a unique $y\in\pos_d$ such that $y^2 = x$; we denote such a $y$ by $x^{1/2}$.

For any $y\in\pos_d$, we define the (noncommutative) `multiplication operation' by $y$ as
\begin{equation}
\label{eq:mult}
T_y\colon \pos_d\to\pos_d \, , \qquad
T_y(x) := y^{1/2} x y^{1/2} \, , \qquad
x\in\pos_d \, .
\end{equation}
{Such a symmetrised product will be used to construct a multiplicative random walk on $\pos_d$ (see Definition~\ref{def:RW} and Remark~\ref{rem:invRW} below).}

We also denote by $I_d$ and $0_d$ the $d\times d$ identity matrix and zero matrix, respectively.

\subsubsection*{Measure and integration on $\pos_d$}

Let $\GL_d$ be the group of $d\times d$ invertible real matrices.
Define the measure $\mu$ on $\pos_d$ by
\begin{equation}
\label{eq:measureMu}
\mu(\diff x) := \abs{x}^{-\frac{d+1}{2}}\prod_{1\leq i\leq j\leq d}\diff x_{i,j} \, ,
\end{equation}
where $\diff x_{i,j}$ is the Lebesgue measure on $\mathbb{R}$ in the variable $x_{i,j}$.
Such a measure is the \emph{$\GL_d$-invariant measure} on $\pos_d$, in the sense that
\begin{align*}
\int_{\pos_d}f\big(a^{\top} x a\big)\mu(\diff x)=\int_{\pos_d}f(x)\mu(\diff x)
\end{align*}
for all $a\in \GL_d$ and for all suitable functions $f$.
In other words, $\mu$ is invariant under the group action of $\GL_d$ on $\pos_d$
\[
\GL_d \times \pos_d \to \pos_d \, , \qquad\qquad
(a,x) \mapsto a^{\top}xa \, .
\]
Furthermore, the measure $\mu$ is preserved under the involution $x\mapsto x^{-1}$.

{\subsubsection*{Wishart distributions and gamma functions}
For $\alpha > \frac{d-1}{2}$, we will refer to the \emph{($d$-variate) Wishart distribution} with parameter $\alpha$ as the probability measure
\begin{align}
\label{eq:Wishart}
\frac{1}{\Gamma_{d}(\alpha)}
\abs{x}^{\alpha}
\e^{-\tr [x]}
\mu(\diff x)
\end{align}
on $\pos_d$, {where $\Gamma_{d}(\alpha)$ is the \emph{$d$-variate gamma function}, i.e.}
\begin{equation*}
\Gamma_{d}(\alpha)
:= \int_{\pos_d} \mu(\diff x) \abs{x}^{\alpha} \e^{-\tr[x]}
= \int_{\pos_d} \mu(\diff x) \abs{x}^{-\alpha} \e^{-\tr[x^{-1}]}
= \pi^{\frac{d(d-1)}{4}} \prod_{k=1}^d \Gamma\left(\alpha - \frac{k-1}{2} \right) .
\end{equation*}
The inverse of a Wishart matrix with parameter $\alpha$ has the distribution
\begin{align}
\label{eq:invWishart}
\frac{1}{\Gamma_{d}(\alpha)}
\abs{x}^{-\alpha}
\e^{-\tr[x^{-1}]}
\mu(\diff x)
\end{align}
on $\pos_d$.
We will refer to the latter as the \emph{($d$-variate) inverse Wishart distribution} with parameter $\alpha$.}

\subsubsection*{Kernels and integral operators}
Let $(S, \mathcal{S})$ and $(T, \mathcal{T})$ be two measurable spaces.
Let $\meas{\mathcal{S}}$ denote the set of complex-valued measurable functions on $(S, \mathcal{S})$.
For our purposes, a \emph{kernel} from $T$ to $S$ will be a map $L \colon T\times\mathcal{S}\to\C$ such that, for each $t\in T$, $L(t;\cdot)$ is a (complex) measure on $(S,\mathcal{S})$ and, for each $A\in\mathcal{S}$, $L(\cdot; A)$ is an element of $\meas{\mathcal{T}}$.
The kernel $L$ can be also, alternatively, thought of as an \emph{integral operator}
\begin{align}
\label{eq:integralOperator}
L\colon \meas{\mathcal{S}} \to \meas{\mathcal{T}} \, , \qquad
L f(t)
:=\int_{S}L(t;\diff s) f(s) \qquad
\text{for } f\in \meas{\mathcal{S}} \, , \,\, t\in T \, ,
\end{align}
whenever the integral is well defined.
Clearly, the composition of kernels/operators yields another kernel/operator; such a composition is associative but, in general, not commutative.
When the complex measure $L(t; \cdot)$ is a probability measure for all $t\in T$, we will talk about \emph{Markov kernels/operators}.

Throughout this article, the measurable spaces will be usually Cartesian powers of $\pos_d$ (which we denote by $\pos_d^k$, $k\geq 1$), with their Borel sigma-algebras.
Moreover, for a kernel $L$ from $\pos_d^k$ to $\pos_d^\ell$, the measure $L(t; \cdot)$ will be, in most cases, absolutely continuous with respect to the reference product measure $\mu^{\otimes \ell}$ on $\pos_d^\ell$, for any $t\in\pos_d^k$; with a little abuse of notation, we will then also write $s\mapsto L(t;s)$ for the corresponding density (a measurable function on $\pos_d^\ell$).

\section{Whittaker functions}
\label{sec:WhittakerFunctions}

In this section we define Whittaker functions of matrix arguments following~\cite{oConnell21}, and then extend them to a further level of generality.
Notice also that the kernels~\eqref{eq:kerK} and~\eqref{eq:kerP} defined below are matrix versions of certain kernels defined in~\cite[\S~3.1]{corwinEtAl14} and~\cite[\S~2]{oConnellSeppalainenZygouras14} (see also references therein).

\subsection{Whittaker functions of matrix arguments}
\label{subsec:WhittakerFn}

We define Whittaker functions of matrix arguments as integrals over `triangular arrays' of $d\times d$ positive definite matrices.
For $N\geq 1$, denote by $\trian{N}{d} := \pos_d \times \pos_d^2 \times \dots \times \pos_d^{N}$ the set of height-$N$ triangular arrays
\begin{equation}
\label{eq:triangularArray}
x=(x^1,\dots,x^N)=(x^i_j)_{1\leq j\leq i\leq N} \, ,
\end{equation}
where $x^i=(x^i_1, \dots, x^i_i)\in \pos_d^i$ will be referred to as the $i$-th row of $x$, for $1\leq i\leq N$.
For $\lambda =(\lambda_1,\dots,\lambda_N) \in \C^N$ and $x\in\trian{N}{d}$, let
\begin{align}
\label{eq:type_Whittaker}
\Delta^N_{\lambda}(x) &:= 
\big\lvert x^1_1 \big\rvert^{-\lambda_1}
\prod_{i=2}^N \left( \frac{\big\lvert x^i_1 \cdots x^i_i\big\rvert}{\big\lvert x^{i-1}_1 \cdots x^{i-1}_{i-1}\big\rvert} \right)^{-\lambda_i} , \\
\label{eq:energyFn_Whittaker}
\Phi^N(x) &:= \sum_{i=1}^{N-1} \sum_{j=1}^i \left(\tr\left[x^{i+1}_{j+1} (x^i_j)^{-1}\right] + \tr\left[x^i_j (x^{i+1}_j)^{-1}\right]\right) .
\end{align}
For a graphical representation of the array~\eqref{eq:triangularArray} and of the `energy function' $\Phi^N$, see Fig.~\ref{fig:energyWhittaker}.
\begin{figure}
\centering
\begin{tikzpicture}[scale=1]

\node (x11) at (1,-1) {$x^1_1$};
\node (x21) at (2,-2) {$x^2_1$};
\node (x22) at (0,-2) {$x^2_2$};
\node (x31) at (3,-3) {$x^3_1$};
\node (x32) at (1,-3) {$x^3_2$};
\node (x33) at (-1,-3) {$x^3_3$};
\node (x41) at (4,-4) {$x^4_1$};
\node (x42) at (2,-4) {$x^4_2$};
\node (x43) at (0,-4) {$x^4_3$};
\node (x44) at (-2,-4) {$x^4_4$};

\draw[->] (x22) -- (x11);
\draw[->] (x11) -- (x21);
\draw[->] (x33) -- (x22);
\draw[->] (x22) -- (x32);
\draw[->] (x32) -- (x21);
\draw[->] (x21) -- (x31);
\draw[->] (x44) -- (x33);
\draw[->] (x33) -- (x43);
\draw[->] (x43) -- (x32);
\draw[->] (x32) -- (x42);
\draw[->] (x42) -- (x31);
\draw[->] (x31) -- (x41);

\end{tikzpicture}
\caption{Graphical representation of a 'triangular' array $x\in \trian{N}{d}$ as in~\eqref{eq:triangularArray}, for $N=4$.
Each row $x^i$, $1\leq i\leq N$, consists of the matrices $(x^i_1,\dots,x^i_i)$, read from right to left.
The arrows refer to the
energy function $\Phi^N(x)$ in~\eqref{eq:energyFn_Whittaker}, where every summand $\tr[ab^{-1}]$ corresponds to an arrow pointing from $a$ to $b$ in the figure.}
\label{fig:energyWhittaker}
\end{figure}
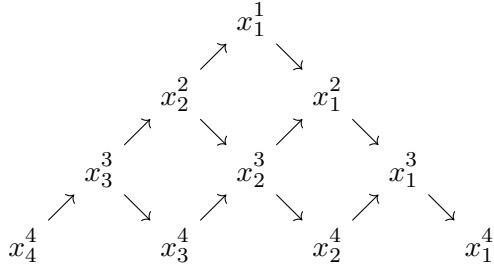
For $z=(z_1,\dots,z_N)\in\pos_d^N$, let $\trian{N}{d}(z) \subset \trian{N}{d}$ be the set of all height-$N$ triangular arrays $x$ with $N$-th row $x^N = z$.
We define the \emph{Whittaker function} $\psi^N_{\lambda}(z)$ with argument $z\in \pos_d^N$ and parameter $\lambda \in \C^N$ as
\begin{equation}
\label{eq:WhittakerClosedForm}
\psi^N_{\lambda}(z)
:= \int_{\trian{N}{d}(z)}
\Bigg(\prod_{i=1}^{N-1} \prod_{j=1}^i \mu(\diff x^i_j)\Bigg)
\Delta^N_{\lambda}(x) \e^{-\Phi^N(x)} \, .
\end{equation}
Notice that, for $N=1$, the expression above reduces to $\psi^1_{\lambda}(z) = \abs{z}^{-\lambda}$.
As proved in~\cite{oConnell21}, the integral~\eqref{eq:WhittakerClosedForm} is absolutely convergent for all $\lambda\in\C^N$, so that Whittaker functions are well defined.

For our purposes, it is convenient to rewrite Whittaker functions in terms of certain kernels that we now introduce.
For $N\geq 1$, $\lambda\in\C^N$ and $x\in\trian{N}{d}$, define the kernel
\begin{equation}
\label{eq:kerSigma}
\Sigma^N_{\lambda}(x^N; \diff x^{1:(N-1)})
:= \Delta^N_{\lambda}(x) \e^{-\Phi^N(x)}
\prod_{i=1}^{N-1} \prod_{j=1}^i \mu(\diff x^i_j) \, ,
\end{equation}
where, as always from now on, $i:j$ denotes the tuple $(i,i+1,\dots,j-1, j)$ for $i\leq j$, so that $x^{1:(N-1)}\in\trian{N-1}{d}$ is the triangular array consisting of the first $N-1$ rows of $x$.
Notice that, for $N=1$, \eqref{eq:kerSigma} reduces to $\Sigma^1_{\lambda}(z; \emptyset) = \abs{z}^{-\lambda}=\psi^1_{\lambda}(z)$.
For $z\in\pos_d^N$, let us also define the kernel
\begin{equation}
\label{eq:kerSigma_tilde}
\tilde{\Sigma}^N_{\lambda}(z; \diff x)
:= \delta(z; \diff x^N) \, \Sigma^N_{\lambda}(x^N; \diff x^{1:(N-1)}) \, ,
\end{equation}
where $\delta$ is the Dirac delta kernel on $\pos_d^N$.
Then, the Whittaker function~\eqref{eq:WhittakerClosedForm} can be written as
\begin{equation}
\label{eq:Whittaker_triangle}
\psi^N_{\lambda}(z)
= \int_{\trian{N-1}{d}} \Sigma^{N}_{\lambda}(z; \diff x)
= \int_{\trian{N}{d}} \tilde{\Sigma}^{N}_{\lambda}(z; \diff x) \, .
\end{equation}
Moreover, for $N\geq 2$, $b\in\C$, $z=(z_1,\dots,z_N)\in \pos_d^N$, and $y=(y_1,\dots,y_{N-1}) \in \pos_d^{N-1}$, let
\begin{equation}
\label{eq:kerK}
K^N_{b}(z; y)
= \left(\prod_{i=1}^N \abs{z_i}^{-b}\right)
\prod_{j=1}^{N-1} \abs{y_j}^{b}
\e^{ - \tr\left[z_{j+1} y_j^{-1} + y_j z_j^{-1}\right] } \, .
\end{equation}
We will usually regard~\eqref{eq:kerK} as a kernel by setting $K^N_b(z;\diff y) := K^N_b(z;y) \mu^{\otimes (N-1)}(\diff y)$.
We then have, for $\lambda\in\C^N$, $z\in \pos_d^N$, and $x\in\trian{N-1}{d}$,
\begin{equation}
\label{eq:kerSigma_recursive}
\begin{split}
\Sigma^N_{\lambda}(z;\diff x) &=
K^N_{\lambda_N}(z; \diff x^{N-1})
K^{N-1}_{\lambda_{N-1}}(x^{N-1}; \diff x^{N-2}) \cdots
K^2_{\lambda_2}(x^2;\diff x^1)
\psi^1_{\lambda}(x^1) \\
&= K^N_{\lambda_N}(z; \diff x^{N-1})
\Sigma^{N-1}_{(\lambda_1,\dots,\lambda_{N-1})}(x^{N-1}; \diff x^{1:(N-2)}) \, .
\end{split}
\end{equation}
This yields a recursive definition of Whittaker functions: 
\begin{equation}
\label{eq:Whittaker}
\psi^N_{\lambda}(z)
= \begin{cases}
\abs{z}^{-\lambda} &N=1 \, , \\
K^N_{\lambda_N}\psi^{N-1}_{(\lambda_1,\dots,\lambda_{N-1})}(z)
= K^N_{\lambda_N} K^{N-1}_{\lambda_{N-1}} \cdots K^2_{\lambda_2} \psi^1_{\lambda_1}(z) &N\geq 2 \, .
\end{cases}
\end{equation}

\subsection{A generalisation of Whittaker functions}
\label{subsec:Whittaker_gen}

We now introduce a generalisation of Whittaker functions of matrix arguments,
which will naturally emerge in~\S~\ref{subsec:bottomEdge} and, in the scalar case $d=1$, corresponds to the one considered in~\cite{oConnellSeppalainenZygouras14}.
These generalised Whittaker functions are integrals over \emph{trapezoidal} arrays of positive definite matrices, similarly to how the Whittaker functions of \S~\ref{subsec:WhittakerFn} are defined as integrals over triangular arrays.

Let $n\geq N\geq 1$ and denote by
\[
\trian{N,n}{d} := \pos_d \times \pos_d^2 \times \dots \times \pos_d^{N} \times \underbrace{\pos_d^N \times \dots \times \pos_d^N}_{n-N \text{ times}}
\]
the set of trapezoidal arrays
\begin{equation}
\label{eq:trapezoidalArray}
x=(x^1,\dots,x^n)=(x^i_j\colon 1\leq i\leq n, \, 1\leq j\leq i \wedge N) \, ,
\end{equation}
with $i$-th row $x^i=(x^i_1, \dots, x^i_{i \wedge N})\in \pos_d^{i \wedge N}$, for $1\leq i\leq n$ (here $i \wedge N$ denotes the minimum between $i$ and $N$).
For $\lambda\in\C^n$, $x\in\trian{N,n}{d}$ and $s\in\pos_d$, let
\begin{align}
\label{eq:type_WhittakerExt}
\Delta^{N,n}_{\lambda}(x) &:=
\big\lvert x^1_1 \big\rvert^{-\lambda_1}
\prod_{i=2}^N \left( \frac{\big\lvert x^i_1 \cdots x^i_i\big\rvert}{\big\lvert x^{i-1}_1 \cdots x^{i-1}_{i-1}\big\rvert} \right)^{-\lambda_i}
\prod_{i=N+1}^n \left( \frac{\big\lvert x^i_1 \cdots x^i_N\big\rvert}{\big\lvert x^{i-1}_1 \cdots x^{i-1}_{N} \big\rvert} \right)^{-\lambda_i} , \\
\label{eq:energyFnWhittakerExt}
\Phi^{N,n}_s(x) &:= \tr\big[s (x^N_N)^{-1}\big]
+ \sum_{i=1}^{n-1} \left( \sum_{j=1}^{i\wedge (N-1)} \tr\big[x^{i+1}_{j+1} (x^i_j)^{-1}\big]
+ \sum_{j=1}^{i\wedge N} \tr\big[x^i_j (x^{i+1}_j)^{-1}\big] \right) .
\end{align}
See Fig.~\ref{fig:energyWhittakerExt} for a graphical representation of the array~\eqref{eq:trapezoidalArray} and of the energy function $\Phi^{N,n}_s$.
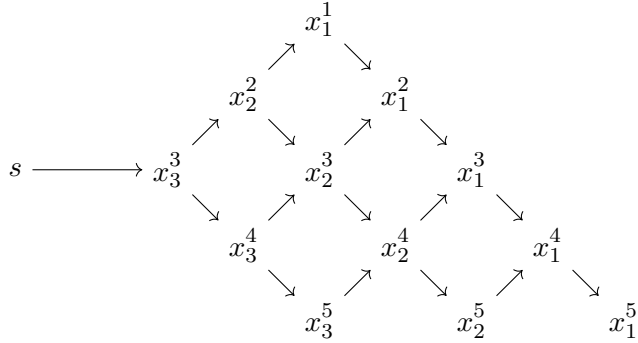
\begin{figure}
\centering
\begin{tikzpicture}[scale=1]

\node (x11) at (1,-1) {$x^1_1$};
\node (x21) at (2,-2) {$x^2_1$};
\node (x22) at (0,-2) {$x^2_2$};
\node (x31) at (3,-3) {$x^3_1$};
\node (x32) at (1,-3) {$x^3_2$};
\node (x33) at (-1,-3) {$x^3_3$};
\node (x41) at (4,-4) {$x^4_1$};
\node (x42) at (2,-4) {$x^4_2$};
\node (x43) at (0,-4) {$x^4_3$};
\node (x51) at (5,-5) {$x^5_1$};
\node (x52) at (3,-5) {$x^5_2$};
\node (x53) at (1,-5) {$x^5_3$};
\node (s) at (-3,-3) {$s$};

\draw[->] (x22) -- (x11);
\draw[->] (x11) -- (x21);
\draw[->] (x33) -- (x22);
\draw[->] (x22) -- (x32);
\draw[->] (x32) -- (x21);
\draw[->] (x21) -- (x31);
\draw[->] (x33) -- (x43);
\draw[->] (x43) -- (x32);
\draw[->] (x32) -- (x42);
\draw[->] (x42) -- (x31);
\draw[->] (x31) -- (x41);
\draw[->] (x43) -- (x53);
\draw[->] (x53) -- (x42);
\draw[->] (x42) -- (x52);
\draw[->] (x52) -- (x41);
\draw[->] (x41) -- (x51);
\draw[->] (s) -- (x33);

\end{tikzpicture}
\caption{Graphical representation of a trapezoidal array $x\in \trian{N,n}{d}$ as in~\eqref{eq:trapezoidalArray}, for $N=3$ and $n=5$.
The additional variable $s$ appears in the definition~\eqref{eq:energyFnWhittakerExt} of the energy function $\Phi^{N,n}_s(x)$, in which every summand $\tr[ab^{-1}]$ corresponds to an arrow pointing from $a$ to $b$ in the figure.}
\label{fig:energyWhittakerExt}
\end{figure}
For $z\in\pos_d^N$, let $\trian{N,n}{d}(z) \subset \trian{N,n}{d}$ be the set of all trapezoidal arrays $x$ with $n$-th row $x^n = z$.
For $n\geq N$,  $\lambda\in\C^n$, $s\in\pos_d$ and $z\in \pos_d^N$, we define
\begin{equation}
\label{eq:WhittakerTrapezoid}
\psi^{N,n}_{\lambda;s}(z)
:= \int_{\trian{N,n}{d}(z)}
\left( \prod_{i=1}^{n-1} \prod_{j=1}^{i\wedge N}  \mu(\diff x^i_j) \right)
\Delta^{N,n}_{\lambda}(x)
\e^{-\Phi_s^{N,n}(x)} \, .
\end{equation}
Notice that, {if $s=0_d$} and $n=N$, $\psi^{N,N}_{\lambda;0}= \psi^N_{\lambda}$ corresponds to the Whittaker function defined in~\eqref{eq:WhittakerClosedForm}.
The absolute convergence of the integral in~\eqref{eq:WhittakerTrapezoid}, for all $\lambda\in \C^n$, can be shown by adapting the proof of~\cite[Prop.~6-(i)]{oConnell21}.

Let us now give an equivalent representation of these generalised Whittaker functions.
The following kernel will play a central role in this work.
For $a\in\C$ and $z,\tilde{z}\in \pos_d^N$, set
\begin{equation}
\label{eq:kerP}
\begin{split}
P^N_{a}(z; \tilde{z})
:= \left(\prod_{i=1}^{N-1}
\e^{-\tr[\tilde{z}_{i+1} z_i^{-1}]}\right)
\prod_{j=1}^{N}
\abs{z_j \tilde{z}_j^{-1}}^{a}
\e^{- \tr\left[z_j \tilde{z}_j^{-1} \right] }.
\end{split}
\end{equation}
We will see $P^N_{a}(z;\tilde{z})$ as a measure in either of the two arguments, defining
\begin{equation}
\label{eq:kerP_inverse}
P^N_{a}(z; \diff\tilde{z}) := P^N_{a}(z; \tilde{z}) \mu^{\otimes N}(\diff \tilde{z})
\qquad\text{and}\qquad
\cev{P}^N_{a}(z; \diff \tilde{z}) := P^N_{a}(\tilde{z};z) \mu^{\otimes N}(\diff \tilde{z}) \, .
\end{equation}
We then have
\begin{equation}
\label{eq:WhittakerExt}
\psi^{N,n}_{\lambda;s}(z)
:= 
\begin{dcases}
\e^{-\tr[s z_N^{-1}]} \psi^N_{\lambda}(z) & n=N \, , \\
\cev{P}^N_{\lambda_n}
\psi^{N,n-1}_{(\lambda_1,\dots,\lambda_{n-1});s}(z)
= \cev{P}^N_{\lambda_n} \cev{P}^N_{\lambda_{n-1}} \cdots \cev{P}^N_{\lambda_{N+1}} \psi^{N,N}_{(\lambda_1,\dots,\lambda_N);s}(z)
&n>N \, .
\end{dcases}
\end{equation}

We also record here two relations between the kernels~\eqref{eq:kerK} and~\eqref{eq:kerP}, which follow directly from the definitions:
\begin{align}
\label{eq:kerP-K}
K^N_a(z;y)
&= \abs{s}^{-a}
\e^{\tr[s z_N^{-1}]}
P^N_a(y_1,\dots,y_{N-1},s;z)
 \, , \\
\label{eq:kerK-P}
K^N_a(z;y)
&= \abs{z_N}^{-a} \e^{-\tr[z_N y_{N-1}^{-1}]}
P^{N-1}_{a}(y;z_1,\dots,z_{N-1}) \, ,
\end{align}
for $y=(y_1,\dots,y_{N-1})\in \pos_d^{N-1}$, $s\in\pos_d$, and $z=(z_1,\dots,z_N) \in \pos_d^N$.
Taking $a=\lambda_N$ in~\eqref{eq:kerK-P}, multiplying both sides by $\psi^{N-1}_{(\lambda_1,\dots,\lambda_{N-1})}(y)$, integrating over $\pos_d^{N-1}$ with respect to $\mu^{\otimes (N-1)}(\diff y)$, and using~\eqref{eq:Whittaker} and~\eqref{eq:WhittakerExt}, we obtain the identity
\begin{equation}
\label{eq:Whittaker_triangle&trapezoid}
\psi^N_{\lambda}(z)
= \abs{z_N}^{-\lambda_N} \psi^{N-1,N}_{\lambda;z_N}(z_1,\dots,z_{N-1}) \, .
\end{equation}

{\begin{remark}
{Let us mention} that we anticipate the function $\psi^{N,n}_{\lambda;s}$ to be symmetric in the parameters $\lambda_1,\dots, \lambda_n$.
This is not obvious from the definition, but it is suggested by an integral identity of Whittaker functions of matrix arguments that will be proven later on (see~\eqref{eq:stade_matrix}).
As argued in~\cite[\S~7.1]{oConnell21}, this symmetry is true at least in the case $N=n=2$.
Moreover, it is known for $d=1$ and arbitrary $n,N$; see, for example, \cite{KharchevLebedev01}, \cite{GerasimovLebedevOblezin08} and \cite[pp.~369--370]{oConnellSeppalainenZygouras14}.
\end{remark}}

\section{Markov dynamics}
\label{sec:MarkovDynamics}

In this section, we define a Markov process $X$ on triangular arrays, which can be viewed as a system of interacting random walks on $\pos_d$.
Next, we prove intertwining relations between certain {transition} kernels related to this process.
This implies, via the theory of Markov functions, that, under certain random initial configurations, the bottom edge of the triangular process $X$ has an autonomous stochastic evolution.
A consequence of these results is that Whittaker functions of matrix arguments are eigenfunctions of certain integral operators and, thereupon, admit a Feynman--Kac interpretation.

\subsection{Interacting Markov dynamics on triangular arrays}
\label{subsec:MarkovDynamics}

Let $\ort_d$ be the real orthogonal group in dimension $d$.
Recall that a random matrix $Y$ in $\pos_d$ is said to be $\ort_d$-invariant (or orthogonally invariant) if $k^{\top} Y k$ has the same distribution of $Y$, for every $k\in\ort_d$.

\begin{definition}
\label{def:RW}
Let $(W(n))_{n\geq 1}$ be a family of independent and $\ort_d$-invariant random matrices in $\pos_d$.
The \emph{$\GL_d$-invariant random walk on $\pos_d$} with initial state $r\in\pos_d$ and increments $(W(n))_{n\geq 1}$ is the $\pos_d$-valued process $R=(R(n))_{n\geq 0}$ such that $R(0)=r$ and
\begin{equation}
\label{eq:RW}
R(n)
:= T_{R(n-1)}(W(n))
= R(n-1)^{1/2} W(n) R(n-1)^{1/2} \, ,
\qquad
n\geq 1 \, .
\end{equation}
\end{definition}

\begin{remark}
\label{rem:invRW}
The random walk $R$ of Definition~\ref{def:RW} is indeed \emph{$\GL_d$-invariant}, in the sense that the conjugated walk $(g^{\top} R(n) g)_{n\geq 0}$ has the same transition kernels for any choice of $g\in\GL_d$ (cf.~\cite[\S~3]{aristaBisiOConnell23}).
Instead of~\eqref{eq:RW}, one could consider a different process through the alternative symmetrisation
\[
R'(n)
:= T_{W(n)}(R'(n-1))
= W(n)^{1/2} R'(n-1) W(n)^{1/2} \, .
\]
One can check that the resulting random walk $R'$ is $\ort_d$-invariant, but in general not $\GL_d$-invariant.
In principle, one could proceed to obtain analogous results to those presented in the present article using this alternative symmetrisation (for a similar approach in the continuous Brownian setting, see~\cite[Prop.~3.5]{oConnell21}).
However, from our point of view, the choice~\eqref{eq:RW} is the most natural and leads to more explicit transition kernels throughout.
\end{remark}

It is well known that the {Wishart distribution~\eqref{eq:Wishart} and the inverse Wishart distribution~\eqref{eq:invWishart} are} $\ort_d$-invariant.
In this article, we will focus on $\GL_d$-invariant random walks with inverse Wishart increments. 

Recall from definition~\eqref{eq:kerP} that $P^1_a(z;\diff \tilde{z})= \big\lvert z\tilde{z}^{-1}\big\rvert^{a} \e^{-\tr[z\tilde{z}^{-1}]} \mu(\diff \tilde{z})$ for $a\in\C$.
Using a straightforward change of variables, we see that, if $\Re(a)>\frac{d-1}{2}$,
\begin{equation}
\label{eq:P_normalization_N=1}
\int_{\pos_d} P^1_{a}(z;\diff \tilde{z}) = \Gamma_d(a)
\qquad\qquad \text{for any } z\in\pos_d \, .
\end{equation}
Define then the renormalised kernel
\begin{equation}
\label{eq:kerP_normal_N=1}
\overline{P}^1_a(z;\diff \tilde{z}) := \frac{1}{\Gamma_d(a)} P^1_a(z;\diff \tilde{z}) \, .
\end{equation}
It is immediate to see that the (time-homogeneous) $\GL_d$-invariant random walk on $\pos_d$ with inverse Wishart increments of parameter $a>\frac{d-1}{2}$ has transition kernel $\overline{P}^1_a$.

We now define a discrete-time Markov process $X=(X(n))_{n\geq 0}$ on the set $\trian{N}{d}$ of height-$N$ triangular arrays {whose components are elements of $\pos_d$}.

\begin{definition}
\label{def:triangularProcess}
Fix a sequence of real parameters $\alpha=(\alpha(n))_{n\geq 1}$, an integer $N\geq 1$,  and a real $N$-tuple $\beta=(\beta^1,\dots,\beta^N)$ such that $\alpha(n) + \beta^i >(d-1)/2$ for all $n,i$.
Denote by $\alpha(n)+\beta$ the $N$-tuple $(\alpha(n)+\beta^1,\dots,\alpha(n)+\beta^N)$.
For $n\geq 1$ and $1\leq j\leq i\leq N$, let $W^i_j(n)$ be an inverse Wishart random matrix with parameter $\alpha(n)+\beta^i$ (the same parameter across $j$); assume further that all these random matrices are independent of each other.
We define the process $X=(X(n))_{n\geq 0}$, where $X(n) = (X^i_j(n))_{1\leq j\leq i\leq N}$ is a random element of $\trian{N}{d}$, as follows: given an initial state $X(0)$ in $\trian{N}{d}$, for $n\geq 1$ we set recursively
\begin{align}
\label{eq:triangularProcess}
X^i_j(n) :=
\begin{cases}
T_{X^1_1(n-1)}(W^1_1(n)) &1=j=i \\
T_{X^{i-1}_1(n)+ X^i_1(n-1)}(W^i_1(n)) &1=j<i\leq N \\
\left[X^{i-1}_{i-1}(n-1)^{-1} + T_{X^i_i(n-1)}(W^i_i(n))^{-1} \right]^{-1} &1<j=i\leq N \\
\left[X^{i-1}_{j-1}(n-1)^{-1} + T_{X^{i-1}_j(n) + X^i_j(n-1)}(W^i_j(n))^{-1} \right]^{-1} &1<j<i\leq N
\end{cases}
\end{align}
The $i$-tuple $X^i:=(X^i_1,\dots,X^i_i)$ will be referred to as the $i$-th \emph{row} of $X$.
\end{definition}

The fact that each $X^i_j(n)$ takes values in $\pos_d$ follows by standard properties of positive definite matrices (cf.~\S~\ref{subsec:notation}).
Notice that, adopting the convention $X^i_0(n)^{-1} = X^i_{i+1}(n)=0_d$ for all $i\geq 0$ and $n\geq 0$, then the last formula in~\eqref{eq:triangularProcess} can be taken as the definition of $X^i_j(n)$ for all $1\leq j\leq i\leq N$.

The {dynamic} on $\trian{N}{d}$ defined by~\eqref{eq:triangularProcess} implies that the `top particle' $X^1_1$ evolves as a $\GL_d$-invariant random walk in $\pos_d$ with inverse Wishart increments $(W^1_1(n))_{n\geq 1}$.

Furthermore, the `right edge' process {$(X^1_1,X^2_1,\dots,X^N_1)$} {equals the system $(Z^1,\dots,Z^N)$ of random particles in $\pos_d$ with one-sided interactions defined in~\eqref{eq:rightEdge1}-\eqref{eq:rightEdge2}, where the random weight $V^i(n)$ equals $W^i_1(n)$.}

{The `left edge' process $(X^1_1,X^2_2,\dots,X^N_N)$ also evolves as a system of particles in $\pos_d$ with one-sided interactions, as we now explain.
Set $L^i(n):= X^i_i(n)^{-1}$ and $U^i(n):= W^i_i(n)^{-1}$ for all $1\leq i\leq N$ and $n\geq 0$.
Then, $U^i(n)$ has the Wishart distribution with parameter $\alpha(n)+\beta^i$, and the process $L=(L^1,\dots,L^N)$ satisfies the recursions
\begin{align}
\label{eq:leftEdge1}
L^1(n) &= T_{L^1(n-1)}(U^1(n) ) \, , \\
\label{eq:leftEdge2}
L^i(n) &= L^{i-1}(n-1) + T_{L^i(n-1)} (U^i(n)), \quad 2\leq i\leq N \, .
\end{align}
Under the (singular) initial configuration
\begin{equation}
\label{eq:stepInitialConf_matrix_leftEdge}
L^1(0)=I_d \, , \qquad
L^2(0) = \cdots = {L}^N(0)=0_d \, ,
\end{equation}
one can see by induction that {$L^i(n)=0_d$ for all {$n<i-1$} and $L^i(i-1)=I_d$}, while $L^i(i)$ reduces to a sum of independent Wishart matrices:
\[
L^i(i)=U^1(1)+U^2(2)+\dots+U^i(i) \, , \quad 1\leq i\leq N \, .
\]
In particular, $L^i(i)$ has the Wishart distribution with parameter $\sum_{j=1}^i(\alpha(j)+\beta^j)$.
}

{
\begin{remark}
\label{rem:d=1}
We make a few remarks about various specialisations of the process $X$ and related Markov dynamics:
\begin{enumerate}
\item The interacting diffusion model on positive definite matrices studied in~\cite{oConnell21} (see also~\cite[\S~9]{oConnell12} for the $d=1$ case) can be regarded as a continuous-time analogue of the process $X$ defined in~\eqref{eq:triangularProcess}.
\item It seems that even the $d=1$ case of the dynamic~\eqref{eq:triangularProcess} has not been explicitly considered elsewhere.
It is related, even though not identical, to the process constructed in~\cite{corwinEtAl14} via the geometric Robinson--Schensted--Knuth correspondence; see the discussion in the introduction for further details.
\item For $d=1$, under the `step' initial configuration, the right edge can be regarded as a process of log-gamma polymer partition functions; see~\eqref{eq:stepInitialConf}-\eqref{eq:polymer} and the discussion therein.
\item For $d=1$, under the `step' initial configuration~\eqref{eq:stepInitialConf_matrix_leftEdge}, the left edge can be regarded as a process of strict-weak polymer partition functions in a gamma environment, studied in~\cite{oConnellOrtmann15, corwinSeppalainenShen15}.
A strict-weak path is a lattice path $\pi$ that, at each lattice site $(m,k)$, is allowed to head either horizontally to the right to $(m+1,k)$ or diagonally up-right to $(m+1,k+1)$.
It is easily seen that the process $L$ defined in~\eqref{eq:leftEdge1}-\eqref{eq:leftEdge2}, in the $d=1$ case, takes the closed form expression
\begin{equation}
\label{eq:strictWeak}
L^i(n) = \sum_\pi \prod_{e\in\pi} d_e \, ,
\end{equation}
where the sum is over all strict-weak paths $\pi$ from $(0,1)$ to $(n,i)$, the product is over all edges $e$ in the path $\pi$, and $d_e$ is a weight attached to the edge $e$ and defined as follows: $d_e:=1$ if $e$ is a diagonal edge from $(m,k)$ to $(m+1,k+1)$; $d_e:=U^{k}(m+1)$ (gamma distributed with parameter $\alpha(m+1)+\beta^k$) if $e$ is a horizontal edge from $(m,k)$ to $(m+1,k)$.
Formula~\eqref{eq:strictWeak} defines the \emph{strict-weak polymer partition function}.  
\item The $d=1$ case of~\eqref{eq:triangularProcess} is a `positive temperature' analogue (equivalently, a $(+,\times)$ version) of the process defined by
\[
\mathcal{X}^i_j(n) := \min\left(\mathcal{X}^{i-1}_{j-1}(n-1), \max\left(\mathcal{X}^{i-1}_j(n),\mathcal{X}^i_j(n-1)\right)+\mathcal{W}^i_j(n)\right) \, ,
\]
where $\mathcal{W}^i_j(n)$ are non-negative random variables representing jumps to the right (see e.g.~\cite{warrenWindridge09}).
Roughly speaking, particle $\mathcal{X}^i_j$ performs a random walk subject to certain interactions with other particles: it is pushed by $\mathcal{X}^{i-1}_j$ and blocked by $\mathcal{X}^{i-1}_{j-1}$.
\item Besides~\cite{warrenWindridge09}, other works~\cite{warren07, nordenstam10, borodinFerrari14, borodinCorwin14, BorodinPetrov16} studied, in various discrete and continuous settings, similar push-and-block dynamics on Gelfand-Tsetlin patterns driven by random updates with $N(N+1)/2$ degrees of freedom.
In particular, again in the case $d=1$, the process $X$ should correspond to a certain $q\to 1$ scaling limit of the $q$-Whittaker processes studied in~\cite{borodinCorwin14, BorodinPetrov16}.
\end{enumerate}
\end{remark}}

Motivated to obtain the explicit Markovian evolution of $X$, we now introduce the following kernels.
For $a\in\C$, $y=(y_1,\dots,y_{N-1})\in \pos_d^{N-1}$, $\tilde{y}=(\tilde{y}_1,\dots,\tilde{y}_{N-1})\in \pos_d^{N-1}$, $z=(z_1,\dots,z_N)\in \pos_d^N$, and $\tilde{z}=(\tilde{z}_1,\dots,\tilde{z}_N) \in \pos_d^N$, we set
\begin{align}
\label{eq:kerQ}
\begin{split}
Q^N_a(y,\tilde{y},z;\diff \tilde{z})
:=&
\prod_{j=1}^N \abs{ \left(\tilde{y}_j + z_j\right) \left(\tilde{z}_j^{-1} - y_{j-1}^{-1}\right) }^a
\e^{-\tr\left[(\tilde{y}_j + z_j)\left(\tilde{z}_j^{-1} - y_{j-1}^{-1}\right)\right]} \\
&\, \abs{I_d - \tilde{z}_j y_{j-1}^{-1}}^{-\frac{d+1}{2}}
\1_{\pos_d}\left(\tilde{z}_j^{-1} - y_{j-1}^{-1}\right)
\mu(\diff \tilde{z}_j) \, ,
\end{split}
\end{align}
with the convention $y_0^{-1}=\tilde{y}_N=0$.
Moreover, for $\lambda=(\lambda_1,\dots,\lambda_N) \in\C^N$, we set
\begin{equation}
\label{eq:kerPi}
\Pi^N_{\lambda}(x;\diff \tilde{x})
:= 
\begin{cases}
P^1_{\lambda}(x;\diff \tilde{x}) &\text{if } N=1 \, , \\ 
\Pi^{N-1}_{(\lambda_1,\dots,\lambda_{N-1})}(x^{1:(N-1)};\diff \tilde{x}^{1:(N-1)}) \,
Q^N_{\lambda_N}(x^{N-1}, \tilde{x}^{N-1}, x^N; \diff \tilde{x}^N) &\text{if } N\geq 2 \, ,
\end{cases}
\end{equation}
where $x \in \trian{N}{d}$ (resp., $\tilde{x} \in \trian{N}{d}$) is a height-$N$ triangular array of $d\times d$ positive definite matrices with $i$-th row $x^i \in \pos_d^i$ (resp., $\tilde{x}^i \in \pos_d^i$), according to the notation of \S~\ref{subsec:WhittakerFn}.
One can show (an analogous computation is made in the proof of Prop.~\ref{prop:intertwining}) that, if $\Re(a)>(d-1)/2$,
\begin{equation}
\label{eq:Q_normalisation}
\int_{\pos_d^N}
Q^N_{a}(y,\tilde{y},z; \diff \tilde{z})
= \Gamma_d(a)^N
\qquad\qquad \text{for any } y,\tilde{y}\in\pos_d^{N-1} \text{ and } z\in\pos_d^N \, .
\end{equation}
Using~\eqref{eq:P_normalization_N=1} and~\eqref{eq:Q_normalisation}, we see that, if $\Re(\lambda_i)>(d-1)/2$ for all $i$, then
\begin{equation}
\label{eq:Pi_normalisation}
\int_{\trian{N}{d}} \Pi^N_{\lambda}(x; \diff \tilde{x})
= \prod_{i=1}^N \Gamma_d(\lambda_i)^i
\qquad\qquad \text{for any } x\in\trian{N}{d} \, .
\end{equation}
Therefore, under the above conditions on the parameters, one can renormalise these kernels, so that they integrate to $1$:
\begin{align}
\label{eq:kerQ_normal}
\overline{Q}^N_{a}(y,\tilde{y},z; \diff \tilde{z})
&:= \frac{1}{\Gamma_d(a)^{N}} Q^N_a(y,\tilde{y},z; \diff \tilde{z}) \, , \\
\label{eq:kerPi_normal}
\overline{\Pi}^N_{\lambda}(x;\diff \tilde{x})
&:= \frac{1}{\prod_{i=1}^N \Gamma_d(\lambda_i)^i}
\Pi^N_{\lambda}(x;\diff \tilde{x}) \, .
\end{align}

The following result can be easily verified using the construction of $X$ in Definition~\ref{def:triangularProcess}.
\begin{proposition}
\label{prop:triangle_Markov}
Let $X$ as in Definition~\ref{def:triangularProcess}. Then, the conditional distribution of $X^N(n)$ given $X^{N-1}(n-1)=y$, $X^{N-1}(n)=\tilde{y}$ and $X^N(n-1)=z$, is $\overline{Q}^N_{\alpha(n)+\beta^N}(y,\tilde{y},z; \cdot )$.
Consequently, the process $X=(X(n))_{n\geq 0}$ is a time-inhomogeneous Markov process with state space $\trian{N}{d}$ and time-$n$ transition kernel $\overline{\Pi}^N_{\alpha(n)+\beta}$.
\end{proposition}

\subsection{Intertwining relations}
\label{subsec:intertwining}

We will now show that the Markov {dynamic} on $X$ (see Definition~\ref{def:triangularProcess}), when started from an appropriate random initial state, induces an autonomous Markov {dynamic} on the $N$-th row, or `bottom edge', of $X$.
This will be a consequence of an intertwining relation between kernels through the theory of Markov functions, which is reviewed in Appendix~\ref{app:markovFunctions} for the reader's convenience.

Let $N\geq 2$ and $a,b\in\C$.
Recalling the definitions~\eqref{eq:kerK} and~\eqref{eq:kerP} of the kernels $K^N_b$ and $P^N_a$, respectively, and denoting by $\delta$ the Dirac delta kernel on $\pos_d^N$, let us set
\begin{align}
\label{eq:kerK_tilde}
\tilde{K}^N_b(z; \diff y \diff \tilde{z})
&:= \delta(z; \diff \tilde{z}) \, K^N_b(\tilde{z}; \diff y) \, , \\
\label{eq:kerLambda}
\Lambda^N_{a,b}(y,z;\diff \tilde{y} \diff \tilde{z})
&:= \,
P^{N-1}_a(y;\diff \tilde{y}) \,
Q^N_{a+b}(y,\tilde{y},z;\diff\tilde{z}) \, ,
\end{align}
for $z,\tilde{z}\in\pos_d^N$ and {$y, \tilde{y}\in\pos_d^{N-1}$}.
We then have the following intertwining relation.

\begin{proposition}
\label{prop:intertwining}
Let $N\geq 2$ and $a,b\in\C$ such that $\Re(a+b)> (d-1)/2$.
Then,
\begin{align}
\label{eq:intertwining}
\tilde{K}^N_{b} \Lambda^N_{a,b} = \Gamma_d(a+b)^{N-1}  P^N_{a} \tilde{K}^N_{b}
\end{align}
holds as an equality between {kernels} from $\pos_d^N$ to $\pos_d^{N-1}\times\pos_d^N$.
\end{proposition}
\begin{proof}
We have to prove that $\tilde{K}^N_{b} \Lambda^N_{a,b} f(z) = \Gamma_d(a+b)^{N-1}  P^N_{a} \tilde{K}^N_{b} f(z)$, for any suitable test function $f\colon \pos_d^{N-1}\times \pos_d^N \to\R$ and any $z\in\pos_d^N$.
Using~\eqref{eq:kerK_tilde}, we see that this is equivalent to the identity
\begin{equation}
\label{eq:intertwining2}
\begin{split}
&\int_{\pos_d^{N-1}}
K^N_b(z; \diff y)
\int_{\pos_d^{N-1}\times \pos_d^N} \Lambda^N_{a,b}(y,z; \diff \tilde{y} \diff \tilde{z}) f(\tilde{y}, \tilde{z}) \\
= \, &\Gamma_d(a+b)^{N-1}
\int_{\pos_d^N} P^N_a(z;\diff \tilde{z})
\int_{\pos_d^{N-1}} K^N_b(\tilde{z}; \diff \tilde{y})
f(\tilde{y},\tilde{z}) \, .
\end{split}
\end{equation}
Using the definitions of $K^N_{b}$ and $\Lambda^N_{a,b}$, we obtain, after some rearrangements and cancellations, that the left-hand side of~\eqref{eq:intertwining2} equals
\[
\begin{split}
&\int_{\pos_d^{N-1}} \mu^{\otimes(N-1)}(\diff y)
\int_{\pos_d^{N-1}} \mu^{\otimes(N-1)}(\diff \tilde{y})
\int_{\pos_d^{N}} \mu^{\otimes N}(\diff \tilde{z})
f(\tilde{y},\tilde{z}) \\
&\prod_{i=1}^{N-1} \left(\abs{\tilde{y}_i}^{-a}
\abs{y_i - \tilde{z}_{i+1}}^{a+b}
\e^{ - \tr\left[y_i (z_i^{-1} + \tilde{y}_i^{-1}) \right] }
\abs{\left(y_i - \tilde{z}_{i+1}\right) y_i^{-1}}^{-\frac{d+1}{2}}
\1_{\pos_d}\left(\tilde{z}_{i+1}^{-1} - y_i^{-1}\right) \right) \\
&\prod_{j=1}^N \left( \abs{z_j}^{-b}
\abs{ \left(\tilde{y}_j + z_j\right) \tilde{z}_j^{-1} }^{a+b}
\e^{-\tr\left[(\tilde{y}_j + z_j)\tilde{z}_j^{-1}\right]} \right) \, ,
\end{split}
\]
with the usual convention $\tilde{y}_N=0$.
By interchanging the order of integration, we see that the latter display equals
\[
\begin{split}
&\int_{\pos_d^{N-1}} \mu^{\otimes(N-1)}(\diff \tilde{y})
\int_{\pos_d^{N}} \mu^{\otimes N}(\diff \tilde{z})
f(\tilde{y},\tilde{z})
\prod_{i=1}^{N-1} \left( \abs{\tilde{y}_i}^{-a}
\mathfrak{J}(z_i,\tilde{y}_i,\tilde{z}_{i+1}) \right) \\
&\prod_{j=1}^N \left(\abs{z_j}^{-b} \abs{\left(\tilde{y}_j + z_j\right) \tilde{z}_j^{-1}}^{a+b}
\e^{-\tr\left[(\tilde{y}_j + z_j)\tilde{z}_j^{-1}\right]} \right) \, ,
\end{split}
\]
where $\mathfrak{J}\colon \pos_d^3 \to\C$ is defined by
\[
\mathfrak{J}(u,v,w)
:= \int_{\pos_d} \mu(\diff s)
\abs{s - w}^{a+b}
\e^{ - \tr\left[ s (u^{-1} + v^{-1}) \right] }
\abs{\left(s - w\right) s^{-1}}^{-\frac{d+1}{2}}
\1_{\pos_d}\left(w^{-1} - s^{-1}\right) \, .
\]
By the properties of positive definite matrices (see~\S~\ref{subsec:notation}), we have that $w^{-1} - s^{-1} \in \pos_d$ if and only if $s - w \in \pos_d$; moreover, for $w\in\pos_d$, the latter condition is stronger than $s\in \pos_d$.
We then make the change of variables $s' := s - w$, which preserves the Lebesgue measure on the `independent' entries of the symmetric matrix $s$, so that
\[
\abs{s}^{\frac{d+1}{2}} \mu(\diff s)
= \abs{s'}^{\frac{d+1}{2}} \mu(\diff s') \, .
\]
Therefore, we have
\[
\mathfrak{J}(u,v,w)
=\int_{\pos_d} \mu(\diff s')
\abs{s'}^{a+b}
\e^{ - \tr\left[ (s' + w) (u^{-1} + v^{-1}) \right] } \, .
\]
After the further, this time $\mu$-preserving, change of variables $s'' := T_{u^{-1} + v^{-1}}(s')$, we obtain
\[
\begin{split}
\mathfrak{J}(u,v,w)
&=\abs{uv \left(u+v\right)^{-1}}^{a+b}
\e^{ - \tr\left[ w (u^{-1} + v^{-1}) \right] }
\int_{\pos_d} \mu(\diff s'') 
\abs{s''}^{a+b} \e^{ - \tr [s''] } \\
&=\abs{uv \left(u+v\right)^{-1}}^{a+b}
\e^{ - \tr\left[ w (u^{-1} + v^{-1}) \right] }
\Gamma_d(a+b) \, ,
\end{split}
\]
where the gamma function is well defined since by hypothesis $\Re(a+b)>(d-1)/2$.
After a few cancellations, we then see that the left-hand side of~\eqref{eq:intertwining2} equals
\[
\begin{split}
\Gamma_d(a+b)^{N-1}
&\int\limits_{\pos_d^{N-1}} \mu^{\otimes(N-1)}(\diff \tilde{y})
\int\limits_{\pos_d^{N}} \mu^{\otimes N}(\diff \tilde{z})
f(\tilde{y},\tilde{z}) \\
&\times \prod_{i=1}^{N-1} \left( \abs{\tilde{y}_i}^b
\e^{-\tr\left[\tilde{z}_{i+1} z_i^{-1} + \tilde{z}_{i+1} \tilde{y}_i^{-1} + \tilde{y}_i \tilde{z}_i^{-1}\right] }
\right)
\prod_{j=1}^N
\left( \abs{z_j}^a \abs{\tilde{z}_j}^{-a-b}
\e^{- \tr\left[z_j \tilde{z}_j^{-1} \right] } \right) \, .
\end{split}
\]
It now follows from the definitions that this equals the right-hand side of~\eqref{eq:intertwining2}, thus concluding the proof.
\end{proof}

A simple inductive argument shows that the intertwining~\eqref{eq:intertwining} can be extended to an intertwining that involves the $\Pi$-kernel~\eqref{eq:kerPi} and the $\tilde{\Sigma}$-kernel~\eqref{eq:kerSigma_tilde}.
From now on, we fix $N\geq 1$, $a\in\C$ and $\lambda=(\lambda_1,\dots,\lambda_N)\in\C^N$ such that $\Re(a+\lambda_i)> (d-1)/2$ for all $i$.
As usual, we also use the notation $a+\lambda := (a+\lambda_1,\dots,a+\lambda_N)$.

\begin{corollary}
\label{coro:intertwining_big}
The intertwining relation
\begin{align}
\label{eq:intertwining_big}
\tilde{\Sigma}^N_{\lambda} \Pi^N_{a+\lambda} =
\left(\prod_{i=1}^N \Gamma_d(a+\lambda_i)^{i-1}\right)
P^N_{a} \tilde{\Sigma}^N_{\lambda}
\end{align}
holds as an equality between {kernels} from $\pos_d^N$ to $\trian{N}{d}$.
\end{corollary}

\begin{proof}
Taking into account~\eqref{eq:kerSigma_tilde}, it is immediate to see that~\eqref{eq:intertwining_big} is equivalent to
\begin{equation}
\label{eq:intertwining_big2}
\int\limits_{\trian{N-1}{d}}
\Sigma^N_{\lambda}(z; \diff x)
\int\limits_{\trian{N-1}{d}\times \pos_d^N} \!\!\!\!\!\! \Pi^N_{a+\lambda}(x,z; \diff \tilde{x} \diff \tilde{z}) f(\tilde{x},\tilde{z})
= \kappa_{a+\lambda}
\int\limits_{\pos_d^N}
P^N_a(z;\diff \tilde{z})
\int\limits_{\trian{N-1}{d}} \Sigma^N_{\lambda}(\tilde{z}; \diff \tilde{x}) f(\tilde{x},\tilde{z})
\end{equation}
for all $z\in \pos_d^N$ and test function $f\colon \trian{N-1}{d}\times \pos_d^N \to\R$, where we set
\[
\kappa_{(\xi_1,\dots,\xi_N)}:=\prod_{i=1}^N \Gamma_d(\xi_i)^{i-1}
\qquad\qquad
\text{if}\quad \Re(\xi_i)>\frac{d-1}{2} \quad\text{for all } i\, .
\]

To prove~\eqref{eq:intertwining_big2}, we proceed by induction.
For $N=1$, \eqref{eq:intertwining_big2} amounts to the identity
\[
\psi^1_{\lambda}(z) \int_{\pos_d} P^1_{a+\lambda}(z;\diff \tilde{z}) f(\tilde{z})
= \int_{\pos_d} P^1_a(z; \diff \tilde{z}) \psi^1_{\lambda}(\tilde{z}) f(\tilde{z})
\]
for $z\in \pos_d$ and $f\colon \pos_d\to\R$.
Using~\eqref{eq:Whittaker} and~\eqref{eq:kerP}, one can easily verify that the latter is true, as both sides equal $\abs{z}^a \int_{\pos_d} \mu(\diff \tilde{z}) \abs{\tilde{z}}^{-a-\lambda} \e^{-\tr[z \tilde{z}^{-1}]} f(\tilde{z})$.

Let now $N\geq 2$ and $\tilde{\lambda} = (\lambda_1,\dots,\lambda_{N-1})$.
Assume by induction that
\begin{equation}
\label{eq:intertwining_big2_induction}
\begin{split}
&\int_{\trian{N-2}{d}}
\Sigma^{N-1}_{\tilde{\lambda}}(y; \diff x)
\int_{\trian{N-2}{d}\times\pos_d^{N-1}} \Pi^{N-1}_{a+\tilde{\lambda}}(x,y; \diff \tilde{x} \diff \tilde{y}) g(\tilde{x},\tilde{y}) \\
= \, & \kappa_{a+\tilde{\lambda}}
\int_{\pos_d^{N-1}}
P^{N-1}_a(y;\diff \tilde{y})
\int_{\trian{N-2}{d}} \Sigma^{N-1}_{\tilde{\lambda}}(\tilde{y}; \diff \tilde{x}) g(\tilde{x},\tilde{y})
\end{split}
\end{equation}
for any $y\in\pos_d^{N-1}$ and any test function $g\colon \trian{N-2}{d}\times\pos_d^{N-1}\to\R$.
Fix $z\in\pos_d^N$ and $f\colon \trian{N-1}{d}\times\pos_d^N \to\R$ (which we view as $f\colon \trian{N-2}{d}\times\pos_d^{N-1}\times\pos_d^N \to\R$).
Choosing 
\[
g(\tilde{x},\tilde{y}):= \int_{\pos_d^N} Q^N_{a+\lambda_N}(y,\tilde{y},z;\diff \tilde{z}) f(\tilde{x},\tilde{y},\tilde{z})
\]
in~\eqref{eq:intertwining_big2_induction} and integrating both sides with respect to the measure $K^N_{\lambda_N}(z; \cdot)$, we obtain
\[
\begin{split}
&\int\limits_{\pos_d^{N-1}} K^N_{\lambda_N}(z; \diff y)
\int\limits_{\trian{N-2}{d}} \Sigma^{N-1}_{\tilde{\lambda}}(y; \diff x) \!\!\!\!\!\!\!
\int\limits_{\trian{N-2}{d}\times\pos_d^{N-1}} \!\!\!\!\!\!\!\!\!\!
\Pi^{N-1}_{a+\tilde{\lambda}}(x,y; \diff \tilde{x} \diff \tilde{y})
\int\limits_{\pos_d^N}
Q^N_{a+\lambda_N}(y,\tilde{y},z;\diff \tilde{z}) f(\tilde{x},\tilde{y},\tilde{z}) \\
= \, &\kappa_{a+\tilde{\lambda}}
\int\limits_{\pos_d^{N-1}}
K^N_{\lambda_N}(z; \diff y) 
\int\limits_{\pos_d^{N-1}} P^{N-1}_a(y;\diff \tilde{y})
\int\limits_{\trian{N-2}{d}} \Sigma^{N-1}_{\tilde{\lambda}}(\tilde{y}; \diff \tilde{x})
\int\limits_{\pos_d^{N}}
Q^N_{a+\lambda_N}(y,\tilde{y},z;\diff \tilde{z})
f(\tilde{x},\tilde{y},\tilde{z}) \, .
\end{split}
\]
Using~\eqref{eq:kerSigma_recursive} and~\eqref{eq:kerPi} for the left-hand side and~\eqref{eq:kerLambda} for the right-hand side, and interchanging the integration order, we then have
\[
\begin{split}
&\int_{\trian{N-2}{d}\times\pos_d^{N-1}}
\Sigma^N_{\lambda}(z; \diff x \diff y)
\int_{\trian{N-2}{d}\times\pos_d^{N-1}\times\pos_d^N}
\Pi^{N}_{a+\lambda}(x,y,z; \diff \tilde{x} \diff \tilde{y} \diff \tilde{z}) f(\tilde{x},\tilde{y},\tilde{z}) \\
= \, &\kappa_{a+\tilde{\lambda}}
\int_{\pos_d^{N-1}}
K^N_{\lambda_N}(z; \diff y)
\int_{\pos_d^{N-1}\times \pos_d^N}
\Lambda^{N}_{a,\lambda_N}(y,z;\diff \tilde{y} \diff\tilde{z})
\left(\int_{\trian{N-2}{d}}
\Sigma^{N-1}_{\tilde{\lambda}}(\tilde{y}; \diff \tilde{x})
f(\tilde{x},\tilde{y},\tilde{z}) \right) \\
= \, &\kappa_{a+\tilde{\lambda}} \,
\Gamma_d(a+\lambda_N)^{N-1}
\int_{\pos_d^N}
P_a^N(z;\diff \tilde{z})
\int_{\pos_d^{N-1}}
K^N_{\lambda_N}(\tilde{z}; \diff \tilde{y})
\left(\int_{\trian{N-2}{d}}
\Sigma^{N-1}_{\tilde{\lambda}}(\tilde{y}; \diff \tilde{x})
f(\tilde{x},\tilde{y},\tilde{z}) \right) \\
= \, &\kappa_{a+\lambda}
\int_{\pos_d^N}
P_a^N(z;\diff \tilde{z})
\int_{\trian{N-2}{d} \times \pos_d^{N-1}}
\Sigma^{N}_{\lambda}(\tilde{z}; \diff \tilde{x} \diff \tilde{y})
f(\tilde{x},\tilde{y},\tilde{z}) \, ,
\end{split}
\]
where the latter two equalities follow from~\eqref{eq:intertwining2} and~\eqref{eq:kerSigma_recursive}, respectively.
The identification $\trian{N-2}{d} \times \pos_d^{N-1} = \trian{N-1}{d}$ concludes the proof of~\eqref{eq:intertwining_big2}.
\end{proof}

Recall now that the $\tilde{\Sigma}$-kernels generate Whittaker functions of matrix arguments, in the sense of~\eqref{eq:Whittaker_triangle}.
By integrating the intertwining relation~\eqref{eq:intertwining_big} and using~\eqref{eq:Pi_normalisation}, we immediately deduce that Whittaker functions are eigenfunctions of the \emph{integral} $P$-operators:
\begin{corollary}
\label{coro:eigenfnEqn}
We have
\begin{equation}
\label{eq:eigenfnEqn}
P^N_a \psi^N_{\lambda} = \left( \prod_{i=1}^N \Gamma_d(a + \lambda_i) \right) \psi^N_{\lambda} \, .
\end{equation}
\end{corollary}

We note that this complements the interpretation of the Whittaker functions $\psi^N_{\lambda}$, given in~\cite{oConnell21}, as eigenfunctions of a \emph{differential} operator, namely the Hamiltonian of a quantisation in $\pos_d^N$ of the $N$-particle non-Abelian Toda chain.

For $x\in \trian{N}{d}$ and $z, \tilde{z}\in \pos_d^N$, we now define
\begin{align}
\label{eq:kerSigma_normal}
\overline{\Sigma}^N_{\lambda}(z; \diff x)
&:= \frac{1}{\psi^{N}_{\lambda}(z)}
\tilde{\Sigma}^N_{\lambda}(z; \diff x) \, , \\
\label{eq:kerP_doob}
\bm{P}^N_{a,\lambda}(z; \diff \tilde{z})
&:= \frac{1}{\prod_{i=1}^N \Gamma_d(a+\lambda_i)}
\frac{\psi^N_{\lambda}(\tilde{z})}{\psi^N_{\lambda}(z)}
P^N_{a}(z; \diff \tilde{z}) \, .
\end{align}
It follows from~\eqref{eq:Whittaker_triangle} and~\eqref{eq:eigenfnEqn} that the above kernels are normalised; therefore, they are Markov kernels when the parameters $a,\lambda_1,\dots,\lambda_N$ are real.
Notice that~\eqref{eq:kerP_doob} may be seen as a Doob $\mathit{h}$-transform of the $P$-kernel~\eqref{eq:kerP}.
It is now immediate to deduce a renormalised version of~\eqref{eq:intertwining_big}:
\begin{corollary}
\label{coro:intertwining_normal}
The intertwining relation
\begin{align}
\label{eq:intertwining_normal}
\overline{\Sigma}^N_{\lambda} \overline{\Pi}^N_{a+\lambda}
=\bm{P}^N_{a,\lambda} \overline{\Sigma}^N_{\lambda}
\end{align}
holds as an equality between {kernels} from $\pos_d^N$ to $\trian{N}{d}$.
\end{corollary}

{From a probabilistic point of view, \eqref{eq:intertwining_normal} states that, for any fixed $z\in\pos_d^N$, {the two following update rules are equivalent}: (i) starting the process $X$ from a (random) initial configuration dictated by the intertwining kernel $\overline{\Sigma}(z; \cdot)$ and letting it evolve according to the dynamic $\overline{\Pi}$; and (ii) running the dynamic $\bm{P}$ on the bottom edge (started at $z$) and then updating the whole triangular array according to the intertwining kernel $\overline{\Sigma}$.
The main result of this section is a precise account of this interpretation.}

\begin{theorem}
\label{thm:bottomRowMarkov}
Let $X=(X(n))_{n\geq 0}$ be the Markov process on $\trian{N}{d}$ as in Definition~\ref{def:triangularProcess}.
Assume that, for an arbitrary $z\in \pos_d^N$, the initial state $X(0)$ of $X$ is distributed according to the measure $\overline{\Sigma}^N_{\beta}(z;\cdot)$.
Then, the $N$-th row $X^N = (X^N(n))_{n\geq 0}$ is a time-inhomogeneous Markov process (in its own filtration) on the state space $\pos_d^N$, with initial state $z$ and time-$n$ transition kernel $\bm{P}^N_{\alpha(n), \beta}$.
Moreover, for any bounded measurable function $f\colon \trian{N}{d} \to \R$ and $n\geq 0$, we have
\begin{equation}
\label{eq:wholeArrayGivenBottomRow}
\E\left[ f(X(n)) \;\middle|\; X^N(0), \dots, X^N(n-1), X^N(n) \right]
= \overline{\Sigma}^N_{\beta}f\left(X^N(n)\right) \qquad \text{a.s.}
\end{equation}
\end{theorem}

\begin{proof}
The statement is an application of Theorem~\ref{thm:MarkovFns}, where the state spaces are $S = \trian{N}{d}$ and $T=\pos_d^N$, and the function $\phi\colon \trian{N}{d} \to \pos_d^N$ is the projection $\phi(x) := x^N$ onto the $N$-th row of $x$, so that $X^N(n)= \phi(X(n))$.
Hypothesis~\ref{it:MarkovFns_initial} of Theorem~\ref{thm:MarkovFns}, i.e.\ the fact that $\overline{\Sigma}_{\beta}(z; \phi^{-1}\{z\})=1$ for any $z\in \pos_d^N$, holds because, by definition, the measure $\overline{\Sigma}_{\beta}(z;\cdot)$ is supported on the set $\trian{N}{d}(z)$ of height-$N$ triangular arrays with $N$-th row equal to $z$.
On the other hand, by Prop.~\ref{prop:triangle_Markov}, the time-$n$ transition kernel of $X$ is $\overline{\Pi}^N_{\alpha(n)+\beta}$.
Therefore, in this case, hypothesis~\ref{it:MarkovFns_intertw} of Theorem~\ref{thm:MarkovFns} reads as the set of intertwining relations
\[
\overline{\Sigma}^N_{\beta} \overline{\Pi}^N_{\alpha(n)+\beta}
=\bm{P}^N_{\alpha(n),\beta} \overline{\Sigma}^N_{\beta} \qquad \text{for all } n\geq 1 \, .
\]
These follow from Corollary~\ref{coro:intertwining_normal}.
\end{proof}

\begin{remark}
By letting $N$ vary, it is immediate to deduce from Theorem~\ref{thm:bottomRowMarkov} that \emph{every} row of $X$ evolves as a Markov process in its own filtration, under an appropriate (random) initial configuration on the previous rows.
Therefore, the focus on the $N$-th row should only be seen as a convenient choice.
\end{remark}

\subsection{Feynman--Kac interpretation}
\label{subsec:feynman-kac}

Here we provide a Feynman--Kac type interpretation of Whittaker functions based on the eigenfunction equation~\eqref{eq:eigenfnEqn}.
Our result should be compared to the one obtained in~\cite[Prop.~9]{oConnell21} in the continuous setting of Brownian particles.

\begin{definition}
\label{def:RW_FK}
Let $\lambda\in\R^N$ with
\[
\min\left( \lambda_1, \lambda_2-\lambda_1, \dots, \lambda_N-\lambda_{N-1}\right) > \frac{d-1}{2} \, .
\]
Let $y\in\pos_d^N$.
We define $Y=(Y(n))_{n\geq 0}=(Y_1(n),\dots,Y_N(n))_{n\geq 0}$ to be a process in $\pos_d^N$ with independent components, such that each component $Y_i=(Y_i(n))_{n\geq 0}$ is a $\GL_d$-invariant random walk on $\pos_d$ with initial state $Y_i(0)=y_i$ and inverse Wishart increments with parameter $\lambda_i$.
\end{definition}

Recalling~\eqref{eq:kerP_normal_N=1}, $Y$ is then a time-homogeneous Markov process starting at $y$ with transition kernel
\[
\Theta^N_{\lambda}(z;\diff\tilde{z}) :=
\overline{P}^1_{\lambda_1}(z_1;\diff\tilde{z}_1) \cdots 
\overline{P}^1_{\lambda_N}(z_N;\diff\tilde{z}_N)
\qquad\qquad
\text{for} \; z,\tilde{z}\in\pos_d^N \, .
\]
For $z,\tilde{z}\in\pos_d^N$, define the sub-Markov kernel
\begin{equation}
\label{eq:kerThetaTilde}
\hat{\Theta}^N_{\lambda}(z;\tilde{z})
:= \e^{-V(z;\tilde{z})} \Theta^N_{\lambda}(z;\tilde{z}) \, ,
\end{equation}
where $V$ is the `killing potential'
\begin{equation}
\label{eq:fnV}
V(z;\tilde{z}) := \sum_{i=1}^{N-1} \tr\big[\tilde{z}_{i+1} z_i^{-1}\big] \, .
\end{equation}

Denote by $\P_{y}$ and $\E_{y}$ the probability and expectation, respectively, with respect to the law of $Y$ with initial state $y$.

\begin{theorem}
\label{thm:feynmanKac}
For all $y\in\pos_d^N$, we have
\begin{equation}
\label{eq:feynmanKac}
\psi^N_\lambda(y)
= \prod_{1\leq i<j\leq N} \Gamma_d(\lambda_j-\lambda_i)
\left(\prod_{i=1}^N \abs{y_i}^{-\lambda_i}\right)
\E_y \Big[ \e^{-
\sum_{n=0}^\infty V(Y(n);Y(n+1))} \Big] \, .
\end{equation}
\end{theorem}

The main purpose of this subsection is to prove~\eqref{eq:feynmanKac}.
In a nutshell, using a fairly standard martingale argument, we will show that the expectation in~\eqref{eq:feynmanKac} is the unique solution to an eigenproblem; the latter is also, essentially, solved by Whittaker functions.

\begin{lemma}
\label{lemma:limitV}
Fix an integer $\ell\ge0$. For any $y\in\pos_d^N$, we have
\[
\limsup_{n\to\infty} \frac{1}{n} \log V(Y(n);Y(n+\ell))
<  0
\qquad \P_y\text{-a.s.}
\]
\end{lemma}

\begin{remark}
In particular Lemma~\ref{lemma:limitV} with $\ell=1$ implies that the infinite series inside the expectation in~\eqref{eq:feynmanKac} converges $\P_y$-a.s.
\end{remark}

\begin{proof}[Proof of Lemma~\ref{lemma:limitV}]
Since
\[
V(Y(n);Y(n+\ell))
= \sum_{i=1}^{N-1} \tr\big[ Y_{i+1}(n+\ell) Y_i(n)^{-1}\big] \, ,
\]
it suffices to show that, for each $1\le i\le N-1$,
\[
\limsup_{n\to\infty} \frac{1}{n} \log\tr\big[ Y_{i+1}(n+\ell) Y_i(n)^{-1}\big]
< 0 
\qquad \P_y\text{-a.s.}
\]
Let us record the following properties, which hold for any $a,b\in\pos_d$:
\begin{itemize}
\item $\tr[ab]\leq\tr[a]\tr[b]$ (submultiplicativity of the trace);
\item $\tr[a]\leq d\, \maxeig(a)$;
\item $\maxeig(a^{-1}) = \mineig(a)^{-1}$.
\end{itemize}
Here, $\maxeig$ and $\mineig$ denote the maximum and minimum eigenvalue, respectively.
Using these facts, we have, for $1\leq i\leq N-1$:
\[
\tr\big[ Y_{i+1}(n+\ell) Y_i(n)^{-1}\big]
\leq \tr\big[ Y_{i+1}(n+\ell)\big] \tr\big[Y_i(n)^{-1}\big]
\leq d^2 \frac{\maxeig(Y_{i+1}(n+\ell))}{\mineig(Y_{i}(n))} \, .
\]
Now, using for example~\cite[Corollary~B.4]{aristaBisiOConnell23}, we have
\[
\lim_{n\to\infty} \frac{1}{n}\log\maxeig(Y_{i+1}(n)) = -\psi\left(\lambda_{i+1}-\frac{d-1}{2}\right) ,
\quad\,\,
\lim_{n\to\infty} \frac{1}{n}\log\mineig(Y_i(n)) = -\psi(\lambda_i) \, ,
\]
$\P_y$-a.s., where $\psi$ is the digamma function. These are the maximum (respectively, minimum) Lyapunov exponent of a $\GL_d$-invariant random walk with inverse Wishart increments of parameter $\lambda_{i+1}$ (respectively, $\lambda_i$).
We then obtain
\[
\limsup_{n\to\infty} \frac{1}{n} \log\tr\big[ Y_{i+1}(n+\ell) Y_i(n)^{-1}\big]
\leq \psi(\lambda_i)-\psi\left(\lambda_{i+1}-\frac{d-1}{2}\right) <0 \, ,
\]
since the digamma function is strictly increasing and, by Definition~\ref{def:RW_FK}, $\lambda_{i+1}-\lambda_i>(d-1)/2$.
\end{proof}

\begin{lemma}
\label{lemma:martingaleArgument}
Let $u\colon \pos_d^N\to\R$ be a measurable function  such that
\begin{enumerate}
\item \label{it:mart_eigenfunctionEqn}
$\hat{\Theta}^N_{\lambda} u = u$ (\emph{eigenfunction equation});
\item \label{it:mart_bounded}
$u$ is bounded (\emph{boundedness property});
\item \label{it:mart_limit}
$u(y) \to 1$ as $V(y;y)\to 0$ (\emph{boundary condition}).
\end{enumerate}
Then, for all $y\in\pos_d^N$,
\[
u(y) = \E_y \Big[ \e^{-
\sum_{n=0}^\infty V(Y(n);Y(n+1))} \Big] \, .
\]
\end{lemma}
\begin{proof}
Consider the process $Y$ as in Definition~\ref{def:RW_FK}, with initial state $y\in\pos_d^N$ and transition kernel $\Theta^N_{\lambda}$.
Denote by $(\mathcal{F}(k))_{k\geq 0}$ its natural filtration.
It follows from the eigenfunction equation that
\[
\begin{split}
\E_y\left[ u(Y(k+1)) \e^{- V(Y(k); Y(k+1))} \;\middle|\; \mathcal{F}(k)\right]
&= \int_{\pos_d^N} \Theta^N_{\lambda}(Y(k);\diff \tilde{z}) \e^{- V(Y(k); \tilde{z})} u(\tilde{z}) \\
&= \hat{\Theta}^N_{\lambda} u(Y(k))
= u(Y(k)) \, .
\end{split}
\]
Therefore, the process $M=(M(k))_{k\geq 0}$ defined by
\begin{equation}
\label{eq:martingale}
M(k) :=
\begin{cases}
u(Y(0)){=u(y)} &k=0 \, , \\
u(Y(k)) \e^{-\sum_{n=0}^{k-1} V(Y(n); Y(n+1))} &k\geq 1 
\end{cases}
\end{equation}
is an $(\mathcal{F}(k))_{k\geq 0}$-martingale.
By the boundedness property, $M$ is uniformly bounded and, thus, a uniformly integrable martingale.
Therefore, $M$ converges $\P_y$-a.s.\ and in $1$-norm to a certain limit $M(\infty)$ and, for all $k\ge0$, we have $M(k)=\E_{y}\left[M(\infty) \;\middle|\; \mathcal{F}(k)\right]$.
By Lemma~\ref{lemma:limitV} (with $\ell=0$), we have $\lim_{k\to\infty}V(Y(k);Y(k))= 0$, $\P_y$-a.s.
The boundary condition then implies $\lim_{k\to\infty} u(Y(k)) =1$, $\P_y$-a.s., whence 
\[
M(\infty)=\e^{-\sum_{n=0}^{\infty} V(Y(n); Y(n+1))} \, .
\]
We conclude that, for any $y\in\pos_d^N$,
\[
u(y) = {M(0)} = \E_y[M(\infty)]
= \E_y \Big[ \e^{-
\sum_{n=0}^\infty V(Y(n);Y(n+1))} \Big] \, .
\qedhere
\]
\end{proof}

\begin{proof}[Proof of Theorem~\ref{thm:feynmanKac}]
It was proven in~\cite[proof of Prop.~9]{oConnell21} that the function
\[
v(y):= \psi^N_\lambda(y)
\prod_{i=1}^N \abs{y_i}^{\lambda_i} \, , \qquad
y\in\pos_d^N \, ,
\]
is bounded and satisfies
\[
\lim_{V(y;y)\to 0} v(y) = \prod_{1\leq i<j\leq N} \Gamma_d(\lambda_j-\lambda_i) \, .
\]
By Lemma~\ref{lemma:martingaleArgument}, it then remains to prove that $\hat{\Theta}^N_{\lambda} v = v$.
It follows from the definition~\eqref{eq:kerP} of the kernel $P^N_a$ that
\[
\hat{\Theta}^N_{\lambda} v(z) =
\left(\prod_{i=1}^N \frac{\abs{z_i}^{\lambda_i}}{\Gamma_d(\lambda_i)}\right) P^N_0 \psi^N_{\lambda}(z)
\]
for $z\in\pos_d^N$.
Using the eigenfunction equation~\eqref{eq:eigenfnEqn}, we see that the right-hand side above equals $v(z)$, as desired.
\end{proof}

\begin{corollary}
Under $\P_y$, we have the distributional equality
\begin{equation}
\label{eq:dufresne}
\sum_{n=0}^{\infty} \tr\left[Y_2(n+1) Y_1(n)^{-1}\right]
\stackrel{\text{d}}{=} \tr\left[a Z\right] \, ,
\end{equation}
where $a:= y_1^{-1}y_2y_1^{-1/2}$ and $Z$ has the inverse Wishart distribution of parameter $\lambda_2-\lambda_1$.
\end{corollary}
\begin{proof}
We may assume that $N=2$, so that $Y=(Y_1,Y_2)$ starts at $y=(y_1,y_2)$.
Using Theorem~\ref{thm:feynmanKac} and the definition of Whittaker functions, we compute the Laplace transform of the left-hand side of~\eqref{eq:dufresne} as
\[
\begin{split}
\E_{(y_1,y_2)}&\left[ \e^{-s\sum_{n=0}^{\infty} \tr\left[Y_2(n+1) Y_1(n)^{-1}\right]} \right]
= \E_{(y_1,sy_2)}\left[ \e^{-\sum_{n=0}^{\infty} \tr\left[Y_2(n+1) Y_1(n)^{-1}\right]} \right] \\
&= \frac{\abs{y_1}^{\lambda_1} \abs{sy_2}^{\lambda_2}}{\Gamma_d(\lambda_2-\lambda_1)}
\int_{\pos_d} \mu(\diff x) \abs{x}^{-\lambda_1} \left(\frac{\abs{sy_1y_2}}{\abs{x}}\right)^{-\lambda_2}
\e^{-\tr[sy_2x^{-1} + xy_1^{-1}]} \\
&= \int_{\pos_d} \mu(\diff z) \e^{-s \tr[y_1^{-1/2}y_2y_1^{-1/2}z]} \frac{\abs{z}^{-(\lambda_2-\lambda_1)} \e^{-\tr[z^{-1}]}}{\Gamma_d(\lambda_2-\lambda_1)}
\end{split}
\]
for $s\in\R$, where we used the change of variables $z=y_1^{1/2}x^{-1}y_1^{1/2}$.
The last integral equals $\E\e^{-s \tr[aZ]}$, where $Z$ is inverse Wishart of parameter $\lambda_2-\lambda_1$.
We conclude that the two sides of~\eqref{eq:dufresne} have the same Laplace transform and, hence, the same law.
\end{proof}

\begin{remark}
{Up to some technical details, identity~\eqref{eq:dufresne} may be also deduced from} the Dufresne type identity for a random walk on $\pos_d$ proved in~\cite{aristaBisiOConnell23}.
Let $(R(n))_{n\geq 0}$ be a $\GL_d$-invariant random walk on $\pos_d$ whose initial state $R(0)$ is an inverse Wishart matrix with parameter $\lambda_2$ and whose increments are Beta type II matrices with parameters $\lambda_1$ and $\lambda_2$ (see~\cite{aristaBisiOConnell23} for more details).
It is then natural to expect that the eigenvalue processes of the two processes $(Y_1(n)^{-1/2}Y_2(n+1)Y_1(n)^{-1/2})_{n\geq 0}$ and $(a^{1/2}R(n)a^{1/2})_{n\geq 0}$, where $a=y_1^{-1}y_2y_1^{-1/2}$, have the same law; this is certainly true at least in the case $d=1$.
By summing the traces of these two processes over all $n\geq 0$, \cite[{Theorem~4.10}]{aristaBisiOConnell23} would then immediately provide a proof of~\eqref{eq:dufresne} that does \emph{not} rely upon the Feynman--Kac formula~\eqref{eq:feynmanKac}.
{See~\cite[Lemma~8]{oConnell21} for an analogous argument in the Brownian setting.}
\end{remark}

\section{Fixed-time laws and {matrix Whittaker measures}}
\label{sec:WhittakerMeasures}

In this section, we first prove a Whittaker integral identity that allows us to introduce {matrix Whittaker measures}.
We then obtain an asymptotic formula for a Whittaker function whose arguments go to zero or infinity in norm.
Using the latter result, we next show that, for a certain singular initial state, {matrix Whittaker measures} appear naturally as the fixed-time laws of the bottom edge of the triangular process $X$ introduced in~\S~\ref{subsec:MarkovDynamics}.
Finally, under the same singular initial state, we study the fixed-time law of the right edge {and of the left edge} of $X$.

\subsection{{Matrix Whittaker measures}}
\label{subsec:Stade_matrix}
Whittaker functions of matrix arguments satisfy a remarkable integral identity:
\begin{theorem}
\label{thm:stade_matrix}
Let $n\geq N \geq 1$.
Let $\lambda=(\lambda_1,\dots,\lambda_n)\in\C^n$ and $\rho=(\rho_1,\dots,\rho_N)\in \C^N$ such that $\Re(\lambda_\ell + \rho_i) > \frac{d-1}{2}$ for all $1\leq \ell\leq n$, $1\leq i\leq N$.
Let $s\in \pos_d$.
Then,
\begin{equation}
\label{eq:stade_matrix}
\int_{\pos_d^N}
\mu^{\otimes N}(\diff z)
\psi^{N,n}_{\lambda;s}(z)
\psi^N_{\rho}(z)
= \abs{s}^{-\sum_{j=1}^N (\lambda_j + \rho_j)}
\prod_{\ell=1}^n \prod_{i=1}^N \Gamma_d(\lambda_\ell + \rho_i) \, .
\end{equation}
\end{theorem}
The case $N=n$ of~\eqref{eq:stade_matrix} was noted in~\cite[Prop.~10]{oConnell21}; however, the identity did not play a key role in that article, and the details of the proof were not provided therein.
Below we provide a proof of the general case $n\ge N$ that involves the generalised Whittaker functions introduced in~\S~\ref{subsec:Whittaker_gen}.

In the scalar $d=1$ setting, \eqref{eq:stade_matrix} goes back to~\cite[Corollary~3.5]{oConnellSeppalainenZygouras14}.
For $d=1$ and $N=n$, it is equivalent to an identity that was originally found in the number theoretic literature~\cite{bump89, stade02}.

Theorem~\ref{thm:stade_matrix} can be also seen as an analogue, in the context of Whittaker functions, of the celebrated Cauchy-Littlewood identity for Schur functions.
In the literature on symmetric functions, the latter is usually proved using either the determinantal structure of Schur functions (see~\cite[I-(4.3)]{macdonald79}) or the Robinson--Schensted--Knuth correspondence, a combinatorial bijection (see~\cite[Theorem~7.12.1]{stanley99}).
None of these tools is available, so far, in our context.
To prove~\eqref{eq:stade_matrix}, we will rather proceed inductively, using the recursive definition of Whittaker functions and the eigenfunction equation~\eqref{eq:eigenfnEqn}.
For the reader's convenience, we also include in Appendix~\ref{app:cauchyLittlewood} a proof of the classical Cauchy-Littlewood identity that similarly relies on an eigenfunction equation for Schur functions (which can be seen as a version of the so-called Pieri rule).

\begin{proof}[Proof of Theorem~\ref{thm:stade_matrix}]
{We will prove~\eqref{eq:stade_matrix} {by induction} on $n$.
For a fixed integer $n\ge1$, let $\mathrm{S}(n)$ be the statement that~\eqref{eq:stade_matrix} holds for all $N$ such that $n\geq N\geq 1$ and for any choice of $\lambda$ and $\rho$ satisfying the assumptions of the theorem.}

{For $n=N=1$ we have}
\[
\begin{split}
&\int_{\pos_d} \mu(\diff z)
\psi^{1,1}_{\lambda;s}(z)
\psi^1_{\rho}(z)
= \int_{\pos_d} \mu(\diff z) \e^{-\tr[sz^{-1}]} \abs{z}^{-\lambda_1} \abs{z}^{-\rho_1} \\
= \, &\abs{s}^{-\lambda_1-\rho_1}
\int_{\pos_d} \mu(\diff \tilde{z})
\abs{\tilde{z}}^{-(\lambda_1+\rho_1)}
\e^{-\tr[\tilde{z}^{-1}]}
= \abs{s}^{-\lambda_1-\rho_1}
\Gamma_d(\lambda_1+\rho_1) \, ,
\end{split}
\]
where we have used the definitions of Whittaker functions and gamma functions and the $\mu$-preserving change of variables $\tilde{z} = T_{s^{-1}}(z)$.
{This proves {the base case} $\mathrm{S}(1)$.}

{Suppose now {by induction} that $\mathrm{S}(n-1)$ holds for some {fixed} $n\ge 2$.
To prove that $\mathrm{S}(n)$ holds, let us first prove that~\eqref{eq:stade_matrix} is valid for all $N$ such that $n>N\ge1$.}
It follows from~\eqref{eq:WhittakerExt}, Fubini's theorem, \eqref{eq:kerP_inverse} and~\eqref{eq:eigenfnEqn}, that
\[
\begin{split}
&\int_{\pos_d^N}
\mu^{\otimes N}(\diff z)
\psi^{N,n}_{\lambda;s}(z)
\psi^N_{\rho}(z)
= \int_{\pos_d^N}
\mu^{\otimes N}(\diff z)
\big(\cev{P}^N_{\lambda_n} \psi^{N,n-1}_{(\lambda_1,\dots,\lambda_{n-1});s}\big)(z)
\psi^N_{\rho}(z) \\
= \, &\int_{\pos_d^N} \mu^{\otimes N}(\diff \tilde{z})
\psi^{N,n-1}_{(\lambda_1,\dots,\lambda_{n-1});s}(\tilde{z})
\big(P^N_{\lambda_n} \psi^N_{\rho}\big)(\tilde{z}) \\
= \, & \prod_{i=1}^N \Gamma_d(\lambda_n + \rho_i)
\int_{\pos_d^N}
\mu^{\otimes N}(\diff \tilde{z})
\psi^{N,n-1}_{(\lambda_1,\dots,\lambda_{n-1});s}(\tilde{z})
\psi^N_{\rho}(\tilde{z}) \, .
\end{split}
\]
{Since $n-1\ge N$, using the assumption $\mathrm{S}(n-1)$ in the latter integral we obtain~\eqref{eq:stade_matrix}.}

{{To conclude $\mathrm{S}(n)$, we are} left to prove the case $N=n$.}
Using~\eqref{eq:WhittakerExt}, \eqref{eq:Whittaker}, Fubini's theorem, \eqref{eq:kerP-K}, \eqref{eq:eigenfnEqn}, and~\eqref{eq:Whittaker_triangle&trapezoid}, we have
\[
\begin{split}
&\int_{\pos_d^n} \mu^{\otimes n}(\diff z)
\psi^{n,n}_{\lambda;s}(z)
\psi^n_{\rho}(z) \\
= \, &\int_{\pos_d^n} \mu^{\otimes n}(\diff z)
\e^{-\tr\left[sz_n^{-1}\right]}
\big(K^n_{\lambda_n} \psi^{n-1}_{(\lambda_1,\dots,\lambda_{n-1})}\big)(z)
\psi^n_{\rho}(z) \\
= \, &\abs{s}^{-\lambda_n}
\int_{\pos_d^{n-1}} \mu^{\otimes (n-1)}(\diff y)
\psi^{n-1}_{(\lambda_1,\dots,\lambda_{n-1})}(y)
\big(P^n_{\lambda_n}\psi^n_{\rho}\big)(y_1,\dots,{y_{n-1}},s) \\
= \, &\abs{s}^{-\lambda_n}
\left(\prod_{i=1}^n \Gamma_d(\lambda_n + \rho_i)\right)
\int_{\pos_d^{n-1}} \mu^{\otimes (n-1)}(\diff y)
\psi^{n-1}_{(\lambda_1,\dots,\lambda_{n-1})}(y)
\psi^n_{\rho}(y_1,\dots,y_{n-1},s) \\
= \, &\abs{s}^{-\lambda_n-\rho_n}
\left(\prod_{i=1}^n \Gamma_d(\lambda_n + \rho_i)\right)
\int_{\pos_d^{n-1}} \mu^{\otimes (n-1)}(\diff y)
\psi^{n-1}_{(\lambda_1,\dots,\lambda_{n-1})}(y)
\psi^{n-1,n}_{\rho;s}(y) \, .
\end{split}
\]
{Recall that we have already proved~\eqref{eq:stade_matrix} for all $N$ such that $n>N\geq 1$.
Applying this, for $N=n-1$, to the latter integral, we conclude that~\eqref{eq:stade_matrix} holds also for $N=n$.}
\qedhere
\end{proof}

\begin{definition}
\label{def:WhittakerMeasure}
For $n\geq N\geq 1$.
Let $\lambda=(\lambda_1,\dots,\lambda_n)\in\R^n$ and $\rho=(\rho_1,\dots,\rho_N)\in \R^N$ such that $\lambda_\ell + \rho_i > \frac{d-1}{2}$ for all $1\leq \ell\leq n$, $1\leq i\leq N$.
We call \emph{{matrix Whittaker measure} with parameters $\lambda$ and $\rho$} the measure on $\pos_d^N$ that is absolutely continuous with respect to $\mu^{\otimes N}(\diff z)$ with density
\begin{equation}
\label{eq:WhittakerMeasure}
W^{N,n}_{\lambda,\rho}(z)
:= \left(\prod_{\ell=1}^n \prod_{i=1}^N \frac{1}{\Gamma_d(\lambda_\ell + \rho_i)}\right)
\psi^{N,n}_{\lambda;I_d}(z)
\psi^N_{\rho}(z) \qquad
\text{for } z\in\pos_d^N \, ,
\end{equation}
where $I_d$ is the $d\times d$ identity matrix.
According to the usual convention, we also denote by $W^{N,n}_{\lambda,\rho}(\diff z)$ the measure itself.
\end{definition}

By Theorem~\ref{thm:stade_matrix}, \eqref{eq:WhittakerMeasure} defines a probability distribution on $\pos_d^N$.
This extends the definition of {matrix Whittaker measures} given in~\cite[\S~7.4]{oConnell21}, which corresponds to the case $n=N$:
\begin{equation}
\label{eq:WhittakerMeasure_N=n}
W^{N,N}_{\lambda, \rho}(z)
= \left(\prod_{\ell,i=1}^N \frac{1}{\Gamma_d(\lambda_\ell + \rho_i)}\right)
\e^{-\tr[z_N^{-1}]}
\psi^N_{\lambda}(z)
\psi^N_{\rho}(z) \, .
\end{equation}

\subsection{Asymptotics of Whittaker functions}
\label{subsec:asymptoticsWhittaker}

For any real $k>0$, let
\begin{equation}
\label{eq:specialInitialCond}
r^i_j(k) := k^{2j-i-1}I_d \qquad\qquad
\text{for } 1\leq j\leq i
\end{equation}
and let $r^i(k) := (r^i_1(k), \dots, r^i_i(k))$.
Our ultimate goal is to obtain the $k\to\infty$ leading order approximation of the Whittaker function $\psi_{\lambda}^N(r^N(k))$. 

We rely on some results (Theorem~\ref{thm:minimiser&Hessian} and Prop.~\ref{prop:saddlePoint}) that we will prove, in a more general setting, in~\S~\ref{sec:minimisation}.
With this purpose in mind, we use the graphical representations of the set of height-$N$ triangular arrays $\trian{N}{d}$ and of the energy function $\Phi^N$, both involved in the definition of the Whittaker function \eqref{eq:WhittakerClosedForm} (see Fig.~\ref{fig:energyWhittaker}).
Given $N\geq 2$, we set
\[
\bm{V}:= \{(i,j)\in \Z^2\colon 1\leq j\leq i\leq N\}
\]
and consider the finite graph $\bm{G}=(\bm{V},\bm{E})$, where $\bm{E}$ consists of all (directed) edges $(i,j)\to (i+1,j)$ and $(i+1,j+1)\to (i,j)$, for $1\leq j\leq i\leq N-1$.
Then, $\trian{N}{d}$ may be identified as the set $\pos_d^{\bm{V}}$ of arrays $x=(x_v)_{v\in \bm{V}}$, where each $x_v\in\pos_d$.
Let also
\[
\bm{\Gamma}:=\{(N,j)\colon 1\leq j\leq N\} \, .
\]
We may thus identify $z\in\pos_d^N$ with $z\in\pos_d^{\bm{\Gamma}}$, so that the set $\trian{N}{d}(z)$ of all height-$N$ triangular arrays whose $N$-th row equals $z$ coincides with the set $\pos_d^{\bm{V}}(z)$, according to the notation~\eqref{eq:cylinder}.
Furthermore, the energy function~\eqref{eq:energyFn_Whittaker} can be equivalently rewritten as
\[
\Phi^N(x)
= \sum_{\substack{(i,j), (k,\ell) \in \bm{V} \colon \\ (i,j) \to (k,\ell)}} \tr\left[x^i_j (x^k_{\ell})^{-1}\right]
= \sum_{\substack{v,w \in \bm{V} \colon \\ v \to w}} \tr[x_v x_w^{-1}]
\qquad\qquad \text{for all } x\in\pos_d^{\bm{V}}.
\]

All the results of~\S~\ref{sec:minimisation} hold for the above `triangular graph' structure, since:
\begin{itemize}
\item $\bm{G}=(\bm{V},\bm{E})$ is an {\it acyclic finite directed graph};
\item $\bm{\Gamma}$ is a proper subset of $\bm{V}$ containing the only {\it source} $(N,N)$ and {\it sink} $(N,1)$ of $\bm{G}$;
\item the energy function $\Phi^N$ is of the form~\eqref{eq:energyFn_Phi}.
\end{itemize}

We first prove a property of the critical points of $\Phi^N$ that, in the scalar $d=1$ setting, was observed in~\cite{oConnell12}.
\begin{lemma}
\label{lemma:detProdRows}
Let $z\in \pos_d^N$.
Let $x$ be any critical point of $\Phi^N$ on $\trian{N}{d}(z)$.
For all $1\leq i\leq N$, let $p_i := \big\lvert x^i_1 \cdots x^i_i \big\rvert$ be the determinant of the product of the $i$-th row of $x$.
Then,
\begin{equation}
\label{eq:detProdRows}
p_1 = \sqrt[2]{p_2} = \dots =
\sqrt[N-1]{p_{N-1}} = \sqrt[N]{p_N}
= \sqrt[N]{\abs{z_1\cdots z_N}} \, .
\end{equation}
\end{lemma}
\begin{proof}
The critical point equations of the energy function $\Phi^N$ are
\[
(x^i_j)^{-1} (x^{i+1}_{j+1} + x^{i-1}_j) (x^i_j)^{-1}
= (x^{i-1}_{j-1})^{-1} + (x^{i+1}_j)^{-1}
\qquad\qquad
\text{for all }
1\leq j\leq i< N \, ,
\]
with the convention $x^{i-1}_i=(x^{i-1}_0)^{-1}=0$ for all $1\leq i<N$ (these correspond to~\eqref{eq:critPointsEqns} in the case of the triangular graph $\bm{G}$).
Taking determinants of both sides, we obtain
\[
\big\lvert x^i_j\big\rvert^2
= \frac{\abs{x^{i+1}_{j+1} + x^{i-1}_j}}{ \abs{(x^{i-1}_{j-1})^{-1} + (x^{i+1}_j)^{-1}}}
\qquad\qquad
\text{for all }
1\leq j\leq i< N \, .
\]
Taking the product over $j$ in the latter, many terms cancel out, yielding
\[
\prod_{j=1}^i \big\lvert x^i_j\big\rvert^2
= \prod_{j=1}^{i-1} \big\lvert x^{i-1}_j\big\rvert
\prod_{j=1}^{i+1} \big\lvert x^{i+1}_j\big\rvert \, .
\]
By definition of $p_1,\dots, p_N$, the latter can be written as
\begin{equation}
\label{eq:detProdRows_proof}
p_i^2 = p_{i-1} p_{i+1} 
\qquad\qquad
\text{for all }
1\leq j\leq i< N \, ,
\end{equation}
with the convention $p_0:=1$.
Finally, it is straightforward to see that equations~\eqref{eq:detProdRows_proof} are equivalent to~\eqref{eq:detProdRows}.
\end{proof}

Let now
\[
I_d^N := \underbrace{(I_d,\dots,I_d)}_{\text{$N$ times}}
= r^N(1) \in \pos_d^N \, .
\]
As the components of $I_d^N$ are scalar matrices, Theorem~\ref{thm:minimiser&Hessian} implies:
\begin{corollary}
\label{coro:uniqueMinimiser}
The function $\Phi^N$ on $\trian{N}{d}(I_d^N)$ has a {unique global minimiser}, at which the Hessian is positive definite.
Moreover, each component $m^i_j$ {of the minimiser $m=(m^i_j)_{1\leq j\leq i\leq N}$} is a positive scalar matrix.
\end{corollary}
Throughout this subsection, $m$ will always denote the above minimiser.

\begin{corollary}
\label{coro:rowProdMinimiser}
We have $m^1_1=I_d$ and
\begin{equation}
\label{eq:detProdRows_m}
\big\lvert m^i_1 \cdots m^i_i \big\rvert = 1
\qquad\qquad \text{for all }
i=1,\dots,N \, .
\end{equation}
\end{corollary}
\begin{proof}
Since $m\in \trian{N}{d}(I_d^N)$, we have $m^N_j = I_d$ for all $j=1,\dots,N$, hence $\big\lvert m^N_1 \cdots m^N_N \big\rvert = 1$.
On the other hand, as a minimiser, $m$ is a critical point of $\Phi^N$ on $\trian{N}{d}(I_d^N)$, hence~\eqref{eq:detProdRows_m} follows from Lemma~\ref{lemma:detProdRows}.
Furthermore, since $m^1_1$ is a multiple of $I_d$ with determinant $1$, we have $m^1_1=I_d$.
\end{proof}

\begin{theorem}
\label{thm:WhittakerAsymptotics}
For any $\lambda\in \C^N$, we have
\begin{equation}
\label{eq:WhittakerAsymptotics}
\psi^N_{\lambda}(r^N(k))
\widesim[2.5]{k\to\infty}
\frac{1}{\sqrt{\abs{\mathcal{H}(m)}}}
\left( \frac{2\pi}{k}\right)^{\frac{N(N-1)d(d+1)}{8}} \e^{-k \Phi^N(m)} \, ,
\end{equation}
where $\abs{\mathcal{H}(m)}>0$ is the Hessian determinant of $\Phi^N$ at $m$.
\end{theorem}

The case $d=1$, $N=2$ of this asymptotic result is classical; the case $d=1$ and general $N$ can be found in~\cite[eq.~20]{oConnell12}.
Finally, the case $d>1$, $N=2$ may be inferred from the Laplace approximation of Bessel functions of matrix arguments studied in~\cite{butlerWood03} (see also~\cite[Appendix B]{grabsch18} and~\cite[Section 2.6]{oConnell21}).

An important feature of~\eqref{eq:WhittakerAsymptotics} is that the leading order asymptotics does \emph{not} depend on the parameter $\lambda$.
This was already remarked in~\cite{bumpHuntley95} in the special case $d=1$ and $N=3$, for which the full asymptotic expansion was obtained.

\begin{proof}[Proof of Theorem~\ref{thm:WhittakerAsymptotics}]
By~\eqref{eq:WhittakerClosedForm}, we have
\[
\psi^N_{\lambda}(r^N(k))
= \int_{\trian{N}{d}(r^N(k))}
\Bigg(\prod_{i=1}^{N-1} \prod_{j=1}^i \mu(\diff x^i_j)\Bigg)
\Delta^N_{\lambda}(x)
\e^{-\Phi^N(x)} \, .
\]
Recalling~\eqref{eq:specialInitialCond}, let us change variables by setting
\begin{equation}
\label{eq:changeOfVars}
\tilde{x}^i_j
= r^i_j(k)^{-1} x^i_j
= k^{i-2j+1} x^i_j
\qquad\quad
\text{for } 1\leq j\leq i\leq N \, .
\end{equation}
One can then easily verify, using also the invariance property of the measure $\mu$, that
\begin{equation}
\label{eq:Whittaker_changeVars}
\psi^N_{\lambda}(r^N(k))
= \int_{\trian{N}{d}(I_d^N)}
\Bigg(\prod_{i=1}^{N-1} \prod_{j=1}^i \mu(\diff \tilde{x}^i_j)\Bigg)
\Delta^N_{\lambda}(\tilde{x}) 
\e^{-k \Phi^N(\tilde{x})} \, .
\end{equation}
Applying Prop.~\ref{prop:saddlePoint} with $g:=\Delta^N_{\lambda}$, we obtain
\[
\psi^N_{\lambda}(r^N(k))
\widesim[2.5]{k\to\infty}
\frac{\Delta^N_{\lambda}(m)}{\sqrt{\abs{\mathcal{H}(m)}}}
\left(\prod_{i=1}^{N-1} \big\lvert m^i_1 \cdots m^i_i \big\rvert \right)
\left( \frac{2\pi}{k}\right)^{\frac{N(N-1)d(d+1)}{8}} \e^{-k \Phi^N(m)} \, ,
\]
since the number of vertices of $\bm{G}$ that do \emph{not} belong to $\bm{\Gamma}$ is $N(N-1)/2$.
The claim then follows from Corollary~\ref{coro:rowProdMinimiser} (which, in particular, implies that $\Delta^N_{\lambda}(m)=1$).
\end{proof}

Recall now the definition~\eqref{eq:kerSigma_tilde} of the $\tilde{\Sigma}$-kernel.

\begin{corollary}
\label{coro:WhittakerLikeAsymptotics}
Let $f\colon \trian{N}{d} \to\R$ be a bounded and continuous function and let
\begin{equation}
\label{eq:f_k}
f_k(x) := f\big((r^i_j(k) x^i_j)_{1\leq j\leq i\leq N}\big) \qquad \text{for } k>0 \text{ and } x\in \trian{N}{d} \, .
\end{equation}
Assume that $f_k \xrightarrow{k\to\infty} f_{\infty}$ uniformly {on any compact subsets} of $\trian{N}{d}(I_d^N)$.
Then, for any $\lambda, \rho\in\R^N$,
\begin{equation}
\label{eq:WhittakerLikeAsymptotics}
\lim_{k\to\infty}
\frac{\tilde{\Sigma}_{\lambda}f(r^N(k))}{\psi^N_{\rho}(r^N(k))}
= f_{\infty}(m) \, .
\end{equation}
\end{corollary}
\begin{proof}
As the leading order asymptotics of the Whittaker function $\psi^N_{\lambda}(r^N(k))$ does not depend on $\lambda$ by Theorem~\ref{thm:WhittakerAsymptotics}, we have
\[
\lim_{k\to\infty} \frac{\psi^N_{\lambda}(r^N(k))}{\psi^N_{\rho}(r^N(k))}
=1 \, .
\]
Therefore, it suffices to prove~\eqref{eq:WhittakerLikeAsymptotics} for $\rho=\lambda$.

Note that, using~\eqref{eq:Whittaker_changeVars} and the fact that $\lambda\in\R^N$, the measure $\mu^N_k$ defined by
\[
\mu^N_k(\diff \tilde{x})
:=
\frac{1}{\psi^N_{\lambda}(r^N(k))}
\Delta^N_{\lambda}(\tilde{x})
\e^{-k \Phi^N(\tilde{x})}
\Bigg(\prod_{i=1}^{N-1} \prod_{j=1}^i \mu(\diff \tilde{x}^i_j)\Bigg)
\]
is a probability measure on $\trian{N}{d}(I_d^N)$.
By definition of $\tilde{\Sigma}_{\lambda}$, we then have
\[
\frac{\tilde{\Sigma}_{\lambda}f(r^N(k))}{\psi^N_{\lambda}(r^N(k))}
= \int\limits_{\trian{N}{d}(r^N(k))}
\Bigg(\prod_{i=1}^{N-1} \prod_{j=1}^i \mu(\diff x^i_j)\Bigg)
\frac{\Delta^N_{\lambda}(x)
\e^{-\Phi^N(x)}}{\psi^N_{\lambda}(r^N(k))}
f(x)
= \int\limits_{\trian{N}{d}(I_d^N)}
\mu^N_k(\diff \tilde{x})
f_k(\tilde{x}) \, ,
\]
where in the integral we performed the change of variables~\eqref{eq:changeOfVars}.
Since $f$ is bounded and continuous, the functions $\{f_k\}_{k>0}$ are uniformly bounded and continuous; moreover, by assumption, they converge as $k\to\infty$ to $f_{\infty}$ {uniformly on any compact subsets} of $\trian{N}{d}(I_d^N)$.
Therefore, by Lemma~\ref{lemma:convergenceOfIntegrals}, it is now enough to show that $\mu^N_k$ converges weakly as $k\to\infty$ to the Dirac measure $\delta_m$, i.e.\ that
\[
\lim_{k\to\infty} \int_{\trian{N}{d}(I_d^N)} \mu^N_k(\diff x) g(x)
= g(m)
\]
for every bounded and continuous function $g\colon \trian{N}{d}(I_d^N) \to\R$.
This claim, in turn, follows readily from Prop.~\ref{prop:saddlePoint}, since, without loss of generality, one can assume $g(m)\neq 0$.
\end{proof}

\subsection{Fixed-time law of the `bottom edge' process}
\label{subsec:bottomEdge}

Let us now go back to the Markov process $X$ on $\trian{N}{d}$ from Definition~\ref{def:triangularProcess}.
Recall that, under the hypotheses of Theorem~\ref{thm:bottomRowMarkov}, the $N$-th row $X^N$ of the process $X$ has an autonomous Markov evolution with time-$n$ transition kernel $\bm{P}^N_{\alpha(n),\beta}$ (cf.~\eqref{eq:kerP_doob}).
The transition kernel of $X^N$ from time $0$ to time $n$ is then given by the composition
\begin{equation}
\label{eq:Umeasure}
U^{N,n}_{\alpha,\beta} := \bm{P}^N_{\alpha(1), \beta}  \bm{P}^N_{\alpha(2), \beta} \cdots \bm{P}^N_{\alpha(n), \beta} \, .
\end{equation}
Thus, if the initial state of $X^N$ is $X^N(0)=z$, then the law of $X^N(n)$ is $U^{N,n}_{\alpha,\beta}(z; \cdot )$.

Let now $\lambda = (\lambda_1,\dots,\lambda_N) \in \C^N$ such that $\alpha(\ell) + \Re(\lambda_i) > \frac{d-1}{2}$ for all $1\leq \ell\leq n$ and $1\leq i\leq N$.
Iterating the eigenfunction equation~\eqref{eq:eigenfnEqn} $n$ times, we obtain the following eigenfunction equation for $U^{N,n}_{\alpha,\beta}$:
\begin{equation}
\label{eq:iteratedEigenfn}
U^{N,n}_{\alpha,\beta}
\frac{\psi^N_{\lambda}}{\psi^N_{\beta}}
= \left( \prod_{\ell=1}^n \prod_{i=1}^N \frac{\Gamma_d(\alpha(\ell) + \lambda_i)}{\Gamma_d(\alpha(\ell) + \beta^i)} \right)
\frac{\psi^N_{\lambda}}{\psi^N_{\beta}}
\, .
\end{equation}

Consider now the initial state $X^N(0)=r^N(k)$ (cf.~\eqref{eq:specialInitialCond}), which becomes singular in the limit as $k\to\infty$.
We will show that the measure $U^{N,n}_{\alpha,\beta}(r^N(k); \cdot)$ converges, as $k\to\infty$, to the {matrix Whittaker measure} with parameters $(\alpha(1),\dots,\alpha(n))$ and $\beta$.
An intuition about this fact is provided by~\eqref{eq:iteratedEigenfn}.
It follows from Theorem~\ref{thm:WhittakerAsymptotics} that the ratio of Whittaker functions on the right-hand side of~\eqref{eq:iteratedEigenfn}, evaluated at $r^N(k)$, converges to $1$ as $k\to\infty$.
It is then easy to see that, if the convergence to {matrix Whittaker measures} holds as claimed above, then~\eqref{eq:iteratedEigenfn} reduces to the Whittaker integral identity proved in~\S~\ref{subsec:Stade_matrix}.

\begin{theorem}
\label{thm:specialInitialCond}
Let $n\geq N$.
As $k\to\infty$, the distribution $U^{N,n}_{\alpha,\beta}(r^N(k); \cdot)$ converges in total variation distance (and, hence, weakly) to the {matrix Whittaker measure} with parameters $(\alpha(1), \dots, \alpha(n))$ and $\beta$ (which we denote by $W^{N,n}_{\alpha,\beta}$ for simplicity).
Namely, we have
\begin{equation}
\label{eq:totVarCov}
\lim_{k\to\infty}
\sup_{A} \abs{U^{N,n}_{\alpha,\beta}(r^N(k); A) - W^{N,n}_{\alpha,\beta}(A)} = 0 \, ,
\end{equation}
where the supremum is taken over all measurable sets $A \subseteq \pos_d^N$.
\end{theorem}

\begin{proof}
We will prove that
\begin{equation}
\label{eq:scheffe}
\lim_{k\to\infty}
\int_{\pos_d^N} \mu^{\otimes N}(\diff z)
\abs{U^{N,n}_{\alpha,\beta}(r^N(k); z) - W^{N,n}_{\alpha,\beta}(z)} = 0 \, .
\end{equation}
This statement is stronger than~\eqref{eq:totVarCov}, as the supremum in~\eqref{eq:totVarCov} is clearly bounded from above by the integral in~\eqref{eq:scheffe}.

Let us fix $N$ and prove~\eqref{eq:scheffe} by induction on $n\geq N$.
Before proving the base case, let us verify the (simpler) induction step.
Assume that~\eqref{eq:scheffe} holds for a certain $n\geq N$.
Using~\eqref{eq:WhittakerMeasure}, \eqref{eq:kerP_doob}, \eqref{eq:kerP_inverse} and~\eqref{eq:WhittakerExt}, we obtain
\[
\begin{split}
\big(W^{N,n}_{\alpha,\beta}
\bm{P}^N_{\alpha(n+1),\beta}\big)(z)
&= \left(\prod_{\ell=1}^{n+1} \prod_{i=1}^N \frac{1}{\Gamma_d(\alpha(\ell) + \beta^i)}\right)
\big(\cev{P}^N_{\alpha(n+1)} \psi^{N,n}_{\alpha(1:n);I_d}\big)(z)
\psi^N_{\beta}(z) \\
&= \left(\prod_{\ell=1}^{n+1} \prod_{i=1}^N \frac{1}{\Gamma_d(\alpha(\ell) + \beta^i)}\right)
\psi^{N,n+1}_{\alpha(1:n+1);I_d}(z)
\psi^N_{\beta}(z)
= W^{N,n+1}_{\alpha,\beta}(z)
\end{split}
\]
for $z\in\pos_d^N$.
On the other hand, by~\eqref{eq:Umeasure} we have $U^{N,n+1}_{\alpha,\beta} = U^{N,n}_{\alpha,\beta} \bm{P}^N_{\alpha(n+1),\beta}$.
Applying Fubini's theorem and recalling that $\bm{P}^N_{\alpha(n+1),\beta}$ is a Markov kernel, we then obtain
\[
\begin{split}
&\quad \int_{\pos_d^N} \mu^{\otimes N}(\diff z) \abs{U^{N,n+1}_{\alpha,\beta}(r^N(k); z) - W^{N,n+1}_{\alpha,\beta}(z)} \\
&= \int_{\pos_d^N} \mu^{\otimes N}(\diff z) \abs{\int_{\pos_d^N} \mu^{\otimes N}(\diff \tilde{z})
\left(U^{N,n}_{\alpha,\beta}(r^N(k); \tilde{z})
- W^{N,n}_{\alpha,\beta}( \tilde{z}) \right)
\bm{P}^N_{\alpha(n+1),\beta}(\tilde{z};z) } \\
&\leq \int_{\pos_d^N} \mu^{\otimes N}(\diff \tilde{z})
\abs{U^{N,n}_{\alpha,\beta}(r^N(k); \tilde{z}) - W^{N,n}_{\alpha,\beta}(\tilde{z})}
\underbrace{\int_{\pos_d^N} \mu^{\otimes N}(\diff z)
\bm{P}^N_{\alpha(n+1),\beta}(\tilde{z};z)}_{=1 \quad\text{for all } \tilde{z}} \, .
\end{split}
\]
The latter expression vanishes as $k\to\infty$ by the induction hypothesis, thus proving the induction step.

It remains to prove the base case, i.e.~\eqref{eq:scheffe} for $n=N$.
Recall that the measures $U^{N,N}_{\alpha,\beta}(r^N(k); \cdot)$, for any $k>0$, and $W^{N,N}_{\alpha,\beta}$ have the same finite total mass, since they are all probability distributions, and are absolutely continuous with respect to $\mu^{\otimes N}$.
By Scheff{\'e}'s theorem (see e.g.~\cite[Theorem~16.12]{billingsley95}), it then suffices to show the convergence of the densities:
\begin{equation}
\label{eq:scheffeDensities}
\lim_{k\to\infty} U^{N,N}_{\alpha,\beta}(r^N(k); z)
= W^{N,N}_{\alpha,\beta}(z)
\qquad\qquad
\text{for $\mu^{\otimes N}$-almost every } z\in \pos_d^N \, .
\end{equation}

Fix $z\in\pos_d^N$ once for all.
Using~\eqref{eq:Umeasure}, we write
\[
\begin{split}
U^{N,N}_{\alpha,\beta}(r^N(k); z)
= \, &\int_{\pos_d^N}
\bm{P}^N_{\alpha(1), \beta}(r^N(k); \diff z^1)
\int_{\pos_d^N}\bm{P}^N_{\alpha(2), \beta}(z^1; \diff z^2) \cdots \\
&\cdots
\int_{\pos_d^N} \bm{P}^{N}_{\alpha(N-1), \beta}(z^{N-2}; \diff z^{N-1})
\bm{P}^N_{\alpha(N), \beta}(z^{N-1}; z)
\end{split}
\]
Define now
\begin{equation}
\label{eq:WhittakerMeasure_prelimit2}
\begin{split}
J^N_{\alpha}(z^0;z^N)
:= \int\limits_{\pos_d^{N(N-1)}}
\left(\prod_{\ell=1}^{N-1} \mu^{\otimes N}(\diff z^\ell)\right)
\prod_{i,j=1}^N
\frac{\big\lvert z^{i-1}_j\big\rvert^{\alpha(i)}}{\big\lvert z^i_j \big\rvert^{\alpha(i)}}
\e^{ - \tr\left[ z^i_{j+1} (z^{i-1}_j)^{-1} + z^{i-1}_j (z^i_j)^{-1} \right] }
\end{split}
\end{equation}
for $z^0, z^N\in \pos_d^N$, with the usual conventions $z^i_{N+1}:=0$ for all $i=0,\dots,N$.
Using the definition~\eqref{eq:kerP_doob} of the $\bm{P}$-kernels, we then have
\begin{equation}
\label{eq:WhittakerMeasure_prelimit1}
U^{N,N}_{\alpha,\beta}(r^N(k); z)
= \left(\prod_{\ell,i=1}^N \frac{1}{\Gamma_d(\alpha(\ell)+\beta^i)}\right)
\frac{\psi^N_{\beta}(z)}{\psi^N_{\beta}(r^N(k))}
J^N_{\alpha}(r^N(k);z) \, .
\end{equation}
Comparing~\eqref{eq:WhittakerMeasure_prelimit1} with~\eqref{eq:WhittakerMeasure_N=n}, we are reduced to show that
\begin{equation}
\label{eq:WhittakerMeasure_limit}
\lim_{k\to\infty} \frac{J^N_{\alpha}(r^N(k);z)}{\psi^N_{\beta}(r^N(k))}
= \e^{-\tr\left[z_N^{-1} \right]} \psi^N_{\alpha(1:N)}(z) \quad\qquad
\text{for $\mu^{\otimes N}$-almost every $z\in \pos_d^N$.}
\end{equation}

Let us relabel the variables in the integral~\eqref{eq:WhittakerMeasure_prelimit2} by setting
\begin{align}
\label{eq:relabelling_x}
z^i_j &= x^{N-i}_{j-i} \qquad &&\text{for } 0\leq i\leq N-1 \text{ and } i+1\leq j\leq N \, , \\
\label{eq:relabelling_y}
z^i_j &= y^i_j \qquad &&\text{for } 1\leq i\leq N \text{ and } 1\leq j\leq i \, .
\end{align}
This relabelling yields two triangular arrays $x,y\in \trian{N}{d}$.
See Fig.~\ref{fig:WhittakerMeasure} for a graphical representation of the variables $z^i_j$ and the corresponding arrays $x$ and $y$.
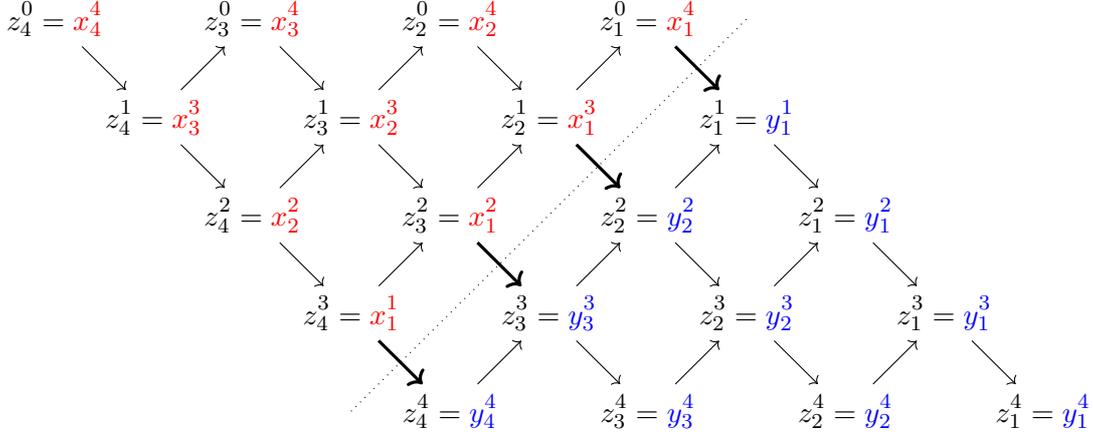
\begin{figure}
\centering
\begin{tikzpicture}[scale=1.3]

\node (z04) at (0,0) {$z^0_4=\red{x^4_4}$};
\node (z03) at (2,0) {$z^0_3=\red{x^4_3}$};
\node (z02) at (4,0) {$z^0_2=\red{x^4_2}$};
\node (z01) at (6,0) {$z^0_1=\red{x^4_1}$};

\node (z14) at (1,-1) {$z^1_4=\red{x^3_3}$};
\node (z13) at (3,-1) {$z^1_3=\red{x^3_2}$};
\node (z12) at (5,-1) {$z^1_2=\red{x^3_1}$};
\node (z11) at (7,-1) {$z^1_1=\blue{y^1_1}$};

\node (z24) at (2,-2) {$z^2_4=\red{x^2_2}$};
\node (z23) at (4,-2) {$z^2_3=\red{x^2_1}$};
\node (z22) at (6,-2) {$z^2_2=\blue{y^2_2}$};
\node (z21) at (8,-2) {$z^2_1=\blue{y^2_1}$};

\node (z34) at (3,-3) {$z^3_4=\red{x^1_1}$};
\node (z33) at (5,-3) {$z^3_3=\blue{y^3_3}$};
\node (z32) at (7,-3) {$z^3_2=\blue{y^3_2}$};
\node (z31) at (9,-3) {$z^3_1=\blue{y^3_1}$};

\node (z44) at (4,-4) {$z^4_4=\blue{y^4_4}$};
\node (z43) at (6,-4) {$z^4_3=\blue{y^4_3}$};
\node (z42) at (8,-4) {$z^4_2=\blue{y^4_2}$};
\node (z41) at (10,-4) {$z^4_1=\blue{y^4_1}$};

\draw[->] (z04) -- (z14);
\draw[->] (z14) -- (z03);
\draw[->] (z03) -- (z13);
\draw[->] (z13) -- (z02);
\draw[->] (z02) -- (z12);
\draw[->] (z12) -- (z01);
\draw[->, very thick] (z01) -- (z11);

\draw[->] (z14) -- (z24);
\draw[->] (z24) -- (z13);
\draw[->] (z13) -- (z23);
\draw[->] (z23) -- (z12);
\draw[->, very thick] (z12) -- (z22);
\draw[->] (z22) -- (z11);
\draw[->] (z11) -- (z21);

\draw[->] (z24) -- (z34);
\draw[->] (z34) -- (z23);
\draw[->, very thick] (z23) -- (z33);
\draw[->] (z33) -- (z22);
\draw[->] (z22) -- (z32);
\draw[->] (z32) -- (z21);
\draw[->] (z21) -- (z31);

\draw[->, very thick] (z34) -- (z44);
\draw[->] (z44) -- (z33);
\draw[->] (z33) -- (z43);
\draw[->] (z43) -- (z32);
\draw[->] (z32) -- (z42);
\draw[->] (z42) -- (z31);
\draw[->] (z31) -- (z41);

\draw[dotted] (3,-4) -- (7,0);

\end{tikzpicture}
\caption{
Graphical representation of the set of variables $z^i_j$, $0\leq i\leq N$ and $1\leq j\leq N$ (here $N=4$) appearing in~\eqref{eq:WhittakerMeasure_prelimit2}.
Each arrow $a\to b$ corresponds to the term $\e^{-\tr[ab^{-1}]}$ in the integral.
Relabelling the $z^i_j$ as in~\eqref{eq:relabelling_x}-\eqref{eq:relabelling_y} yields two triangular arrays $x$ (coloured in red) and $y$ (coloured in blue).}
\label{fig:WhittakerMeasure}
\end{figure}
Recalling the definition~\eqref{eq:kerSigma_tilde} of the $\tilde{\Sigma}$-kernel, we have
\[
J^N_{\alpha}(r^N(k);z)
= \tilde{\Sigma}^N_{\hat{\alpha}(1:N)} f (r^N(k)) \, ,
\]
where $\hat{\alpha}(1:N) := (-\alpha(N),\dots,-\alpha(1))$ and the function $f\colon \trian{N}{d}\to \R$ is defined by
\[
f(x)
:= \int_{\trian{N}{d}(z)}
\left(\prod_{i=1}^{N-1} \prod_{j=1}^i \mu(\diff y^i_j)\right)
\left(\prod_{i=1}^N \e^{-\tr\left[x^{N-i+1}_1 (y^i_i)^{-1} \right]}\right)
\Delta^N_{\alpha(1:N)}(y) \e^{-\Phi^N(y)} \, .
\]
Here, each term $\e^{-\tr\left[x^{N-i+1}_1 (y^i_i)^{-1} \right]}$ corresponds, graphically, to a bold arrow in Fig.~\ref{fig:WhittakerMeasure}.

We now wish to apply Corollary~\ref{coro:WhittakerLikeAsymptotics}.
Notice first that $f$ is a continuous function of $x$; moreover, it is bounded below by $0$ and above by $\psi^N_{\alpha(1:N)}(z)$ (cf.~\eqref{eq:WhittakerClosedForm}).
The associated functions $f_k$ defined in~\eqref{eq:f_k} are
\[
f_k(x)
= \int_{\trian{N}{d}(z)}
\left(\prod_{i=1}^{N-1} \prod_{j=1}^i \mu(\diff y^i_j)\right)
\left(\prod_{i=1}^N \e^{-k^{-(N-i)}\tr\left[x^{N-i+1}_1 (y^i_i)^{-1} \right]}\right)
\Delta^N_{\alpha(1:N)}(y) \e^{-\Phi^N(y)} \, .
\]
By dominated convergence and by the definition~\eqref{eq:WhittakerClosedForm} of Whittaker function, we have the pointwise convergence
\[
\lim_{k\to\infty} f_k(x) 
=\e^{-\tr\left[x^1_1 z_N^{-1} \right]}
\psi^N_{\alpha(1:N)}(z) =: f_{\infty}(x) \, .
\]
Notice that $\{f_k\}_{k>0}$ is a collection of continuous functions, increasing with $k$, that converges pointwise to a continuous limit; hence, by Dini's theorem (see e.g.~\cite[Theorem~7.13]{rudin76}), the convergence is uniform on compacts.
Then, the assumptions of Corollary~\ref{coro:WhittakerLikeAsymptotics} are satisfied and we have
\[
\lim_{k\to\infty}
\frac{J^N_{\alpha}(r^N(k);z)}{\psi^N_{\beta}(r^N(k))}
= \lim_{k\to\infty}
\frac{\tilde{\Sigma}^N_{\hat{\alpha}(1:N)} f (r^N(k))}{\psi^N_{\beta}(r^N(k))}
= f_{\infty}(m) \, ,
\]
where $m$ is the unique global minimiser of $\Phi^N$ on $\trian{N}{d}(I_d^N)$ (cf.\ Corollary~\ref{coro:uniqueMinimiser}).
Since $m^1_1=I_d$ by Corollary~\ref{coro:rowProdMinimiser}, we have
\[
f_{\infty}(m) = \e^{-\tr\left[z_N^{-1} \right]} \psi^N_{\alpha(1:N)}(z) \, .
\]
This yields the desired limit~\eqref{eq:WhittakerMeasure_limit}.
\end{proof}

\subsection{{Fixed-time laws of the `right edge' and `left edge' processes}}
\label{subsec:rightEdge}

Throughout this subsection, it will be convenient to work with the space of $d\times d$ \emph{positive semidefinite} matrices, i.e.\ $d\times d$ real symmetric matrices with nonnegative eigenvalues; such a space is the closure of $\pos_d$ under the standard Euclidean topology, and we thus denote it by $\spos_d$.

It is clear from the definition given in~\S~\ref{subsec:MarkovDynamics} that the `right edge' $X_1=(X^1_1,\dots,X^N_1)$ of $X$ is a Markov process in its own filtration.
Furthermore, {as mentioned before,} $X_1$ equals the system $Z=(Z^1,\dots,Z^N)$ of random particles {in} $\pos_d$ with one-sided interactions defined in~\eqref{eq:rightEdge1}-\eqref{eq:rightEdge2}, where the random weight $V^i(n)$ equals $W^i_1(n)$, an inverse Wishart random matrix with parameter $\alpha(n)+\beta^i$.
If the initial state $Z(0)$ of this process is in $\pos_d^N$ (respectively, $\spos_d^N$), then  clearly $Z$ evolves as a process in $\pos_d^N$ (respectively, $\spos_d^N$).

{Analogously, the `left edge' of $X$ is a Markov process in its own filtration.
Its `inverse' $L=(L^1,\dots,L^N):=((X^1_1)^{-1},\dots,(X^N_N)^{-1})$ {is given by}~\eqref{eq:leftEdge1}-\eqref{eq:leftEdge2}, where $U^i(n):=(W^i_i(n))^{-1}$ is a Wishart random matrix with parameter $\alpha(n)+\beta^i$.
If the initial state $L(0)$ of this process is in $\pos_d^N$ (respectively, $\spos_d^N$), then  clearly $L$ evolves as a process in $\pos_d^N$ (respectively, $\spos_d^N$).}

As the next lemma shows, the singular initial state of the bottom edge of $X$ considered in~\S~\ref{subsec:bottomEdge} induces (through Theorem~\ref{thm:bottomRowMarkov}) the initial state~\eqref{eq:stepInitialConf_matrix} on the right edge $X_1$, which resembles the step or `narrow wedge' initial configuration in systems of interacting particles/random walks.
{A similar statement holds for the left edge.}

\begin{lemma}
\label{lemma:rightEdge_initialConf}
Let $X(0)$ be distributed according to $\overline{\Sigma}_{\beta}(r^N(k);\cdot)$.
Then, on the space $\spos_d^N$, {both $(X^1_1(0),\dots,X^N_1(0))$ and $((X^1_1(0))^{-1},\dots,(X^N_N(0))^{-1})$} converge in law, as $k\to\infty$, to $(I_d,0_d,\dots,0_d)$.
\end{lemma}
\begin{proof}
{We prove the claim about $(X^1_1(0),\dots,X^N_1(0))$, as the proof of the claim about $((X^1_1(0))^{-1},\dots,(X^N_N(0))^{-1})$ is completely analogous.}

Let $g\colon \spos_d^N \to \R$ be a bounded and continuous test function.
We need to prove that
\begin{equation}
\label{eq:rightEdge_initialConf}
\lim_{k\to\infty} \E\left[ g(X_1(0)) \right] = g(I_d,0_d,\dots,0_d) \, .
\end{equation}
Let $f\colon \trian{N}{d} \to \R$, $f(x):= g(x_1) = g(x^1_1,\dots,x^N_1)$ for all $x\in \trian{N}{d}$.
By definition~\eqref{eq:kerSigma_normal} of $\overline{\Sigma}_{\beta}$, we then have
\[
\E\left[ g(X_1(0)) \right]
= \E\left[ f(X(0)) \right]
= \overline{\Sigma}_{\beta}f(r^N(k))
= \frac{\tilde{\Sigma}_{\beta}f(r^N(k))}{\psi^N_{\beta}(r^N(k))} \, .
\]
We now wish to apply Corollary~\ref{coro:WhittakerLikeAsymptotics}.
Since $g$ is bounded and continuous, $f$ also is.
The associated functions $f_k$ defined in~\eqref{eq:f_k} are
\[
f_k(x) := g\bigg(x^1_1, \frac{x^2_1}{k}, \dots, \frac{x^N_1}{k^{N-1}}\bigg) \, , \qquad x\in \trian{N}{d} \, .
\]
These functions converge as $k\to\infty$ to $f_{\infty}(x) := g(x^1_1,0_d,\dots,0_d)$ uniformly on compacts, since $g$ is continuous on $\spos_d^N$.
Therefore, by Corollary~\ref{coro:WhittakerLikeAsymptotics}, $\E\left[ g(X_1(0)) \right]$ converges as $k\to\infty$ to $g(m^1_1,0_d,\dots,0_d)$, where $m$ is the minimiser of $\Phi^N$ on $\trian{N}{d}(I_d^N)$.
By Corollary~\ref{coro:rowProdMinimiser} we have $m^1_1=I_d$, and the claim~\eqref{eq:rightEdge_initialConf} follows.
\end{proof}

As a consequence of Theorem~\ref{thm:specialInitialCond} and Lemma~\ref{lemma:rightEdge_initialConf}, we obtain:

\begin{corollary}
\label{coro:rightMarginal}
{As above, let $Z=(Z^1,\dots,Z^N)$ and $L=(L^1,\dots,L^N)$ be the right edge process and the (inverse) left edge process, respectively, with initial states $Z(0)=L(0)=(I_d,0_d,\dots,0_d)\in\spos_d^N$.
Then, for $n\geq N$, $Z^N(n)$ and $L^N(n)$ are distributed as the first marginal and the $N$-th marginal, respectively, of the {matrix Whittaker measure} with parameters $(\alpha(1), \dots, \alpha(n))$ and $\beta$.}
\end{corollary}
\begin{proof}
{Again, we only prove the claim about the right edge, as the proof of the claim about the left edge is completely analogous.}

Let the process $X$ be as in Definition~\ref{def:triangularProcess}, with initial state $X(0)$ distributed according to $\overline{\Sigma}_{\beta}(r^N(k);\cdot)$.
It is clear from the definition that $X^N_1(n)$ can be written as a continuous, deterministic function of the right edge initial state $X_1(0)$ and of the collection of random matrices $(W^1_1(\ell),\dots,W^N_1(\ell))_{1\leq \ell \leq n}$.
Therefore, by Lemma~\ref{lemma:rightEdge_initialConf} and the continuous mapping theorem~\cite[Theorem~2.7]{billingsley99}, $X^N_1(n)$ converges in law as $k\to\infty$ to $Z^N(n)$.

On the other hand, by Theorem~\ref{thm:specialInitialCond}, for $n\geq N$, $X^N_1(n)$ converges in law as $k\to\infty$ to the first marginal of a {matrix Whittaker measure} with parameters $(\alpha(1), \dots, \alpha(n))$ and $\beta$.
\end{proof}

\begin{remark}
The following generalisation of Corollary~\ref{coro:rightMarginal} is immediate: under the same hypotheses, for every $1\leq i\leq N$ and {$n\geq i$}, $Z^i(n)$ is distributed as the first marginal of the {matrix Whittaker measure} with parameters $(\alpha(1), \dots, \alpha(n))$ and $(\beta^1,\dots,\beta^i)$.
This is due to the fact that, by definition, for any fixed $i\geq 1$, the process $(Z_1,\dots,Z_i)$ has both an initial configuration $(I_d,0_d,\dots,0_d)$ and a Markov evolution that do not depend on the choice of $N\geq i$.
{Analogously, for every $1\leq i\leq N$ and {$n\geq i$}, $L^i(n)$ is distributed as the {$i$-th marginal} of the same {matrix Whittaker measure}.}
\end{remark}

\section{Minimisation of energy functions and Laplace approximations}
\label{sec:minimisation}

In this section, we study minimisation problems for certain energy functions of matrix arguments on directed graphs.
As a consequence, we obtain Laplace approximations for integrals of exponentials of these energy functions.
For our purposes, the most important application of such results consists in certain asymptotics of Whittaker functions of matrix arguments; see~\S~\ref{subsec:asymptoticsWhittaker}.
However, the results of this section may be of independent interest.
For instance, the general framework we work with may be applied to obtain analogous asymptotics for \emph{orthogonal} Whittaker functions, which also appeared in the study of stochastic systems -- see~\cite{bisiZygouras19a, barraquandWang23}.

\subsection{Energy functions on directed graphs}
\label{subsec:energyFunctions}

Let us recall some terminology of graph theory that will be useful throughout this section.
A \emph{finite directed graph} $G=(V,E)$ is a pair consisting of a nonempty finite set $V$ of \emph{vertices} and a set $E \subset \{(v,w)\in V^2\colon v\neq w\}$ of \emph{edges}.
Note that edges connecting a vertex to itself are not allowed, nor are multiple edges.
The direction of an edge $(v,w)$ connecting $v$ to $w$ is given by the ordering of the pair.
For the sake of notational convenience, we also write $v\rightarrow w$ when $(v,w)\in E$, and $v \not\to w$ when $(v,w)\notin E$.
A vertex $v$ is called a \emph{sink} if it has no outcoming edges (i.e.\ if $v \not\to w$ for all $w\in V$) and a \emph{source} if it has no incoming edges (i.e.\ if $w \not\to v$ for all $w\in V$).
For any $v,w\in V$ and $0\leq l< \infty$, we call \emph{path} of length $l$ in $G$ from $v$ to $w$ any sequence $(v_0, v_1,\dots,v_l)$ such that $v_0=v$, $v_l = w$, and $v_{i-1}\to v_i$ for all $1\leq i\leq l$.
A \emph{cycle} is any path $(v_0, v_1,\dots,v_l)$ such that $v_0 = v_l$ and any other two vertices are distinct.
We say that $G$ is \emph{acyclic} if it has no cycles.
From now on, throughout the whole section, $G=(V,E)$ will always be an acyclic finite directed graph.

\begin{lemma}
\label{lemma:graph_directedPaths}
For all $v\in V$, there exists a path in $G$ from $v$ to a sink; moreover, there exists a path in $G$ from a source to $v$.
\end{lemma}
\begin{proof}
We will prove the existence of the first path only, as the existence of the second path follows from a similar argument.
We construct the path algorithmically.
Set $v_0 := v$.
For all $i=0,1,2,\dots$, we proceed as follows: if $v_i$ is a sink, then we stop the algorithm; otherwise, we pick $v_{i+1}$ to be any vertex such that $v_i \to v_{i+1}$.
If the algorithm never terminates, then there exist two distinct indices $i, j$ with $v_i=v_j$, since $G$ is finite; this implies that $G$ has a cycle, against the hypotheses.
Therefore, the procedure must stop in a finite number $l$ of steps, thus yielding a path $(v_0, v_1, \dots , v_l)$ from $v_0=v$ to a sink $v_l$.
\end{proof}

For any integer $d\geq 1$, let $\Sym_d$, $\Diag_d$, and $\Scal_d$ be the sets of $d\times d$ real symmetric matrices, real diagonal matrices, and real scalar matrices (i.e.\ multiples of the $d\times d$ identity matrix $I_d$), respectively.
We will write $\Sym_d^V$ for the set of arrays $x = (x_v)_{v\in V}$, where each $x_v\in\Sym_d$.
We will use the notations $\Diag_d^V$ and $\Scal_d^V$ in a similar way.

Let us define the `energy functions'
\begin{align}
\label{eq:energyFn_phi}
\phi_d \colon \Sym_d^V \to \R \, , &&\phi_d(x) &:= \sum_{\substack{v,w \in V \colon \\ v \to w}} \tr[\e^{x_v} \e^{-x_w}] \, , \\
\label{eq:energyFn_chi}
\chi_d \colon \Sym_d^V \to \R \, , &&\chi_d(x) &:= \sum_{\substack{v,w \in V \colon \\ v \to w}} \tr[\e^{x_v -x_w}] \, ,
\end{align}
where $\e^a$ denotes the usual exponential of the matrix $a$.
The Golden-Thompson inequality (see e.g.~\cite{bhatia97}) states that $\tr[\e^a \e^b] \geq \tr[\e^{a+b}]$ if $a$ and $b$ are symmetric matrices.
It follows that
\begin{equation}
\label{eq:goldenThompson}
\phi_d(x) \geq \chi_d(x)
\qquad\qquad
\text{for all } x\in \Sym_d^V \, .
\end{equation}
However, the two energy functions are identical only for $d=1$.

Notice that, by Lemma~\ref{lemma:graph_directedPaths}, $G$ has at least one sink and one source, possibly coinciding.
Throughout, we also assume that there exists at least one vertex of $G$ that is neither a source nor a sink.
We can thus fix a subset $\Gamma\subset V$ that contains all the sinks and sources and such that $\Gamma^{\mathsf{c}}$, the complement of $\Gamma$ in $V$, is nonempty.
For any set $S$ and any fixed array $z = (z_v)_{v\in \Gamma} \in S^{\Gamma}$, let
\begin{equation}
\label{eq:cylinder}
S^V(z) := \{ x=(x_v)_{v\in V} \in S^V \colon x_v = z_v \text{ for all } v\in \Gamma\} \, .
\end{equation}

Our first result concerns the asymptotic behaviour of the energy functions on $\Sym_d^V(z)$.
Let $\norm{\cdot}$ denote any norm on $\Sym_d^V$.
\begin{proposition}
\label{prop:phiAtInfinity}
Let $z\in\Sym_d^\Gamma$.
For $x\in\Sym_d^V(z)$, we have $\phi_d(x) \to \infty$ and $\chi_d(x) \to \infty$ as $\norm{x} \to \infty$.
\end{proposition}
\begin{proof}
By inequality~\eqref{eq:goldenThompson}, it suffices to prove the claim for $\chi_d$.
As all norms on a finite-dimensional space are equivalent, we may arbitrarily take
\begin{equation}
\label{eq:norm}
\norm{x} := \sum_{v\in V} \rho(x_v) \qquad\qquad
\text{for } x\in \Sym_d^V \, ,
\end{equation}
where $\rho(a)$ denotes the spectral radius of a symmetric matrix $a$ (i.e.\ the largest absolute value of its eigenvalues).
As the spectral radius is a norm on $\Sym_d$, it can be easily verified that~\eqref{eq:norm} defines a norm on $\Sym_d^V$.
We will show that, for any sequence $(x^{(n)})_{n\geq 1} \subseteq \Sym_d^V(z)$ such that $\lVert x^{(n)}\rVert \to\infty$ as $n\to\infty$, we have $\chi_d(x^{(n)})\to\infty$ as $n\to\infty$.
For the sake of notational simplicity, we will drop the superscript of $x^{(n)}$ and leave the dependence on $n$ implicit.

By contradiction, assume that there exists a positive constant $C$ such that, along a subsequence, $\chi_d(x)\leq C$.
Since $\norm{x}\to\infty$, there exists $w\in \Gamma^{\mathsf{c}}$ such that, along a further subsequence, $\rho(x_w)\to\infty$.
This implies that, passing to a final subsequence, either $\maxeig(x_w) \to\infty$ or $\maxeig(-x_w) \to\infty$, where $\maxeig(a)$ denotes the maximum eigenvalue of a symmetric matrix $a$.
As $w\in \Gamma^{\mathsf{c}}$, it is neither a source nor a sink.
By Lemma~\ref{lemma:graph_directedPaths}, there exists a path $(v_0, v_1, \dots, v_l)$ of length $l\geq 1$ in $G$ from $v_0=w$ to a sink $v_l \in \Gamma$.
Since $G$ has no cycles, we have $v_i \neq v_j$ for all $i\neq j$; therefore, all directed edges $v_{i-1} \to v_i$ ($1\leq i\leq d$) are distinct.
We thus have
\begin{equation}
\label{eq:upperBoundChi}
\begin{split}
C \geq \chi_d(x) \geq \sum_{i=1}^l \tr[\e^{x_{v_{i-1}} - x_{v_i}}] \geq \sum_{i=1}^l \e^{\maxeig(x_{v_{i-1}} - x_{v_i})} \geq \sum_{i=1}^l \maxeig(x_{v_{i-1}} - x_{v_i}) \, ,
\end{split}
\end{equation}
where we used the bounds $\tr[\e^y] \geq \maxeig(\e^y) = \e^{\maxeig(y)}$ for $y\in \Sym_d$ and $\e^\alpha \geq \alpha$ for $\alpha\in\R$.
Recall now that, for any $a,b\in \Sym_d$,
\begin{equation}
\label{eq:maxEingevalue}
\maxeig(a+b) \leq \maxeig(a) + \maxeig(b) \, .
\end{equation}
By iterating~\eqref{eq:maxEingevalue} several times and using~\eqref{eq:upperBoundChi}, we obtain
\[
\maxeig(x_w)
\leq \sum_{i=1}^l 
\maxeig(x_{v_{i-1}}-x_{v_i}) + \maxeig(x_{v_l})
\leq C + \maxeig(x_{v_l}) \, .
\]
By considering a path $(u_0, u_1, \dots, u_m)$ of length $m\geq 1$ from a source $u_0\in\Gamma$ to $u_m=w$ (which again exists by Lemma~\ref{lemma:graph_directedPaths}) and using similar bounds, we also have
\[
\maxeig(-x_w)
\leq \maxeig(-x_{u_0}) + C \, .
\]
Since either $\maxeig(x_w) \to\infty$ or $\maxeig(-x_w) \to\infty$, it follows that either $\maxeig(x_{v_l}) \to\infty$ or $\maxeig(-x_{u_0}) \to\infty$.
This contradicts the fact that $x_{v_l} = z_{v_l}$ and $x_{u_0} = z_{u_0}$ are both fixed for all $x\in\Sym_d^V(z)$, since $v_l, u_0 \in \Gamma$.
\end{proof}

\begin{remark}
Above we have assumed that $G=(V,E)$ is acyclic and that $\Gamma$ is a subset of $V$ containing all the sinks and sources of $G$.
We stress that both hypotheses are necessary for Prop.~\ref{prop:phiAtInfinity} to hold.
As a counterexample, let $G$ be the cycle graph with $n$ vertices and let $\Gamma=\emptyset$.
If $a\in\Sym_d$ and $x=(x_v)_{v\in V}$ is the array with $x_v =a$ for all $v$, then
\[
\phi_d(x)
= \chi_d(x)
= \underbrace{\tr[I_d] + \dots + \tr[I_d]}_{n \text{ times}}
= dn
\]
is constant in $a$; however, for the norm $\norm{\cdot}$ defined in~\eqref{eq:norm}, if $\rho(a)\to\infty$, then $\norm{x}\to\infty$.
\end{remark}

\subsection{Minima of energy functions}

We now study the minima of the functions~\eqref{eq:energyFn_phi}-\eqref{eq:energyFn_chi} on the set $\Sym_d^V(z)$, where $z\in\Sym_d^\Gamma$.
In words, we wish to minimise the energy functions subject to the constraint that some of the entries of the input array (precisely, those indexed by the vertices of the subset $\Gamma$) are fixed.

We start with the simplest case $d=1$, in which $\Sym_1 = \Diag_1 = \Scal_1 = \R$ and the two energy functions coincide:
\[
\phi_1 = \chi_1 \colon \R^V \to \R \, , \qquad \phi_1(x) = \chi_1(x) = \sum_{v \to w} \e^{x_v - x_w} \, .
\]
We denote by $\partial_v$ the partial derivative of a function on $\R^V$ with respect to the variable $x_v$.

\begin{lemma}
\label{lemma:convexityPhi_1}
Let $z\in\R^\Gamma$.
The Hessian matrix of $\phi_1$ on $\R^V(z)$ is positive definite everywhere.
In particular, $\phi_1$ is strictly convex on $\R^V(z)$.
\end{lemma}
\begin{proof}
On $\R^V(z)$ the variables indexed by $\Gamma$ are fixed to the assigned values $z$, hence we can consider $\phi_1$ and its Hessian as functions of $(x_v)_{v\in \Gamma^{\mathsf{c}}}$.
For $v,w\in\Gamma^{\mathsf{c}}$, we have
\[
\partial_v \partial_w \phi_1 =
\begin{dcases}
\sum_{u\in V} \left(\e^{x_v - x_u} \1_{v\to u} + \e^{x_u - x_v} \1_{u\to v} \right) &\text{if } v=w \, , \\
- \e^{x_v - x_w} \1_{v\to w} - \e^{x_w - x_v} \1_{w\to v} &\text{if } v\neq w \, .
\end{dcases}
\]
Thus, the quadratic form of the Hessian of $\phi_1$ on $\R^V(z)$, as a function of $\alpha = (\alpha_v)_{v\in \Gamma^{\mathsf{c}}}$, is
\[
\begin{split}
\sum_{v,w \in \Gamma^{\mathsf{c}}} \alpha_v \alpha_w \partial_v \partial_w \phi_1
= \, &\sum_{v\in \Gamma^{\mathsf{c}}} \alpha_v^2 \sum_{u\in V} \left(\e^{x_v - x_u} \1_{v\to u} + \e^{x_u - x_v} \1_{u\to v} \right) \\
&+ 2 \sum_{v,w\in \Gamma^{\mathsf{c}}} \alpha_v \alpha_w \left( - \e^{x_v - x_w} \1_{v\to w} - \e^{x_w - x_v} \1_{w\to v} \right) \, .
\end{split}
\]
Setting $\alpha_v :=0$ for all $v\in \Gamma$, it is easy to see that the latter expression equals
\[
\sum_{\substack{v,w\in V \colon \\ v\to w}} \e^{x_v - x_w} \left(\alpha_v^2 + \alpha_w^2 - 2\alpha_v \alpha_w \right)
=\sum_{\substack{v,w\in V \colon \\ v\to w}} \e^{x_v - x_w} \left(\alpha_v - \alpha_w \right)^2 \geq 0 \, .
\]
Therefore, the Hessian is positive semidefinite everywhere.
To prove that it is in fact positive definite, we will show that, if the quadratic form of the Hessian vanishes at $\alpha$, then $\alpha=0$.
If the above expression vanishes, then $\alpha_v = \alpha_w$ for all $v,w\in V$ such that $v\to w$.
Let $v \in \Gamma^{\mathsf{c}}$.
By Lemma~\ref{lemma:graph_directedPaths}, there exists a path from $v$ to a sink $s\in \Gamma$.
The value $\alpha_w$ is then the same for all the vertices $w$ along such a path.
We then have $\alpha_{v} = \alpha_s = 0$, since $s\in\Gamma$.
As $v\in \Gamma^{\mathsf{c}}$ was arbitrary, it follows that $\alpha=(\alpha_v)_{v\in \Gamma^{\mathsf{c}}} =0$.
\end{proof}

\begin{proposition}
\label{prop:minimiser_d=1}
The function $\phi_1=\chi_1$ has a unique (global) minimiser on $\R^V(z)$.
\end{proposition}
\begin{proof}
By Lemma~\ref{lemma:convexityPhi_1}, $\phi_1$ is a strictly convex function over the convex set $\R^V(z)$; therefore, it has at most one minimiser.
It remains to show the existence of a minimiser.
Since $\phi_1$ is a continuous function, it admits at least one minimiser on every closed ball $B_r := \{ x\in \R^V(z)\colon \norm{x}\leq r\}$.
By Prop.~\ref{prop:phiAtInfinity}, for $r$ large enough, the minimiser on $B_r$ is also a (global) minimiser on $\R^V(z)$.
\end{proof}

The case $d>1$ is much more challenging, and we are able to deal with it only under rather strong assumptions on the fixed array $z$.
Nonetheless, this is sufficient for our ultimate purposes.

We will be using the fact that the relation between the eigenvalues and the diagonal entries of a symmetric matrix is completely characterised by the \emph{majorisation} relation.
Let us briefly explain this statement, referring to~\cite[\S~4.3]{hornJohnson13} for proofs and details.
For any $\alpha = (\alpha_1, \dots, \alpha_d)\in \R^d$, let us denote by $\alpha^{\downarrow} = (\alpha^{\downarrow}_1, \dots, \alpha^{\downarrow}_d)$ its nonincreasing rearrangement, i.e.\ the permutation of the coordinates of $\alpha$ such that $\alpha^{\downarrow}_1 \geq \alpha^{\downarrow}_2 \geq \dots \geq \alpha^{\downarrow}_d$.
Given  $\alpha, \beta \in\R^d$, we say that $\alpha$ \emph{majorises} $\beta$, and write $\alpha \succ \beta$, if
\begin{equation}
\label{eq:majorisation}
\sum_{i=1}^k \alpha^{\downarrow}_i \geq \sum_{i=1}^k \beta^{\downarrow}_i
\qquad
\text{for } 1\leq k\leq d-1
\qquad
\text{and}\qquad
\sum_{i=1}^d \alpha_i = \sum_{i=1}^d \beta_i \, .
\end{equation}
\begin{theorem}[{\cite[Theorem~4.3.45]{hornJohnson13}}]
\label{thm:majorisation}
Let $x\in \Sym_d$.
Let $\lambda = (\lambda_1, \dots \lambda_d)$ be the vector of the (real) eigenvalues of $x$, taken in any order.
Let $\delta_i := x(i,i)$ for $1\leq i\leq d$, so that $\delta = (\delta_1, \dots, \delta_d)$ is the vector of the diagonal entries of $x$.
Then we have $\lambda \succ \delta$, and the equality $\lambda^{\downarrow} = \delta^{\downarrow}$ holds if and only if $x$ is a diagonal matrix.
\end{theorem}

We now briefly introduce the concept of Schur convexity and state the criterion that is useful for our purposes, referring e.g.\ to~\cite[Ch.~I.3]{marshallOlkinArnold11} for more details.
A function $H\colon \R^d \to \R$ is called \emph{Schur-convex} if $H(\alpha) \geq H(\beta)$ for all $\alpha,\beta\in \R^d$ such that $\alpha \succ \beta$.
In particular, for all $\alpha,\beta$ such that $\alpha^{\downarrow} = \beta^{\downarrow}$, we have $\alpha \succ \beta \succ \alpha$, hence $H(\alpha) = H(\beta)$; in other words, every Schur-convex function is a symmetric function.
Additionally, $H$ is called \emph{strictly Schur-convex} if $H(\alpha) > H(\beta)$ for all $\alpha,\beta\in \R^d$ such that $\alpha \succ \beta$ and $\alpha^{\downarrow} \neq \beta^{\downarrow}$.

\begin{theorem}[{\cite[Ch.~I.3, \S~C]{marshallOlkinArnold11}}]
\label{thm:SchurConvex}
Let $h\colon \R \to \R$ and
\[
H\colon \R^d \to \R \, , \qquad\qquad
H(\alpha_1,\dots,\alpha_d) = \sum_{i=1}^d h(\alpha_i) \, .
\]
If $h$ is convex, then $H$ is Schur-convex.
If $h$ is strictly convex, then $H$ is strictly Schur-convex.
\end{theorem}

As a consequence of the results just stated, we obtain:    
\begin{proposition}
\label{prop:traceInequality}
Suppose that $x\in \Sym_d$ and $y\in \Diag_d$ have the same diagonal entries.
Then $\tr[\e^x] \geq \tr[\e^y]$, and the equality holds if and only if $x=y$.
\end{proposition}
\begin{proof}
Let $\lambda = (\lambda_1,\dots,\lambda_d)$ be the vector of the eigenvalues of $x$, taken in any order.
Let $\delta=(\delta_1,\dots,\delta_d)$ be the vector of (common) diagonal entries of $x$ and $y$, i.e.\ $\delta_i = x(i,i) = y(i,i)$ for all $1\leq i\leq d$.
Since $y$ is diagonal, notice that the $\delta_i$'s are also its eigenvalues.
Therefore, the claimed inequality $\tr[\e^x] \geq \tr[\e^y]$ reads as $H(\lambda) \geq H(\delta)$, where
\[
H\colon \R^d \to \R \, , \qquad\qquad
H(\alpha) = H(\alpha_1,\dots,\alpha_d) := \sum_{i=1}^d \e^{\alpha_i} \, .
\]
The function $H$ is strictly Schur-convex by Theorem~\ref{thm:SchurConvex}, since the exponential function is strictly convex.
Since $\lambda \succ \delta$ by Theorem~\ref{thm:majorisation}, we then have $H(\lambda) \geq H(\delta)$, as required.
Moreover, assume that $H(\lambda) = H(\delta)$.
Then, by \emph{strict} Schur-convexity of $H$, we have $\lambda^{\downarrow} = \delta^{\downarrow}$.
Again by Theorem~\ref{thm:majorisation}, we conclude that $x$ is diagonal, which in turn implies $x=y$.
\end{proof}

From the latter proposition we deduce the existence and uniqueness of a minimiser of $\chi_d$ on $\Sym_d^V(z)$, under the assumption that all the `fixed' entries $z$ are diagonal matrices.

\begin{theorem}
\label{thm:minimiserChi}
Let $z= (z_v)_{v\in \Gamma} \in \Diag^{\Gamma}_d$ and set $z(i,i) := (z_v(i,i))_{v\in \Gamma} \in \R^{\Gamma}$ for all $1\leq i\leq d$.
Then, the function $\chi_d$ admits a unique minimiser on $\Sym_d^V(z)$.
Such a minimiser is of the form $m=(m_v)_{v\in V} \in \Diag_d^V(z)$, where $m(i,i) := (m_v(i,i))_{v\in V} \in \R^V$ denotes the unique minimiser of $\chi_1$ on $\R^V(z(i,i))$ for all $i$.
\end{theorem}
\begin{proof}
The claim will immediately follow from the two following facts:
\begin{enumerate}
\item
for any $x\in \Sym_d^V(z)$, there exists $y\in \Diag_d^V(z)$ such that $\chi_d(x)\geq \chi_d(y)$, with equality if and only if $x=y$;
\item there exists $m\in \Diag_d^V(z)$ (as in the statement of the theorem) such that $\chi_d(x) \geq \chi_d(m)$ for any $x\in \Diag_d^V(z)$, with equality if and only if $x=m$.
\end{enumerate}

\emph{Proof of (i).} Fix any $x\in \Sym_d^V(z)$.
Define $y=(y_v)_{v\in V}$ so that, for all $v\in V$, $y_v$ is the diagonal matrix with the same diagonal entries as $x$, i.e.\ $y_v(i,i) = x_v(i,i)$ for $1\leq i\leq d$.
Since each $z_v$ (for $v\in \Gamma$) is diagonal by hypothesis, we have $y\in \Diag_d^V(z)$.
For any $v,w\in V$, the matrices $x_v-x_w\in \Sym_d$ and $y_v-y_w\in\Diag_d$ have the same diagonal entries, hence $\tr[\e^{x_v - x_w}]
\geq \tr[\e^{y_v - y_w}]$ by Prop.~\ref{prop:traceInequality}; summing over $v\to w$, we obtain that $\chi_d(x) \geq \chi_d(y)$.
Assume now that $\chi_d(x) = \chi_d(y)$.
Then, $\tr[\e^{x_v - x_w}]
= \tr[\e^{y_v - y_w}]$ whenever $v\to w$.
Again by Prop.~\ref{prop:traceInequality}, we then have $x_v-y_v= x_w - y_w$ for all $v\to w$.
For any $v\in V$, by Lemma~\ref{lemma:graph_directedPaths} there exists a path $(v_0,v_1,\dots,v_l)$ in $G$ from $v_0=v$ to a sink $v_l$.
Since all sinks are in $\Gamma$ by assumption (see \S~\ref{subsec:energyFunctions}) and both $x$ and $y$ are in $\Sym_d^V(z)$, we have $x_{v_l}=z_{v_l}=y_{v_l}$.
Therefore, $x_v - y_v = x_{v_1} - y_{v_1} = \cdots = x_{v_l} - y_{v_l} =0$; in particular, $x_v=y_v$.
As $v\in V$ is arbitrary, we conclude that $x=y$.

\emph{Proof of (ii).} For $x\in \Diag_d^V(z)$, set $x(i,i):=(x_v(i,i))_{v\in V} \in \R^V$.
As each $x_v$ is diagonal, we have
\[
\chi_d(x)
= \sum_{v \to w} \tr[\e^{x_v - x_w}]
= \sum_{v \to w} \sum_{i=1}^d \e^{x_v(i,i) - x_w(i,i)}
=\sum_{i=1}^d \chi_1(x(i,i)) \, .
\]
By Prop.~\ref{prop:minimiser_d=1}, for all $i$, $\chi_1$ has a unique minimiser $m(i,i)$ on $\R^V(z(i,i))$.
Therefore, we have
\[
\chi_d(x)
= \sum_{i=1}^d \chi_1(x(i,i))
\geq \sum_{i=1}^d \chi_1(m(i,i))
= \chi_d(m) \, ,
\]
and the inequality is strict whenever $x\neq m$.
\end{proof}

In the case where the `fixed' entries $z$ are scalar matrices, the inequality~\eqref{eq:goldenThompson} immediately implies the existence and uniqueness of a minimiser of $\phi_d$.

\begin{corollary}
\label{coro:minimiser_phi}
Let $z= (z_v)_{v\in \Gamma} \in \Scal^{\Gamma}_d$, so that $z_v = \zeta_v I_d$ for all $v\in \Gamma$ and for a certain $\zeta = (\zeta_v)_{v\in V} \in \R^{\Gamma}$.
Then, the function $\phi_d$ admits a unique minimiser on $\Sym_d^V(z)$.
Such a minimiser is of the form $m=(m_v)_{v\in V} = (\mu_v I_d)_{v\in V} \in \Scal_d^V(z)$, where $\mu = (\mu_v)_{v\in V} \in \R^V$ is the unique minimiser of $\phi_1$ on $\R^V(\zeta)$.
\end{corollary}
\begin{proof}
Since $\Scal_d^{\Gamma} \subseteq \Diag_d^{\Gamma}$, it follows from Theorem~\ref{thm:minimiserChi} that $\chi_d$ has a unique minimiser $m$ on $\Sym_d^V(z)$, which is of the form specified above.
Since the scalar matrices $m_v$ and $m_w$ commute for any $v,w\in V$, we have $\e^{m_v} \e^{-m_w} = \e^{m_v-m_w}$, hence $\phi_d(m) = \chi_d(m)$.
By~\eqref{eq:goldenThompson}, we then have
\[
\phi_d(x) \geq \chi_d(x) \geq \chi_d(m) = \phi_d(m)
\qquad\qquad
\text{for all } x\in \Sym_d^V(z) \, ,
\]
where the second inequality is strict if $x\neq m$.
It follows that $m$ is also the unique minimiser of $\phi_d$ on $\Sym_d^V(z)$.
\end{proof}

\subsection{Energy functions in logarithmic variables}
\label{subsec:logarithmic variables}

It is a well-known fact that the functions
\begin{align*}
\Sym_d \to \pos_d \, , \quad a \mapsto \e^{a} \qquad\qquad\text{and}\qquad\qquad
\pos_d \to \Sym_d \, , \quad a \mapsto \log a \, ,
\end{align*}
namely the matrix exponential and the matrix logarithm, are both bijections on the stated domains and inverse to each other.
From now on, for any set $S$, we will use the following compact notations: $\log x := (\log x_v)_{v\in S} \in \Sym_d^S$ for $x=(x_v)_{v\in S}\in \pos_d^S$, and $\e^{x} := (\e^{x_v})_{v\in S} \in \pos_d^S$ for $x=(x_v)_{v\in S}\in \Sym_d^S$.

Let us consider the analogue of $\phi_d$ `in logarithmic variables', that is the energy function $\Phi_d(x) := \phi_d(\log x)$ for $x\in \pos_d^V$.
More explicitly, recalling~\eqref{eq:energyFn_phi}, we define
\begin{align}
\label{eq:energyFn_Phi}
\Phi_d \colon \pos_d^V \to \R \, ,
\qquad
\Phi_d(x) := \sum_{\substack{v,w \in V \colon \\ v \to w}} \tr[x_v x_w^{-1}]
\qquad \text{for all } x= (x_v)_{v\in V} \in \pos_d^V \, .
\end{align}

Take now $z = (z_v)_{v\in \Gamma}$ such that each $z_v$ is a positive multiple of $I_d$, or equivalently $\log z \in \Scal^{\Gamma}_d$.
By Corollary~\ref{coro:minimiser_phi}, $\Phi_d$ has a unique minimiser $m$ on $\pos_d^V(z)$, where $\log m$ is the unique minimiser of $\phi_d$ on $\Sym_d^V(\log z)$.
This implies that, on $\pos_d^V(z)$, the Hessian of $\Phi_d$ at $m$ is positive semidefinite.
We now aim to prove the stronger statement that the Hessian of $\Phi_d$ at $m$ is \emph{positive definite}.

As in the previous subsections, we first work with $d=1$.
Recall that, by Lemma~\ref{lemma:convexityPhi_1}, $\phi_1$ is strictly convex on $\R^V(\log z)$, for $z \in \pos_1^{\Gamma}$.
The analogous statement does not hold for $\Phi_1$ on $\pos_1^V(z)$; however, the following is still true:

\begin{lemma}
\label{lemma:hessPhi_d=1}
For $z\in \pos_1^{\Gamma}$, the Hessian of $\Phi_1$ on $\pos_1^V(z)$ is positive definite at any critical point.
\end{lemma}
\begin{proof}
We prove the claim by simply expressing the derivatives of $\Phi_1$ in terms of the derivatives of $\phi_1$.
For $v\in \Gamma^{\mathsf{c}}$, the first partial derivative of $\Phi_1$ w.r.t.\ $x_v$ is
\begin{equation}
\label{eq:derivativesPhi&phi}
\partial_v \Phi_1 (x)
= \frac{1}{x_v} \partial_v \phi_1 (\log x) \qquad\qquad
\text{for any } x\in \pos_1^V(z) \, .
\end{equation}
Therefore, for $v,w\in \Gamma^{\mathsf{c}}$,
\[
\partial_v \partial_w \Phi_1 (x) =
\begin{cases}
\displaystyle
\frac{1}{x_v^2} \left[ \partial^2_v \phi_1 - \partial_v \phi_1 \right] (\log x) &\text{if } v=w \, , \\
\displaystyle
\frac{1}{x_v x_w} \partial_v \partial_w \phi_1 (\log x) &\text{if } v\neq w \, .
\end{cases}
\]
Assume now that $x$ is a critical point of $\Phi_1$ on $\pos_1^V(z)$, i.e.\ that both sides of~\eqref{eq:derivativesPhi&phi} vanish for all $v\in \Gamma^{\mathsf{c}}$.
Then, we have
\[
\partial_v \partial_w \Phi_1 (x)
= \frac{1}{x_v x_w} \partial_v \partial_w \phi_1 (\log x)
\qquad\qquad
\text{for all } v,w\in \Gamma^{\mathsf{c}} \, .
\]
As the Hessian of $\phi_1$ on $\R^V(\log z)$ is positive definite everywhere by Lemma~\ref{lemma:convexityPhi_1}, it follows that the Hessian of $\Phi_1$ on $\pos_1^V(z)$ is positive definite at $x$.
\end{proof}

To compute the Hessian in the general case $d\geq 1$, we will use the following basic formulas (see e.g.~\cite{matrixCookbook12}) that hold for any $a,b\in\pos_d$, $1\leq i\leq j\leq d$, and $1\leq k\leq \ell\leq d$:
\begin{align}
\label{eq:inverseDerivative}
\frac{\partial (a^{-1})(k,\ell)}{\partial a(i,j)}
&= -\frac{1}{1+\delta(i,j)} \left[ a^{-1}(k,i) a^{-1}(\ell,j) + a^{-1}(k,j) a^{-1}(\ell,i) \right] \, , \\
\label{eq:inverseDerivative2}
\frac{\partial (a^{-1}ba^{-1})(k,\ell)}{\partial a(i,j)}
&= -\frac{a^{-1}(k,i) (a^{-1}b a^{-1})(\ell,j) + (a^{-1}b a^{-1})(k,i) a^{-1}(\ell,j) + \boxed{i \leftrightarrow j}}{1+\delta(i,j)} \\
\label{eq:traceDer1}
\frac{\partial }{\partial a(i,j)} \tr[ab^{-1}]
&= \frac{2}{1+\delta(i,j)} b^{-1}(i,j) \, , \\
\label{eq:traceDer2}
\frac{\partial }{\partial b(i,j)} \tr[ab^{-1}]
&= -\frac{2}{1+\delta(i,j)} \left[b^{-1} a b^{-1}\right](i,j) \, ,
\end{align}
where $\delta(i,j)$ is $1$ if $i=j$ and $0$ otherwise.
In~\eqref{eq:inverseDerivative2}, $\boxed{i \leftrightarrow j}$ denotes the preceding expression with the indices $i$ and $j$ swapped.
Notice that~\eqref{eq:inverseDerivative2} and~\eqref{eq:traceDer2} can be deduced from~\eqref{eq:inverseDerivative}.

Let $\Scal_d^+$ be the set of positive definite scalar matrices, i.e.\ positive multiples of $I_d$.

\begin{lemma}
\label{lemma:hessPhi}
Let $z\in (\Scal_d^+)^{\Gamma}$.
Then, the Hessian of $\Phi_d$ on $\pos_d^V(z)$ is positive definite at any critical point $x$ such that $x \in (\Scal_d^+)^{V}(z)$.
\end{lemma}

\begin{proof}
We will prove that, under the stated assumptions, the Hessian of $\Phi_d$ (for $d\geq 1$) can be expressed in terms of the Hessian of $\Phi_1$; the claim will then follow from Lemma~\ref{lemma:hessPhi_d=1}.

For ease of notation, given any $v\in V$ and $1\leq i\leq j\leq d$, we will denote by $\partial_{v;i,j}$ the partial derivative of a function of $x\in\pos_d^V$ with respect to the real variable $x_v(i,j)$.

It follows from the definition~\eqref{eq:energyFn_Phi} and from the formulas~\eqref{eq:traceDer1}-\eqref{eq:traceDer2} that
\[
\partial_{v;i,j} \Phi_d(x)
= \frac{2}{1+\delta(i,j)} \left( \sum_{\substack{w\in V\colon \\ v\to w}} x_w^{-1}(i,j) - \sum_{\substack{u\in V\colon \\ u\to v}} \left[x_v^{-1} x_u x_v^{-1}\right](i,j) \right)
\]
for $v\in \Gamma^{\mathsf{c}}$ and $1\leq i\leq j\leq d$.
The critical point equations of $\Phi_d$ on $\pos_d^V(z)$ are then
\begin{equation}
\label{eq:critPointsEqns}
x_v^{-1} \Bigg(\sum_{\substack{u\in V\colon \\ u\to v}} x_u\Bigg) x_v^{-1} = \sum_{\substack{w\in V \colon \\ v\to w}} x_w^{-1}
\qquad\qquad
\text{for all } v\in \Gamma^{\mathsf{c}} \, .
\end{equation}

We will now compute the second derivatives at any critical point $x=(x_v)_{v\in V} \in \pos_d^V(z)$.
Using~\eqref{eq:inverseDerivative2} and~\eqref{eq:critPointsEqns}, we have
\[
\begin{split}
& \partial_{v;k,\ell} \partial_{v;i,j} \Phi_d (x)
= \frac{2}{1+\delta(i,j)}
\frac{1}{1+\delta(k,\ell)} \Bigg( x_v^{-1}(i,k) \Bigg[ x_v^{-1} \Bigg( \sum_{\substack{u\in V\colon \\ u\to v}} x_u\Bigg) x_v^{-1} \Bigg](j,\ell) \\
&\qquad\qquad\qquad\qquad\qquad\qquad\qquad\qquad\quad + \Bigg[ x_v^{-1} \Bigg( \sum_{\substack{u\in V\colon \\ u\to v}} x_u\Bigg) x_v^{-1} \Bigg](i,k) x_v^{-1}(j,\ell) + \boxed{k \leftrightarrow \ell} \Bigg) \\
&= \frac{2}{1+\delta(i,j)}
\frac{1}{1+\delta(k,\ell)} \Bigg( x_v^{-1}(i,k) \sum_{\substack{u\in V \colon \\ v\to u}} x_u^{-1}(j,\ell) + \sum_{\substack{u\in V \colon \\ v\to u}} x_u^{-1}(i,k) x_v^{-1}(j,\ell) + \boxed{k \leftrightarrow \ell} \Bigg)
\end{split}
\]
for $v\in \Gamma^{\mathsf{c}}$, $1\leq i\leq j\leq d$, and $1\leq k\leq \ell\leq d$.
Recall now that the acyclic structure of the underlying graph guarantees that, if $v\to w$, then $w\not\to v$.
Therefore, for $v,w\in \Gamma^{\mathsf{c}}$ such that $v\to w$, $1\leq i\leq j\leq d$, and $1\leq k\leq \ell\leq d$, we have
\[
\partial_{w;k,\ell} \partial_{v;i,j} \Phi_d (x)
= - \frac{2}{1+\delta(i,j)}
\frac{1}{1+\delta(k,\ell)}
\left[ x_w^{-1}(i,k) x_w^{-1}(j,\ell) + x_w^{-1}(i,\ell) x_w^{-1}(j,k) \right] \, .
\]
On the other hand, the second derivative w.r.t.\ $x_v(i,j)$ and $x_w(k,\ell)$ vanishes for all $v,w\in \Gamma^{\mathsf{c}}$ such that $v\not\to w$ and $w\not\to v$.

According to the hypotheses of the theorem, we further assume from now on that there exists $\zeta = (\zeta_v)_{v\in V} \in \pos_1^{\Gamma}$ such that $z_v = \zeta_v I_d$ for all $v\in \Gamma$, and there exists $\xi = (\xi_v)_{v\in V} \in \pos_1^V$ such that $x_v = \xi_v I_d$ for all $v\in V$.
Using the identity
\[
\frac{\delta(i,k) \delta(j,\ell) + \delta(i,\ell) \delta(j,k)}{1+\delta(k,\ell)}
= \delta((i,j),(k,\ell))
\qquad
\text{for } 1\leq i\leq j\leq d \, , \,\, 1\leq k\leq \ell\leq d \, ,
\]
we see that the second derivatives at $x$ factorise as
\[
\partial_{w;k,\ell} \partial_{v;i,j} \Phi_d (x)
= f((i,j),(k,\ell)) \, g_{\xi}(v,w) \, ,
\qquad\text{with}\qquad
f((i,j),(k,\ell)) = \frac{2 \delta((i,j),(k,\ell))}{1+\delta(i,j)}  \, .
\]
Here, $g_{\xi}(v,w)$ is an explicit function of $\xi$, $v$ and $w$; we stress that it is the same function for all $d\geq 1$.
It follows from~\eqref{eq:critPointsEqns} that, since $x=(\xi_v I_d)_{v\in V}$ is a critical point of $\Phi_d$ on $\pos_d^V(z)$, $\xi$ is a critical point of $\Phi_1$ on $\pos_1^V(\zeta)$.
Therefore, the matrix
\[
g_{\xi}(v,w) = \partial_v \partial_w \Phi_1(\xi)
\]
(with `row index' $v$ and `column index' $w$), which is the Hessian matrix of $\Phi_1$ on $\pos_1^V(\zeta)$ at $\xi$, is positive definite by Lemma~\ref{lemma:hessPhi_d=1}.
On the other hand, the matrix $f((i,j),(k,\ell))$ (with `row index' $(i,j)$ and `column index' $(k,\ell)$) is clearly positive definite as a diagonal matrix with positive diagonal entries.
Therefore, the Hessian of $\Phi_d$ on $\pos_d^V(z)$ at $x$ is positive definite, as it can be written as a Kronecker product of two positive definite matrices.
\end{proof}

As any minimiser is a critical point, the main result of this section follows immediately from Corollary~\ref{coro:minimiser_phi} and Lemma~\ref{lemma:hessPhi}.

\begin{theorem}
\label{thm:minimiser&Hessian}
Let $z\in (\Scal_d^+)^{\Gamma}$.
Then, the function $\Phi_d$ on $\pos_d^{V}(z)$ has a unique (global) minimiser $m$, at which the Hessian is positive definite.
Moreover, we have $m\in (\Scal_d^+)^{V}(z)$.
\end{theorem}

\subsection{Laplace approximation}
\label{subsec:LaplaceAsymptotics}

We will now use Theorem~\ref{thm:minimiser&Hessian} to study the asymptotic behaviour of integrals of exponentials of $\Phi_d$, via Laplace's approximation method.
Recall the definition~\eqref{eq:measureMu} of the measure $\mu$ on $\pos_d$.
\begin{proposition}
\label{prop:saddlePoint}
Let $z\in (\Scal_d^+)^{\Gamma}$ and let $m$ be the unique global minimiser of $\Phi_d$ on $\pos_d^{V}(z)$ (see Theorem~\ref{thm:minimiser&Hessian}).
Let $g\colon \pos_d^V(z) \to\C$ be a continuous function in a neighbourhood of $m$, with $g(m)\neq 0$, and such that
\[
\int_{\pos_d^{V}(z)}
\left(\prod_{v\in \Gamma^{\mathsf{c}}} \mu(\diff x_v)\right)
\abs{g(x)} \e^{-k \Phi_d(x)} < \infty
\qquad\quad
\text{for some } k>0 \, .
\]
Then
\begin{equation}
\label{eq:saddlePoint}
\begin{split}
&\int_{\pos_d^V(z)} 
\left(\prod_{v\in \Gamma^{\mathsf{c}}} \mu(\diff x_v)\right)
g(x)
\e^{-k \Phi_d(x)} \\
&\widesim[2.5]{k\to\infty}
\frac{g(m)}{\sqrt{\abs{\mathcal{H}(m)}}}
\left(\prod_{v\in \Gamma^{\mathsf{c}}} \abs{m_v}^{-\frac{d(d+1)}{2}}\right)
\left( \frac{2\pi}{k}\right)^{\lvert\Gamma^{\mathsf{c}}\rvert \frac{d(d+1)}{4}} \e^{-k \Phi_d(m)}
\, ,
\end{split}
\end{equation}
where $\abs{\mathcal{H}(m)}>0$ is the Hessian determinant of $\Phi_d$ at $m$ and $\big\lvert\Gamma^{\mathsf{c}}\big\rvert$ is the number of vertices in $\Gamma^{\mathsf{c}}$.
\end{proposition}

We start by stating the Laplace approximation integral formula in the multivariate context, which can be found e.g.\ in~\cite{evansSwartz00}.
\begin{theorem}[{\cite[Theorem~4.14]{evansSwartz00}}]
\label{thm:laplace}
Let $A$ be an open subset of the $p$-dimensional space $\R^p$.
Let $h\colon A \to\C$ and $\rho\colon A\to\R$ be functions such that
\begin{enumerate}
\item \label{it:integrability}
$\int_{A} \abs{h(x)} \e^{-k \rho(x)} \diff x < \infty$ for some $k>0$.
\item \label{it:minimiser}
$\rho$ has a global minimiser $x_0\in A$ such that, for every $\epsilon>0$,
\begin{equation}
\label{eq:minimiser_strong}
\inf\{\rho(x)-\rho(x_0) \colon x\in A, \, \abs{x-x_0} \geq \epsilon\} >0 \, .
\end{equation}
\item \label{it:continuity}
$h$ is continuous in a neighbourhood of $x_0$ and $h(x_0) \neq 0$.
\item \label{it:Hessian}
$\rho$ is twice continuously differentiable on $A$ and its Hessian matrix $\mathcal{H}(x_0)$ at $x_0$ is positive definite (in particular, its determinant $\abs{\mathcal{H}(x_0)}$ is positive).
\end{enumerate}
Then,
\begin{equation}
\label{eq:laplace}
\int_{A} h(x) \e^{-k \rho(x)} \diff x
\widesim[2.5]{k\to\infty}
\frac{h(x_0)}{\sqrt{\abs{\mathcal{H}(x_0)}}}
\left( \frac{2\pi}{k} \right)^{\frac{p}{2}}
\e^{- k \rho(x_0)} \, .
\end{equation}
\end{theorem}

\begin{proof}[Proof of Prop.~\ref{prop:saddlePoint}]
We will apply Theorem~\ref{thm:laplace} with
\[
A=\pos_d^V(z) \, , \quad\qquad
h(x)=g(x) \prod_{v\in \Gamma^{\mathsf{c}}} \abs{x_v}^{-\frac{d(d+1)}{2}} \, , \quad\qquad
\rho=\Phi_d \, , \quad\qquad
x_0 = m \, .
\]
The set $\pos_d^V(z)$ can be clearly viewed as an open subset of $\R^p$, where $p=\big\lvert\Gamma^{\mathsf{c}}\big\rvert d(d+1)/2$ is the number of `free' real variables in $A$ and $d$ is the dimension of each matrix in the array.
The extra product in the definition of $h$ is the density of the measure $\prod_{v\in \Gamma^{\mathsf{c}}} \mu(\diff x_v)$ with respect to the Lebesgue measure on $\pos_d^V(z)$.

Hypothesis~\ref{it:minimiser} of Theorem~\ref{thm:laplace} is satisfied due to Theorem~\ref{thm:minimiser&Hessian} and Prop.~\ref{prop:phiAtInfinity}.
Hypotheses~\ref{it:integrability} and~\ref{it:continuity} are matched by the assumptions of Prop.~\ref{prop:saddlePoint}.
Finally, hypothesis~\ref{it:Hessian} also holds because of Theorem~\ref{thm:minimiser&Hessian}.
The asymptotic formula~\eqref{eq:saddlePoint} then follows from~\eqref{eq:laplace}.
\end{proof}

\subsection*{Acknowledgements}
The authors thank the anonymous referees for their helpful comments and suggestions, which have led to a much improved version of the paper.

\appendix

\section{A proof of the Cauchy-Littlewood identity}
\label{app:cauchyLittlewood}

In this appendix we include a proof of the classical Cauchy-Littlewood identity for Schur functions that is based on a version of the Pieri rule.
The proof of the Whittaker integral identity~\eqref{thm:stade_matrix} is based, \emph{mutatis mutandis}, on the same line of reasoning.

For any two integer partitions $\mu=(\mu_1\geq \mu_2\geq \cdots)$ and $\lambda=(\lambda_1\geq \lambda_2\geq \cdots)$, we write $\mu \prec \lambda$ if $\lambda_i\geq \mu_i\geq \lambda_{i+1}$ for all $i\geq 1$.
The \emph{size} of $\lambda$ is $\abs{\lambda}:=\lambda_1+\lambda_2+\dots$, while its \emph{length} is the smallest $i\geq 0$ such that $\lambda_{i+1}=0$.
We will adopt the following recursive definition of Schur polynomials: for a partition $\lambda$ of length $\leq n$, we set
\begin{equation}
\label{eq:Schur}
s_{\lambda}(x_1,\dots,x_n)
:= 
\begin{cases}
x_1^{\abs{\lambda}} &\text{if } n=1 \, , \\
\sum_{\mu \prec \lambda} x_n^{\abs{\lambda} - \abs{\mu}} s_{\mu}(x_1,\dots,x_{n-1}) &\text{if } n>1 \, .
\end{cases}
\end{equation}
For the sake of convenience, we also set $s_{\lambda}(x_1,\dots,x_n):=0$ if the length of $\lambda$ exceeds $n$.
This definition is easily seen to be equivalent to the classical combinatorial definition of Schur polynomials as generating functions of semistandard Young tableaux.

We will use the following version of the Pieri rule:
\begin{equation}
\label{eq:PieriRule}
\sum_{\lambda\colon \mu\prec \lambda} s_{\lambda}(x_1,\dots,x_n)
= \left(\prod_{i=1}^n \frac{1}{1-x_i}\right)
s_{\mu}(x_1,\dots,x_n) \, .
\end{equation}
The latter can be deduced from the usual Pieri rule (see e.g.~\cite[I-(5.16)]{macdonald79})
\[
\sum_{\substack{\lambda\colon \mu\prec \lambda, \\ \abs{\lambda}-\abs{\mu}=r}} s_{\lambda}(x_1,\dots,x_n)
= \left(\sum_{\substack{k_1,\dots,k_n\geq 0\colon \\ k_1+\dots+k_n=r}} x_1^{k_1} \cdots x_n^{k_n}\right)
s_{\mu}(x_1,\dots,x_n)
\]
by summing over all $r\geq 0$.
Notice that~\eqref{eq:PieriRule} can be read as an eigenfunction equation for the operator defined through the kernel $I(\mu;\lambda) := \1_{\mu\prec \lambda}$, with the Schur function $s_{\mu}$ (viewed as a function of the partition $\mu$) as an eigenfunction.

\begin{theorem}[Cauchy-Littlewood identity]
For any $n,N\geq 1$, we have
\begin{equation}
\label{eq:CauchyLittlewood}
\sum_{\lambda} s_{\lambda}(x_1,\dots,x_n) s_{\lambda}(y_1,\dots,y_N)
= \prod_{\ell=1}^n \prod_{i=1}^N \frac{1}{1-x_\ell y_i} \, ,
\end{equation}
where the sum is over all integer partitions $\lambda$.
\end{theorem}
\begin{proof}
Note first that we can restrict the sum on the left-hand side of~\eqref{eq:CauchyLittlewood} to the partitions $\lambda$ with length $\leq \min(n,N)$.
When $N=n=1$, the identity reduces to a geometric sum:
\begin{equation}
\sum_{\lambda_1 \geq 0} x_1^{\lambda_1} y_1^{\lambda_1}
= \frac{1}{1-x_1 y_1} \, .
\end{equation}
We can then proceed by induction on $n+N$.

Let $n+N> 2$ and assume, without loss of generality, that $n>1$.
Using the definition~\eqref{eq:Schur}, the fact that $s_{\lambda}$ is a homogeneous polynomial of degree $\abs{\lambda}$, and identity~\eqref{eq:PieriRule}, we obtain
\[
\begin{split}
\sum_{\lambda} s_{\lambda}(x_1,\dots,x_n) & s_{\lambda}(y_1,\dots,y_N)
= \sum_{\lambda} \left(\sum_{\mu\prec \lambda}
x_n^{\abs{\lambda}-\abs{\mu}}
s_{\mu}(x_1,\dots,x_{n-1}) \right)
s_{\lambda}(y_1,\dots,y_N) \\
&= \sum_{\mu} \left(\sum_{\lambda\colon \mu\prec \lambda}
s_{\lambda}(x_n y_1,\dots,x_n y_N) \right)
x_n^{-\abs{\mu}}
s_{\mu}(x_1,\dots,x_{n-1}) \\
&= \sum_{\mu} \left(\prod_{i=1}^N \frac{1}{1-x_n y_i}\right)
s_{\mu}(x_n y_1,\dots,x_n y_N)
x_n^{-\abs{\mu}}
s_{\mu}(x_1,\dots,x_{n-1}) \\
&= \left(\prod_{i=1}^N \frac{1}{1-x_n y_i}\right)
\sum_{\mu}
s_{\mu}(x_1,\dots,x_{n-1})
s_{\mu}(y_1,\dots,y_N) \, .
\end{split}
\]
The claim then follows from the induction hypothesis applied to the latter sum.
\end{proof}

\section{Markov functions and intertwinings}
\label{app:markovFunctions}

In this appendix we review the theory of Markov functions, in the case of inhomogeneous discrete-time Markov processes, which we are concerned with in the present article.

Let $(S, \mathcal{S})$ and $(T, \mathcal{T})$ be measurable spaces and $\phi\colon S\to T$ be a measurable function.
Let $X=(X(n))_{n\geq 0}$ be a time-inhomogeneous Markov process with state space $S$, time-$n$ transition kernel $\Pi_n$ and any initial distribution on $X(0)$.
Defining $Z(n):=\phi(X(n))$ for all $n\geq 0$, we will give conditions under which the transformed process $Z=(Z(n))_{n\geq 0}$ with state space $T$ is still Markov in its own filtration.
The well-known Dynkin criterion~\cite{dynkin61} ensures that $Z$ satisfies the Markov property for any possible initial distribution on $X$.
On the other hand, the theory of Markov functions (developed at various levels of generality in~\cite{kemenySnell76, rogersPitman81, kelly82, kurtz98}) provides a more subtle criterion, in which the Markov property of $Z$ is guaranteed only under certain specific initial states of $X$.

Let $\bdd{\mathcal{S}}$ be the space of bounded measurable functions from $(S, \mathcal{S})$ to $\R$.

\begin{theorem}
\label{thm:MarkovFns}
Let $X=(X(n))_{n\geq 0}$ be a time-inhomogeneous Markov process on $S$ with time-$n$ transition kernel $\Pi_n$.
Let $\phi\colon S\to T$ be a measurable function.
Let $Z=(Z(n))_{n\geq 0}$, where $Z(n)=\phi(X(n))$ for all $n\geq 0$.
Assume that $\mathcal{T}$ contains all the singleton sets $\{z\}$.
Let $\Sigma$ be a Markov kernel from $T$ to $S$ and, for all $n\geq 1$, let $P_n$ be a Markov kernel from $T$ to $T$.
Suppose:
\begin{enumerate}
\item
\label{it:MarkovFns_initial}
$\Sigma\left(z; \phi^{-1}\{z\}\right)=1$ for every $z\in T$;
\item
\label{it:MarkovFns_intertw}
$\Sigma \Pi_n = P_n \Sigma$ for all $n\geq 1$.
\end{enumerate}
Assume that, for an arbitrary $z\in T$, the initial state $X(0)$ of $X$ is distributed according to the measure $\Sigma(z;\cdot)$.
Then, $Z$ is a time-inhomogeneous Markov process (in its own filtration), with initial state $z$ and time-$n$ transition kernel $P_n$.
Moreover, for all $f\in\bdd{\mathcal{S}}$ and $n\geq 0$, we have
\begin{equation}
\label{eq:MarkovFns_condExp}
\E[f(X(n)) \mid Z(0),\dots,Z(n-1),Z(n)]
=\Sigma f(Z(n)) \qquad \text{a.s.}
\end{equation}
\end{theorem}

\begin{proof}
This proof is an inhomogeneous discrete-time version of the argument given for continuous-time Markov processes in~\cite{rogersPitman81}.
Note that~\ref{it:MarkovFns_initial} implies
\[
\int_{S} \Sigma(z; \diff x) g(\phi(x)) f(x)
= g(z) \int_{S} \Sigma(z; \diff x) f(x)
\]
for all $g\in\bdd{\mathcal{T}}$, $f\in\bdd{\mathcal{S}}$, and $z\in T$.
Letting $\Phi\colon \bdd{\mathcal{T}} \to \bdd{\mathcal{S}}$ be the Markov operator defined by $\Phi g := g\circ\phi$ for $g\in\bdd{\mathcal{T}}$, we may rewrite the above identity as
\begin{align}
\label{eq:concatenation}
\Sigma(\Phi g) f = g \Sigma f \, .
\end{align}
Here, as in the following, the operations should be read from right to left, prioritising the brackets
(for example, on the left-hand side of~\eqref{eq:concatenation}, one first multiplies the two functions $f$ and $\Phi g$ and then applies the operator $\Sigma$ to the resulting function).
Applying $P_i$ to both sides of \eqref{eq:concatenation} and using hypothesis~\ref{it:MarkovFns_intertw}, we have
\begin{align}
\label{eq:concatenation2}
\Sigma \Pi_i (\Phi g) f = P_i g \Sigma f \qquad\quad \text{for all } i\geq 1 \, .
\end{align}
Consider now test functions $g_{0},\dots,g_{n}$ in $\bdd{\mathcal{T}}$ and $f\in\bdd{\mathcal{S}}$.
Using \eqref{eq:concatenation} and \eqref{eq:concatenation2} several times, we obtain
\begin{equation}
\label{eq:concatenation3}
\begin{split}
& \Sigma(\Phi g_{0}) \Pi_1(\Phi g_{1}) \Pi_2(\Phi g_{2})\cdots \Pi_n(\Phi g_{n}) f
=g_{0} \Sigma \Pi_1(\Phi g_{1}) \Pi_2(\Phi g_{2})\cdots \Pi_n(\Phi g_{n}) f \\
=\, &g_{0} P_1 g_{1} \Sigma \Pi_2 (\Phi g_{2})\cdots \Pi_n(\Phi g_{n}) f
=\dots = g_{0} P_1 g_{1} P_2 g_{2}\cdots P_n g_{n} \Sigma f \, .
\end{split}
\end{equation}

Fix now an arbitrary $z\in T$ and assume that $X(0)$ is distributed according to $\Sigma(z;\cdot)$.
Then, \eqref{eq:concatenation3} yields
\begin{align*}
\E\left[g_{0}(Z(0)) \, g_1(Z(1)) \cdots g_{n}(Z(n)) \, f(X(n))\right]
= g_{0} P_1 g_{1} \cdots P_n g_{n} \Sigma f(z) \, .
\end{align*}
Taking $f\equiv 1$, we deduce that $Z$ is a Markov process started at $z$ with time-$n$ transition kernel $P_n$.
For general $f$, the right-hand side of the equation above agrees with
\begin{align*}
\E\left[g_{0}(Z(0)) \, g_1(Z(1)) \cdots g_{n}(Z(n)) \, \Sigma f(Z(n))\right] \, .
\end{align*}
This, by definition of conditional expectation, proves~\eqref{eq:MarkovFns_condExp}.
\end{proof}

\begin{remark}
\label{rem:MarkovFns}
Taking $f\equiv 1$ in~\eqref{eq:concatenation}, we see that $\Sigma \Phi$ is the identity on $\bdd{\mathcal{T}}$.
Combining this with hypothesis~\ref{it:MarkovFns_intertw} of Theorem~\ref{thm:MarkovFns}, it is immediate to deduce that every kernel $P_n$ is uniquely determined by the relation $P_n =\Sigma\Pi_n\Phi$.
\end{remark}

\section{A convergence lemma}
\label{app:convergenceLemma}

Here we state a useful convergence lemma.
For completeness we also include its proof, which follows from the properties of weak convergence and standard estimates.

\begin{lemma}
\label{lemma:convergenceOfIntegrals}
Let $S$ be a locally compact metric space equipped with its Borel $\sigma$-algebra.
Let $(\mu_k)_{k>0}$ be a collection of probability measures on $S$ that converges weakly, as $k\to\infty$, to a Dirac measure $\delta_s$ for some $s\in S$.
Let $(f_k)_{k>0}$ be a uniformly bounded collection of continuous functions $S\to\R$ such that $f_k \xrightarrow{k\to\infty} f_{\infty}$ uniformly on any compact subset of $S$.
Then
\begin{equation}
\label{eq:convergencelemma}
\lim_{k\to\infty} \int_{S} \mu_k(\diff x) f_k(x)
= f_{\infty}(s) \, .
\end{equation}
\end{lemma}
\begin{proof}
Fix $\epsilon>0$.
For any Borel set $U\subset S$, we may write
\[
\begin{split}
&\abs{\int_{S} \mu_k(\diff x) f_k(x) - f_{\infty}(s)}
\leq \abs{\int_{U} \mu_k(\diff x) [f_k(x) - f_{\infty}(s)] }
+ \abs{\int_{U^{\mathsf{c}}} \mu_k(\diff x) [f_k(x) - f_{\infty}(s)] }
\\
&\leq \int_{\overline{U}} \mu_k(\diff x) \abs{f_k(x) - f_{\infty}(x)} 
+ \int_U \mu_k(\diff x) \abs{f_{\infty}(x) - f_{\infty}(s)}
+ \left( \sup_{x\in S} \abs{f_k(x)} + \abs{f_{\infty}(s)} \right) \mu_k(U^{\mathsf{c}}) \, ,
\end{split}
\]
where $\overline{U}$ and $U^{\mathsf{c}}$ are the closure and the complement of $U$, respectively.
Since $f_{\infty}$ is continuous (as a uniform limit of continuous functions) and $S$ is a locally compact metric space, we can choose $U$ to be a precompact open neighbourhood of $s$ such that $\abs{f_{\infty}(x) - f_{\infty}(s)} \leq \epsilon$ for all $x\in U$.
Moreover, as $\overline{U}$ is compact, for $k$ large enough we have $\abs{f_k(x) - f_{\infty}(x)}\leq \epsilon$ for all $x\in\overline{U}$.
Finally, the Portmanteau theorem (see~\cite[\S~2]{billingsley99}) yields
\[
\limsup_{k\to\infty} \mu_k(U^{\mathsf{c}})
= \delta_s(U^{\mathsf{c}})
= 0 \, ,
\]
since $\mu_k$ converges weakly to $\delta_s$, $U^{\mathsf{c}}$ is closed, and $s\notin U^{\mathsf{c}}$; therefore, for large enough $k$ we also have $\mu_k(U^{\mathsf{c}}) \leq \epsilon$.
By the hypothesis of uniform boundedness, there exists $M>0$ such that $\abs{f_k(x)}\leq M$ for all $x\in S$ and $k>0$.
Hence, we have
\[
\abs{\int_{S} \mu_k(\diff x) f_k(x)  - f_{\infty}(s)}
\leq \mu_k(\overline{U}) \epsilon  +  \mu_k(U) \epsilon + \left[ M + \abs{f_{\infty}(s)} \right] \epsilon
\leq \left[ 2+M+\abs{f_{\infty}(s)} \right] \epsilon
\]
for $k$ large enough.
As $\epsilon>0$ is arbitrary, the claim follows.
\end{proof}

\printbibliography

\end{document}